\documentclass[11pt]{amsart}

\allowdisplaybreaks

\usepackage[colorlinks = true,
            linkcolor = red,
            urlcolor  = magenta,
            citecolor = black,
            anchorcolor = blue]{hyperref}
\usepackage{color}
\usepackage{enumitem}
\usepackage{comment}
\usepackage[shortalphabetic,initials,nobysame]{amsrefs}
\usepackage{upgreek}
\usepackage{tikz}
\usepackage[applemac]{inputenc} 
\usetikzlibrary{arrows}
\usepackage{xcolor}
\usepackage[all]{xy}
\usepackage{graphicx}
\usepackage{wrapfig}
\usepackage{yfonts}
\usepackage{graphicx}
\usepackage{amssymb}
\usepackage{epstopdf}
\usepackage{mathrsfs}
\usepackage[mathcal]{eucal}
\usepackage{bm}
\usepackage{subfigure}

\usepackage{array}
\usepackage{tabularx}

\newcolumntype{Y}{>{\centering\arraybackslash}X}

\renewcommand{\eprint}[1]{\href{https://arxiv.org/abs/#1}{#1}}

\DeclareMathOperator{\Op}{Op}
\DeclareMathOperator{\Fun}{Fun}

\BibSpec{article}{%
    +{}  {\PrintAuthors}                {author}
    +{,} { \textit}                     {title}
     +{,} {}                            {note}
    +{.} { }                            {part}
    +{:} { \textit}                     {subtitle}
    +{,} { \PrintContributions}         {contribution}
    +{.} { \PrintPartials}              {partial}
    +{,} { }                            {journal}
    +{}  { \textbf}                     {volume}
    +{}  { \PrintDate}                {date}
    +{,} { \issuetext}                  {number}
    +{,} { \eprintpages}                {pages}
    +{,} { }                            {status}
    +{,} { \DOI}                   {doi}
    +{,} { \eprint}        {eprint}
      +{,} {\publisher}                {publisher}
    +{,} { \address}                   {address}
    +{}  { \parenthesize}               {language}
    +{}  { \PrintTranslation}           {translation}
    +{;} { \PrintReprint}               {reprint}
    +{.} {}                             {transition}
    +{}  {\SentenceSpace \PrintReviews} {review}
}

\newtheorem{Thm}{Theorem}[section]
\newtheorem{Lem}[Thm]{Lemma}
\newtheorem{Prop}[Thm]{Proposition}
\newtheorem{Cor}[Thm]{Corollary}
\newtheorem{Con}[Thm]{Conjecture}

\newtheorem{Nota}[Thm]{Notation}

\theoremstyle{definition}
\newtheorem{Def}[Thm]{Definition}
\theoremstyle{remark}
\newtheorem{Rem}[Thm]{Remark}

\newtheoremstyle{named}{}{}{\itshape}{}{\bfseries}{.}{.5em}{#1 #3}
\theoremstyle{named}

\def\Q{\mathbb{Q}}

\def\C{\mathbb{C}}
\def\Z{\mathbb{Z}}

\def\P{\mathbb{P}}

\def\g{\mathfrak{g}}
\def\Frenkel:2013uda{\mathfrak{h}}

\def\sl{\mathfrak{sl}}

\def\cA{\mathcal{A}}

\def\a{\alpha}
\def\b{\beta}

\def\D{\Delta}

\def\e{\epsilon}

\def\l{\lambda}
\def\L{\Lambda}

\def\bo{\textbf{o}}

\def\=>{\Longrightarrow}

\def\to{\longrightarrow}

\def\o+{\oplus}
\def\bo+{\bigoplus}
\def\x{\times}
\def\<{\langle}
\def\>{\rangle}
\def\({\left(}
\def\){\right)}

\def\^{\wedge}
\def\+{\dagger}

\def\dd[#1,#2]{\frac{d#1}{d#2}}
\def\del[#1,#2]{\frac{\partial #1}{\partial #2}}
\def\over[#1]{\overline{#1}}
\def\vec[#1]{\overrightarrow{#1}}

\makeatletter
\def\mr@ignsp#1 {\ifx\:#1\@empty\else #1\expandafter\mr@ignsp\fi}%
\newcommand{\multiref}[1]{\begingroup
\xdef\mr@no@sparg{\expandafter\mr@ignsp#1 \: }%
\def\mr@comma{}%
\@for\mr@refs:=\mr@no@sparg\do{\mr@comma\def\mr@comma{,}\ref{\mr@refs}}%
\endgroup}
\makeatother

\newcommand{\hypref}[2]{\ifx\href\asklFrenkel:2013udaas #2\else\href{#1}{#2}\fi}
\newcommand{\Secref}[1]{Section~\multiref{#1}}

\tikzset{->-/.style={decoration={
  markings,
  mark=at position .5 with {\arrow{latex}}},postaction={decorate}}}
\tikzset{
    >=latex
    }

\newcommand{\nc}{\newcommand}
\nc{\on}{\operatorname}
\nc{\la}{\lambda}
\nc{\wh}{\widehat}
\nc{\ghat}{\wh\g}
\nc{\mb}{\mathbf}

\begin{document}
\title{$q$-Opers and Quantum/Classical Duality Beyond Type A}

\author[P. Koroteev]{Peter Koroteev}
\address{
\newline
Beijing Institute for Mathematical Sciences and Applications
\newline
Beijing, China
\newline
\href{mailto:peter.koroteev@gmail.com}{peter.koroteev@gmail.com}
}

\author[M. Shim]{Myungbo Shim}
\address{
\newline
          Yau Mathematical Sciences Center
          Tsinghua University
          \newline
          Beijing, China
        \newline
        \href{mailto: mbshim@tsinghua.edu.cn}{mbshim@tsinghua.edu.cn}
          }

\author[R. Singh]{Rahul Singh}
\address{
\newline
          Yau Mathematical Sciences Center
          Tsinghua University
        \newline
          Beijing, China 
          \newline
          \href{mailto: 95rahul32@gmail.com}{95rahul32@gmail.com}
          }

\begin{abstract}
    Using the language of $q$-opers, we propose an algebro-geometric description of the duality between Bethe ansatz equations for quantum XXZ spin chains and many-body trigonometric Ruijsenaars-Schneider/Macdonald systems for classical groups. Upon this duality, the energy level sets of the classical $B/C/D$-type Hamiltonians are found to be in bijection with the set of solutions of the XXZ Bethe ansatz equations of type $A$, albeit with {\it open} boundary conditions. This dictionary generalizes previous results in which a $GL(N)$ XXZ spin chain with twisted periodic boundary conditions was dual to an $N$-body Ruijsenaars-Schneider system. Our construction implements a $\mathbb{Z}/2\mathbb{Z}$ folding both at the level of the $(GL(N),q)$-oper data as well as at the level of the $N$-body trigonometric Ruijsenaars-Schneider model. The nonreduced $BC$-type systems appear in our analysis as well.
\end{abstract}

\date{\today}

\numberwithin{equation}{section}

\maketitle

\setcounter{tocdepth}{1}
\tableofcontents

\section{Introduction}

\subsection{Geometric Realization of Integrability}
A key idea underlying the geometric approach to integrable systems is that the spectral problem of an integrable model can be realized as a moduli problem. In this formulation, both the parameters of the model and the eigenvalues of its commuting Hamiltonians are encoded in suitable algebro-geometric data. The resulting moduli space often possesses rich geometric structure and, in particular, reveals connections to other integrable models.

For instance, for quantum spin models such as the XXX and XXZ spin chains, the Bethe ansatz equations determining their spectra can be realized as incidence or transversality conditions between Lagrangian subvarieties of an appropriate moduli space. The spectral problem can also be effectively formulated in terms of spaces of opers and their $q$-difference analogues. Opers encode the Bethe/Baxter equations, as well as the $QQ$-relations, in the geometry of a connection equipped with compatible Borel reductions, thereby replacing a collection of functional identities with a single intrinsic geometric object.
This perspective, and its relation to representation theory and classical many-body systems, is reviewed in \cite{koroteev2023quantumgeometryintegrabilityopers}.

For Nakajima quiver varieties \cite{nakajima1994}, this philosophy places the Bethe algebra, equivariant quantum $K$-theory, and multiplicative classical integrable systems in a single framework. The $q$-oper/Bethe correspondence identifies the
parameters and solutions of a quantum spin chain with functions on a
moduli space of $q$-connections, whereas the {\it quantum/classical duality} presents
the same algebra through spectral level sets, or intersections of
Lagrangian subvarieties, in a trigonometric Ruijsenaars--Schneider (tRS) phase
space. Thus, the geometric formulation does more than merely match individual
eigenvalues: it explains why the two systems provide different coordinate
descriptions of the same algebraic variety. For an introduction to this
circle of ideas, with particular emphasis on quantum $K$-theory of
Nakajima varieties and the $q$-oper/Bethe correspondence, see Zeitlin's
lecture notes \cite{zeitlin2024geometricrealizationsbetheansatz}. From
this viewpoint, imposing an involution is a geometric operation on the
underlying moduli data, and it is therefore essential to specify in which
of the two dual descriptions that operation is performed.

The quantum/classical duality at the XXZ/tRS level can be formulated in two different frames, which
we call \emph{electric} and \emph{magnetic}.
This distinction is particularly transparent in the tRS realization
of the quantum equivariant $K$-theory of type-$A$ quiver varieties
developed in \cite{Koroteev:2023aa}. In either frame, the relevant
algebra is described as the algebra of functions on an intersection of
two Lagrangian subvarieties inside a tRS phase space. What changes between
the two descriptions is which collection of parameters is regarded as
the tRS positions and which collection prescribes the level sets of the
commuting Hamiltonians.

Let $\{a_i\}$ denote the equivariant parameters associated with the
framing of the quiver variety or, equivalently, the parameters of the regular
singularities of the corresponding $q$-oper, and let $\{\xi_i\}$ denote the
K\"ahler parameters, which appear as the components of the oper twist.
Schematically, the two descriptions take the form
\begin{align}
\text{electric frame:}\qquad
&\text{tRS positions}=\{a_i\},
&\text{Hamiltonian levels}&=e_k(\{\xi_i\}),\notag
\\
\text{magnetic frame:}\qquad
&\text{tRS positions}=\{\xi_i\},
&\text{Hamiltonian levels}&=e_k(\{a_i\}).\notag
\end{align}
For the ordinary type-$A$ correspondence, both $a_i$ and $\xi_i$ are
constant parameters. The two frames are related by the
electric--magnetic transformation of the Calogero--Moser space, which
interchanges the two matrices in its rank-one commutator presentation
and sends $q$ to $q^{-1}$. In the geometric interpretation of
\cite{Koroteev:2023aa}, this exchange implements the interchange of
equivariant and K\"ahler parameters under three-dimensional mirror
symmetry. Thus, before imposing reflection invariance, one could in
principle try to fold either collection of tRS coordinates.

\subsection{Folding}
The reflection-invariant setup is genuinely asymmetric. In
\cite{Koroteev:2025tka}, we showed that reflection invariance of the
$q$-oper under a $\mathbb{Z}/2\mathbb{Z}$ involution of the base, of the form $z\mapsto c/z$,
requires the constant twist parameters to be replaced by rational
functions of the base coordinate. In the notation used below, the
diagonal twist data are described by functions $\xi_i(z)$ and
$\widetilde\xi_{i+1}(z)$ satisfying a relation of the form
\begin{equation}\label{eq:reflection-twists-intro}
 \widetilde\xi_{i+1}(z)
 =-\xi_i\left(\frac{d_i}{z}\right),
\end{equation}
where $d_i$ are some constants.
For the reflection-invariant Miura--Pl\"ucker opers introduced in the
present paper, these functions are additionally related by the gluing
conditions between successive Miura lines. In the folded systems, the
twist functions are tied to the $Q$-functions, and hence to the Bethe ansatz equations
themselves.

This observation rules out folding in the magnetic frame. In that
frame the twist parameters would have to serve as tRS coordinates.
After reflection, however, they are no longer a finite collection of
constant points of the multiplicative torus: they are rational
functions on the base curve, constrained by reflection and gluing and,
in general, depending on the Bethe solution. Consequently, a magnetic
fold would not define a fixed involutive subvariety of the ordinary
finite-dimensional tRS phase space.

\vskip.1in
We therefore perform folding in the electric frame in which the
equivariant parameters $a_i$ of the quiver remain constant and are
identified with the tRS coordinates. They can be paired by the usual
multiplicative involution $a_i\leftrightarrow a_i^{-1}$, together with
the corresponding constraints on the tRS momenta, which in turn can be computed as derivatives of the generating functions $Y(s,a,q)$ (Yang--Yang functions in the physics literature) on a Lagrangian subvariety $p_i=\exp a_i\frac{\partial Y}{\partial a_i}$.

On the $q$-oper side, where the $Q$-functions serve as coordinates of the local section
$s(z)=\begin{pmatrix}
    Q^+(z)\\ Q^-(z)
\end{pmatrix}$,
the same folding operation pairs their roots together.
A convenient parameterization of $Q^\pm(z)$ is therefore in terms of symmetric
Laurent polynomials of the form
\begin{equation}\label{eq:QfunLaurent}
    Q^+(z)=\prod_i(z-s_i)\left(\frac{b_i}{z}-s_i\right)
\end{equation}
for some constants $b_i$ which will be specified later in the text. The rational twist functions and the open-chain boundary factors are then the direct consequences of this electric-frame folding and of the Miura--Pl\"ucker gluing conditions.

Since $\{(a_i,p_i)\}$ form a canonical pair, the $\mathbb{Z}/2\mathbb{Z}$ involution must be compatible with the canonical structure. It will be shown in Sections~\ref{sec:CD} and \ref{sec:BCB} that the structure of the $Q$-function \eqref{eq:QfunLaurent} is \textit{enforced} by the above canonical presentation of the tRS momenta in order to keep the framework consistent.

\subsection{The Main Result}
For each root system of type $R$, where $R$ is $B$, $C$, $D$, or $BC$, we establish a connection between four sets that are summarized in the theorem below. More details will be provided in Theorems \ref{thm:CD} and \ref{thm:BCB}, in which the respective bijections can be formulated as isomorphisms of algebras.

\begin{Thm}
There is a pairwise bijection between the following sets:
    \begin{enumerate}[label=\textnormal{(\roman*)}]
        \item the set of isomorphism classes of nondegenerate generalized $Z$-twisted reflection-invariant Miura--Pl\"ucker $(GL(n),q)$-opers with twists and regular singularities given by the tuple of polynomials $\bm{\Lambda}(z)^{(R)}$ with gluing data $\bm{B}(z)^{(R)}$;
        \item the set of nondegenerate solutions of XXZ Bethe ansatz equations with open boundary conditions for type $R$;
        \item the energy level set of the trigonometric Ruijsenaars-Schneider model of type $R$;
        \item the set of nondegenerate solutions of the folded $QQ$-system of type $R$.
    \end{enumerate}
\end{Thm}

In other words, $\mathbb{Z}/2\mathbb{Z}$ folding is implemented differently on each side of the quantum/classical duality. On the classical tRS side, it provides mirror boundary conditions when the particle positions are mutual reflections of each other; in the $\mathbb{C}^\times$-valued coordinates and momenta, these conditions read $a_ia_{N+1-i}^{-1}=1$ and $p_ip_{N+1-i}^{-1}=1$. On the quantum XXZ spin chain side, it folds the Hilbert space of states of the closed periodic spin chain, thereby transforming it into the open spin chain. The latter has been reviewed in the introduction to \cite{Koroteev:2025tka}.

\subsection{Relation to Earlier Work}
The geometric part of our construction extends the correspondence
between Miura--Pl\"ucker $q$-opers and $QQ$-systems developed in
\cite{Frenkel:2020} (see also \cite{KSZ}), as well as the reflection-invariant $q$-opers framework of
\cite{Koroteev:2025tka}. The quantum/classical identification extends
the type-$A$ result of \cite{Koroteev:2023aa}.  On the classical side,
the commuting difference Hamiltonians for the classical root systems
were constructed by van Diejen
\cites{vanDiejen1994,vanDiejen1995}; our Lax description uses the
Cherednik operator construction of Chalykh \cite{Chalykh2019}, together
with the earlier explicit classical-type Lax matrices discussed in
\cite{ChenHou2001,AvanRollet2002}. One of the roles of the present paper is to identify the corresponding spectral level sets of the trigonometric Ruijsenaars-Schneider model with $q$-oper and open spin chain Bethe ansatz data and to explain their boundary factors through reflection and gluing.

Historically, the origin of the quantum/classical duality lies in the connection between quantum transfer matrices and the modified Kadomtsev-Petviashvili (mKP) hierarchy: the so-called $T$-operator has eigenvalues that are classical $\tau$-functions, and the evolution of their zeros in the spectral parameter is governed by the Ruijsenaars-Schneider dynamics \cite{AlexandrovEtAl2013}. This construction underlies the correspondence between the inhomogeneous XXX spin chain and the rational Ruijsenaars-Schneider (rRS) system, together with its Gaudin/Calogero-Moser degenerations \cite{GorskyZabrodinZotov2014}. At the same time, its trigonometric extension relates the XXZ spin chain to the trigonometric Ruijsenaars-Schneider system \cite{BeketovLiashykZabrodinZotov2016}.

The first example of the quantum/classical duality at the level of Gaudin and  Calogero--Moser systems was obtained by Mukhin, Tarasov, and Varchenko \cite{MukhinTarasovVarchenko2012} through the identification of the characteristic variety of the Knizhnik-Zamolodchikov equations with the zero locus of the classical Calogero-Moser Hamiltonians.

Beyond type $A$, the quantum/classical duality was established between quantum Gaudin magnets with boundary and classical Calogero-Moser systems associated with the $B$, $C$, and $D$ root systems \cite{VasilyevZabrodinZotov2020Boundary}.
Supersymmetric extensions include $\mathfrak{gl}(N|M)$ spin chains in the type-$A$ setting, whose $T$-operator eigenvalues are the mKP $\tau$-functions \cite{TsuboiZabrodinZotov2015}, and the $\mathfrak{gl}(1|1)$ Gaudin magnet with boundary, corresponding to Calogero--Moser systems of types $B$, $C$, and $D$ \cite{VasilyevZabrodinZotov2020Super}.

Additionally, there is a useful comparison with the geometry of quiver
varieties for classical groups. Li's construction of fixed-point loci
of Nakajima varieties under combinations of diagram automorphisms,
reflection functors, and transpose maps gives quiver varieties for
symmetric pairs \cite{Li2018SymmetricPairs}. Nakajima's
$\sigma$-quiver varieties similarly realize instanton moduli for
classical groups as fixed loci of involutions and relate their
equivariant geometry to quantum symmetric pairs and boundary
$K$-matrices \cite{Nakajima2025SigmaQuiver}. Our construction follows
the same fixed-locus principle on the $q$-connection side: the ambient
objects and quivers remain of type $A$, while classical-type data are
extracted by an involution. The gluing functions $B_i(z)$ record the
compatibility between the reflected Pl\"ucker pieces. We expect to be able to establish, in future publications, precise connections between our results and both the work of \cite{Li2018SymmetricPairs, Nakajima2025SigmaQuiver} and a plethora of physics papers by Witten, Gaiotto, Hanany et al.

\subsection{Structure of the paper}
\Secref{Sec:RIMPOpers} develops the geometry of generalized $Z$-twisted
reflection-invariant Miura--Pl\"ucker $(GL(N),q)$-opers.  It introduces
the gluing isomorphisms, derives the folded $QQ$-system and open-chain
Bethe equations, and defines the corresponding function algebras for
framed type-$A$ quivers. \Secref{Sec:tRSClassical} reviews the classical tRS phase
space, performs the $\mathbb Z/2\mathbb Z$ folding, and gives the Lax
matrices and potentials for the $B_N$, $C_N$, $D_N$, and $BC_N$
families. \Secref{sec:QCD} proves the quantum/classical correspondence by a
master residue identity and recursive momentum parameterization; it
also isolates the boundary factors and formulates the isomorphism of
coordinate algebras.

\Secref{Sec:GL2folding} treats rank-one and $GL(2)$ examples, where the folding,
rational twists, and boundary contributions can be seen explicitly.
Sections~\ref{sec:CD} and \ref{sec:BCB} construct the folded $QQ$-systems for the $C/D$ and
$B/BC$ families, respectively, derive their Bethe equations, and
compute the gluing isomorphisms that realize them as
reflection-invariant Miura--Pl\"ucker opers. In \Secref{sec:SummaryConj}, we discuss the extension
of the folding and gluing formalism to generic type-$A$ quivers.

\subsection{Acknowledgments}
We thank B. Vlaar, R. Weston, and O. Chalykh for fruitful discussions.
MS is supported by the Beijing Natural Science Foundation (BJNSF) grant IS24010 and by the Shuimu Scholar program of Tsinghua University. PK thanks the Beijing Junior Talent Award Committee and acknowledges support from the Beijing Natural Science Foundation (BJNSF) under grant IS26008.

\section{Reflection-Invariant \texorpdfstring{$(GL(N),q)$}{(GL(N),q)}-opers}\label{Sec:RIMPOpers}
In this section, we recall generalized $Z$-twisted Miura $(GL(N),q)$-opers from \cite{Koroteev:2025tka} and then define a reflection-invariant version of twisted Miura-Pl\"ucker $(GL(N),q)$-opers (cf. \cite[Section 4]{Frenkel:2020}).
\subsection{Generalized $Z$-Twisted Miura $(GL(N),q)$-opers}
Let $q\in\mathbb{C}^\times$ be not a root of unity. Consider a holomorphic vector bundle $\mathcal{E}$ of rank $N$ over $\mathbb{P}^1$ 
and let $M_q$ denote
the scaling of the base $M_q:\mathbb{P}^{1}\to \mathbb{P}^{1}$ that sends $z\mapsto qz$. Let $\mathcal{E}^q$ denote the pullback of $\mathcal{E}$ under $M_q$.

Recall (\cite[Section 2]{Frenkel:2020}) that a meromorphic $q$-connection $A$ on $\mathcal{E}$ is a meromorphic section of $\text{Hom}_{\mathcal{O}_{\mathbb{P}^1}}(\mathcal{E},\mathcal{E}^q)$.
\begin{Def}\label{Def:qOper}
    A meromorphic $(GL(N),q)$-oper on $\mathbb{P}^1$ is a triple $(\mathcal{E},A,\mathcal{L}_\bullet)$ where for an open Zariski dense subset $U\subset\mathbb{P}^1$ and $V=U\cap M_q^{-1}(U)$ the $q$-connection $A\in\text{Hom}_{\mathcal{O}(V)}(\mathcal{E},\mathcal{E}^q)$ satisfies the following conditions on the complete flag $\mathcal{L}_\bullet$ of subbundles $\mathcal{L}_{1}\subset \dots\subset\mathcal{L}_{N-1}\subset \mathcal{E}$

    i) $A\cdot \mathcal{L}_i\subset \left(\mathcal{L}_{i+1}\right)^{q}$

    ii) The induced maps $\bar{A_i}:\mathcal{L}_{i}/\mathcal{L}_{i-1}\to \left(\mathcal{L}_{i+1}/\mathcal{L}_{i}\right)^{q}$ are isomorphisms on an open dense subset of $\mathbb{P}^1$.
\end{Def}

Changing the trivialization of $E$ via $g(z)\in GL(N)(z)$ yields the following $q$-gauge transformation
\begin{equation}
    A(z)\mapsto g(qz)A(z)g(z)^{-1} \,.
\end{equation}
Let $H\subset GL(N)$ be the Cartan subgroup consisting of diagonal matrices and let $B\subset GL(N)$ (resp. $B_{-}$) be the Borel subgroup containing $H$ consisting of upper (resp. lower) triangular matrices.
\begin{Def}
    We say that the oper is \textit{generalized $Z$-twisted} if there is a gauge change $g(z)\in GL(N)(z)$ such that $g(qz)A(z)g(z)^{-1}=Z(z)\in H(z)$ 
\end{Def}
\begin{Def}
    A Miura $(GL(N),q)$-oper is a quadruple $(\mathcal{E},A,\mathcal{L}_\bullet,\hat{\mathcal{L}}_\bullet)$ where $(\mathcal{E},A,\mathcal{L}_\bullet)$ is a $(GL(N),q)$-oper and the complete flag $\hat{\mathcal{L}}_\bullet$ of subbundles is preserved by the $q$-connection.
\end{Def}
\begin{Rem} We remark here that Miura $(GL(N),q)$-opers may also be equivalently defined using reductions of principal $GL(N)$-bundles to Borel subgroups by viewing a full flag of subbundles as reductions to Borel subgroups (see \cite{Frenkel:2020} for more details).
\end{Rem}
\begin{Def}
The Miura $(GL(N),q)$-oper is called \textit{generalized $Z$-twisted} if there is a $q$-gauge transformation $v(z)\in B(z)$ such that
\begin{equation}\label{eq:MiuraGenN}
A(z)=v(qz)Z(z)v(z)^{-1}
\end{equation}
where $Z(z)\in H(z)$.
\end{Def}
\begin{Def}
    We say that a (Miura) $(GL(N),q)$-oper has regular singularities defined by a collection of Laurent polynomials $\Lambda_i(z)$ for $i=1,\dots, N-1$ if the induced maps $\bar{A}_i$ in Definition \ref{Def:qOper} are isomorphisms away from the roots of these polynomials.
\end{Def}
\subsection{Generalized $Z$-Twisted Reflection-Invariant Miura-Pl\"ucker $(GL(N),q)$-opers}
In this subsection, we define a reflection-invariant version of Miura-Plucker opers that are obtained by gluing the associated Miura $(GL(2),q)$-opers along the Miura line subbundles. We first recall the notation introduced in \cite[Section 4]{Frenkel:2020}. 

Let $\alpha_{i}$, (resp. $\omega_i$) $1\leq i\leq N-1$ be the simple roots (resp. fundamental weights) of $GL(N)$ corresponding to the choice of the pair $H\subset B$. Let $V_i$ be the irreducible representation of $GL(N)$ with the highest weight $\omega_i$. Let $\nu_{\omega_i}$ be the highest weight vector of $V_i$ generating a line, say $L_i\subset V_i$. Let $W_i\subset V_i$ denote the two-dimensional subspace spanned by the vectors $\nu_{\omega_i}$ and $f_{i}\cdot\nu_{\omega_i}$. Note that the subspaces $L_i$ and $W_i$ are invariant under the action of $B$.

Now, let $(\mathcal{F},A,\mathcal{F}_{B_-},\mathcal{F}_{B})$ be a Miura $(GL(N),q)$-oper with regular singularities determined
by polynomials $\Lambda_1(z)\ldots, \Lambda_{N-1}(z)$. Here, $\mathcal{F}$ is a principal $GL(N)$-bundle, and $\mathcal{F}_{B_{-}}$ (resp. $\mathcal{F}_{B}$) are $B_{-}$(resp. $B$)-reductions. Since
$\mathcal{F}_B$ is preserved by the $(GL(N),q)$-connection $A$, for each $i=1,\ldots,N-1$, we have the following inclusions of
vector bundles over $\mathbb{P}^1$:
\[
\mathcal{L}_i\subset\mathcal{W}_i\subset\mathcal{V}_i,
\]
where
\[
\mathcal{L}_i=\mathcal{F}_{B}\times_{B} L_i
\quad
\mathcal{W}_i=\mathcal{F}_{B}\times_{B} W_i
\quad
\mathcal{V}_i=\mathcal{F}_{B}\times_{B} V_i
= \mathcal{F}\times_{GL(N)} V_i
\]
are vector bundles of rank $1$, $2$ and $\binom{N}{i}$ respectively.

Let $\phi_{i}(A)$ denote the induced $q$-connection on the vector bundle $\mathcal{V}_i$. Since $A$ preserves $\mathcal{F}_B$, $\phi_{i}(A)$ preserves the subbundles $\mathcal{L}_{i}$ and $\mathcal{W}_i$ and hence gives a $q$-connection on $\mathcal{W}_i$ (resp. $\mathcal{L}_i$), which we denote by $A_i(z)$ (resp. $A_i^{M}(z)$).

One can show that we get $N-1$ Miura $(GL(2),q)$-opers $(\mathcal{W}_i,A_{i},\tilde{\mathcal{L}}_{i},\mathcal{L}_{i})$, where $\tilde{\mathcal{L}}_{i}$ is a line subbundle of $\mathcal{W}_{i}$ that is generated by the basis vector $\begin{pmatrix}
0\\
1
\end{pmatrix}$ 
in the gauge where $A_{i}$ has a certain upper triangular form (see \cite{Frenkel:2020} for more details).
\begin{Def}\label{mpopers}
    Let $\zeta_{i}(z), 1\leq i\leq N-1$ be rational functions. A generalized $Z$-twisted reflection-invariant Miura-Pl\"ucker $(GL(N),q)$-oper is a meromorphic Miura $(GL(N),q)$-oper $(\mathcal{F},A,\mathcal{F}_{B_-},\mathcal{F}_{B})$
on $\mathbb{P}^1$ satisfying the following conditions:

    i) There exists $v(z)\in B(z)$ such that the associated
Miura $(GL(2),q)$-opers $(\mathcal{W}_i,A_{i},\tilde{\mathcal{L}}_{i},\mathcal{L}_{i})$ can be written in the form
\begin{equation}
A_i(z)
=
v_i(qz)\,Z_i(z)\,v_i(z)^{-1}, \qquad 1\leq i \leq N-1,
\end{equation}
where $v_i(z)=v(z)\big|_{W_i}$ and $
Z_i(z)=\begin{pmatrix}
       \widetilde{\zeta}_{i+1}(z) & 0 \\
        0 &  \zeta_{i}(z)
    \end{pmatrix}
$ for some $\widetilde{\zeta}_{i+1}(z)$. In addition, we require the associated $Z_i(z)$-twisted Miura $(GL(2),q)$-opers to be $i$-th reflection-invariant, i.e. each $\tilde{\mathcal{L}}_i$ is generated by a section $s_{i}(z):=\begin{pmatrix}
    Q^-_{i}(z)\\Q^+_{i}(z)
\end{pmatrix}$ such that $s_{i}\left(\frac{q^{N-i}}{z}\right)=s_i(z)$.

  ii) The rank-one Miura $q$-connections $A_{i}^{M}$ are glued via gluing isomorphisms, that is, for each $1\leq i\leq N-2$, there exist isomorphisms of rank-one $q$-connections
\begin{align}\label{gluing}
   \phi_{i}:(\mathcal{L}_{i+1},A_{i+1}^{M})\rightarrow \left(\rho_{i}^{*}\left(\left(\mathcal{L}_{i}^{q}\right)^{\vee}\right),\rho_{i}^{*}\left(\left(A_{i}^{M}\right)^{\vee}\right)\right)\otimes (\mathcal{O}_{\mathbb{P}^{1}},B_{i}),
\end{align} 
where $\left(\mathcal{L}_{i}^{q}\right)^{\vee}$ denotes the line bundle dual to $\mathcal{L}_{i}^{q}$, $\left(A_{i}^{M}\right)^{\vee}$ denotes the morphism of line bundles dual to $A_{i}^{M}$, $(\mathcal{O}_{\mathbb{P}^{1}},B_{i})$ is the rank-one $q$-connection on the trivial line bundle over $\mathbb{P}^1$ with $q$-connection $B_i(z)\in\mathbb{C}(z)^{\times}$ and
\begin{align}
    \rho_i:\mathbb{P}^1\rightarrow\mathbb{P}^1, \qquad z\mapsto\frac{1}{q^{i+1-u_{i}}z}
\end{align}
for some $u_i\in\frac{1}{2}\mathbb{Z}$.
\end{Def}

\begin{Rem}
i) In \cite{Koroteev:2025tka}, the notion of generalized $Z$-twisted reflection-invariant Miura $(GL(N),q)$-opers is defined. However, for our purposes, we will need the notion of generalized $Z$-twisted reflection-invariant Miura-Pl\"ucker $(GL(N),q)$-opers, which strictly generalizes reflection-invariant Miura $(GL(N),q)$-opers when $N>2$. 

ii) Note that, unlike \cite{Frenkel:2020}, we cannot write the twists of a generalized $Z$-twisted reflection-invariant Miura-Pl\"ucker $(GL(N),q)$-oper as an $N\times N$ matrix.

iii) The naive pullback $\rho_{i}^{*}(A_{i}^{M})$ of the $q$-connection $A_{i}^{M}$ is a $q^{-1}$-connection. Taking the dual as above gives us a $q$-connection.

iv) An important feature of the present definition is that, for suitable gluing data, it admits constant twist solutions when $N>2$, in contrast to the more restrictive notion of generalized $Z$-twisted reflection-invariant Miura $(GL(N),q)$-opers considered in \cite{Koroteev:2025tka}.
\end{Rem}

The following proposition shows that the functions $\widetilde{\zeta}_{i+1}(z)$ in the above definition are not arbitrary.
\begin{Prop}
    For a generalized $Z$-twisted reflection-invariant Miura-Pl\"ucker $(GL(N),q)$-oper with regular singularities given by Laurent polynomials $\Lambda_i(z)$, $i=1,\dots, N-1$, one has the following constraints
    \begin{equation}
    \widetilde{\zeta}_{i+1}(z)=-\zeta_{i}\left(\frac{q^{N-1-i}}{z}\right), \quad 1\leq i\leq N-1,
\end{equation}
\begin{equation}\label{eq:involutionofQ}
    Q^{\pm}_{i}\left(\frac{q^{N-i}}{z}\right)=Q^{\pm}_{i}(z), \quad 1\leq i\leq N-1,
\end{equation}
and
\begin{equation}
    \Lambda_{i}\left(\frac{q^{N-1-i}}{z}\right)=\Lambda_{i}(z), \quad 1\leq i\leq N-1.
\end{equation}
Moreover, one has 
\begin{align}
    \zeta_{i+1}(z)=B_{i}(z)\zeta_i\left(\frac{1}{q^{i+1-u_{i}}z}\right)\label{eq:gluing_general}, \qquad
    1\leq i\leq N-2.
\end{align}
\end{Prop}
\begin{proof}
    The first three parts of the proposition follow from the genericity argument used in \cite[Theorem 2.7]{Koroteev:2025tka}. The last part of the proposition is a restatement of the gluing condition in the definition of generalized $Z$-twisted reflection-invariant Miura-Pl\"ucker $(GL(N),q)$-oper. 
\end{proof}
\subsection{The $QQ$-system and the Bethe equations}
Similarly to \cite[Section 4]{Frenkel:2020}, the $QQ$-system associated to a generalized $Z$-twisted reflection-invariant Miura-Pl\"ucker $(GL(N),q)$-oper with regular singularities given by Laurent polynomials $\Lambda_i(z)$, $i=1,\dots, N-1$ reads as
\begin{align}\label{eq:foldedQQ}
\widetilde{\zeta}_{i+1}(z)Q_i^{+}(qz)Q_i^{-}(z)
-
\zeta_i(z)Q_i^{+}(z)Q_i^{-}(qz)
=
\Lambda_i(z)Q_{i-1}^{+}(qz)Q_{i+1}^{+}(z),
\qquad
i=1,\ldots,N-1.
\end{align}
We now extract the Bethe equations from the $QQ$-system \eqref{eq:foldedQQ}. Following the standard argument for $Z$-twisted $q$-opers, we evaluate \eqref{eq:foldedQQ} at the roots $s_{i,j}$ of $Q_i^{+}(z)$ and at $q^{-1}s_{i,j}$.
Writing $Q_{i}^{+}(z)$ as
\begin{align}
Q_{i}^{+}(z)=
\prod_{j=1}^{p_{i}}\left(z-s_{i,j}\right)\left(1-\frac{q^{N-i}}{s_{i,j}z}\right), \quad i=1,\ldots,N-1,
\end{align}
the Bethe equations take the form
\begin{equation}\label{eq:betheforgeneralizedmiurapluckeropers}
-1
=
\frac{\zeta_i(q^{-1}s_{i,m})}
     {\widetilde{\zeta}_{i+1}(s_{i,m})}
\cdot
\frac{Q_{i-1}^{+}(qs_{i,m})}
     {Q_{i-1}^{+}(s_{i,m})}
\cdot
\frac{Q_i^{+}(q^{-1}s_{i,m})}
     {Q_i^{+}(qs_{i,m})}
\cdot
\frac{Q_{i+1}^{+}(s_{i,m})}
     {Q_{i+1}^{+}(q^{-1}s_{i,m})}
\cdot
\frac{\Lambda_{i}(s_{i,m})}
     {\Lambda_{i}(q^{-1}s_{i,m})}
     ,
     \end{equation}
     where $i=1,\dots, N-1, m=1,\ldots,p_i:=\deg(Q_{i}^{+})$, by which we mean the degree of $Q_{i}^{+}$ viewed as a polynomial in $z+\frac{q^{N-i}}{z}$.
     
More explicitly, 
\begin{align}\label{eq:explicitBethe}
1&=q^{p_{i+1}-p_i}\frac{\zeta_i(q^{-1}s_{i,m})}
     {\zeta_i\left(\frac{q^{N-1-i}}{s_{i,m}}\right)}
\cdot
\frac{\Lambda_i(s_{i,m})}
     {\Lambda_i(q^{-1}s_{i,m})}
     \cr
&\cdot\prod_{j=1}^{p_i}
\frac{\left(s_{i,m}-q s_{i,j}\right)}
     {\left(qs_{i,m}-s_{i,j}\right)}
\frac{\left(s_{i,m} s_{i,j}-q^{N+1-i}\right)}
     {\left(s_{i,m} s_{i,j}-q^{N-1-i}\right)}
     \cr
&\cdot\prod_{j=1}^{p_{i+1}}
\frac{\left(s_{i,m}-s_{i+1,j}\right)}
     {\left(s_{i,m}-q s_{i+1,j}\right)}
\frac{\left(s_{i,m} s_{i+1,j}-q^{N-1-i}\right)}
     {\left(s_{i,m} s_{i+1,j}-q^{N-i}\right)}
     \cr
&\cdot\prod_{j=1}^{p_{i-1}}
\frac{\left(qs_{i,m}-s_{i-1,j}\right)}
     {\left(s_{i,m}-s_{i-1,j}\right)}
\frac{\left(s_{i,m} s_{i-1,j}-q^{N-i}\right)}
     {\left(s_{i,m} s_{i-1,j}-q^{N+1-i}\right)}
\end{align}
for $i=1,\ldots,N-1$, $m=1,\ldots,p_i$.

Similarly to \cite[Section 3]{Koroteev:2025tka}, we say that a generalized $Z$-twisted reflection-invariant Miura-Pl\"ucker $(GL(N),q)$-oper is \emph{nondegenerate} if no two zeros of $Q_{i}^{\pm}(z)$ and $\Lambda_{i}(z)$ that appear in the Definition \ref{mpopers} belong to the same integral $q$-lattice: $q^\mathbb{Z}s^{+}_i\cap q^\mathbb{Z}s^{-}_j=q^\mathbb{Z}(s^{+}_i)^{\pm 1}\cap q^\mathbb{Z}(s^{+}_j)^{\pm 1}=q^\mathbb{Z}a^{+}_i\cap q^\mathbb{Z}a^{-}_j=\varnothing$.

Now we will define the nondegeneracy of solutions of the folded Bethe equations \eqref{eq:explicitBethe} in an analogous way to \cite[Section 3]{Koroteev:2025tka}. We will say that a solution of the folded Bethe equations \eqref{eq:explicitBethe}  is \emph{nondegenerate} if $q^\mathbb{Z}(s^{+}_i)^{\pm 1}\cap q^\mathbb{Z}(s^{+}_j)^{\pm 1}=q^\mathbb{Z}a^{+}_i\cap q^\mathbb{Z}a^{-}_j=\varnothing$.

Let $\zeta_{i}(z)=\sum_{m=N_{i}}^{\infty}\zeta_{i,m}z^{m}$ and $\widetilde{\zeta}_{i+1}(z)=\sum_{m=N_{i+1}'}^{\infty}\widetilde{\zeta}_{i+1,m}z^{m}$ be the Laurent expansions of $\zeta_{i}(z)$ and $\widetilde{\zeta}_{i+1}(z)$ around $0$, where $N_{i},N_{i+1}'\in\mathbb{Z}$ and $\zeta_{i,N_{i}},\widetilde{\zeta}_{i+1,N_{i}'}\neq 0$. From now on, we also assume that 
\begin{equation}\label{nondegeneracy}
\frac{\widetilde{\zeta}_{i+1,\text{min}(N_{i+1}',N_{i})}}{\zeta_{i,\text{min}(N_{i+1}',N_{i})}}\notin q^{\mathbb{Z}}\quad \text{or} \quad\frac{\zeta_{i,\text{min}(N_{i+1}',N_{i})}}{\widetilde{\zeta}_{i+1,\text{min}(N_{i+1}',N_{i})}}\notin q^{\mathbb{Z}}, \quad 1\leq i\leq N-1.
\end{equation}
The following theorem is an adaptation of the proofs of \cite[Theorem 4.1]{Koroteev:2025tka} and \cite[Theorem 6.1, Theorem 6.4]{Frenkel:2020}.
\begin{Thm}\label{thm:opersandbethe}
For fixed gluing data of Miura line subbundles, i.e. for fixed $u_{i}\in\frac{1}{2}\mathbb{Z}$ and $B_{i}(z)$, $i=1,\ldots,N-2$,
there is a one-to-one correspondence between the set of isomorphism classes of nondegenerate generalized $Z$-twisted reflection-invariant Miura-Pl\"ucker $(GL(N),q)$-opers with regular singularities $\Lambda_i(z)$ and the set of nondegenerate solutions of the open XXZ Bethe equations \eqref{eq:explicitBethe}.
\end{Thm}

\subsection{Quivers and Miura $q$-opers} In this subsection, we define an algebra of functions on the space of generalized $Z$-twisted reflection-invariant Miura-Pl\"ucker $(GL(r+1),q)$-opers with twists having a prescribed number of zeros and poles and fixed gluing data $\bm{B}=\left(\left(B_{1}(z),\ldots,B_{r-1}(z)),(u_{1},\ldots,u_{r-1}\right)\right)$ along the lines of \cite[Section~3]{Koroteev:2023aa}.

Recall \cite[Section~3]{Koroteev:2023aa} that a framed quiver of type $A_r$ with dimension vectors $\bm{v}=(v_{1},\ldots,v_r)$, $\bm{w}=(w_1,\ldots,w_r)$ is the following labeled graph $Y_{\bm{v},\bm{w}}$:
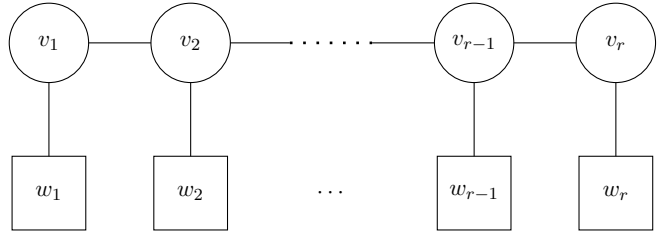
\begin{center}
        \begin{tikzpicture}[scale=.75, every node/.style={scale=0.8}]
        \draw[] (-6,0) circle (.7);
        \draw[] (-5.3,0) -- (-4.2,0);
        \draw[] (-3.5,0) circle (.7);
        \draw[] (-2.8,0) -- (-1.75,0);
        \draw[line width=1pt,loosely dotted] (-1.75,0) -- (-.3,0);
        \draw[] (-.3,0) -- (.8,0);
        \draw[] (1.5,0) circle (.7);   
        \draw[] (2.2,0) -- (3.3,0);
        \draw[] (4,0) circle (.7);      
        \draw[] (4,-.7) -- (4,-2);
        \draw[] (1.5,-.7) -- (1.5,-2);
        \draw[] (-3.5,-.7) -- (-3.5,-2);
        \draw[] (-6,-.7) -- (-6,-2);   
        \draw[] (3.35,-2) rectangle (4.65,-3.3);
        \draw[] (0.85,-2) rectangle (2.15,-3.3);
        \draw[] (-4.15,-2) rectangle (-2.85,-3.3);
        \draw[] (-6.65,-2) rectangle (-5.35,-3.3);
  
        \node at (-6,0) {$v_1$};
        \node at (-3.5,0) {$v_2$};
        \node at (1.5,0) {$v_{r-1}$};
        \node at (4,0) {$v_r$};

        \node at (-6,-2.65) {$w_1$};
        \node at (-3.5,-2.65) {$w_2$};
        \node at (1.5,-2.65) {$w_{r-1}$};
        \node at (4,-2.65) {$w_r$};
        \node at (-1,-2.65) {$\cdots$};
    \end{tikzpicture}
    \end{center}
Given such a framed quiver $Y_{\bm{v},\bm{w}}$, gluing data $\bm{B}$ and vectors $\bm{n}=(n_i)_{i=1}^{r},\bm{d}=(d_i)_{i=1}^{r}\in\mathbb{Z}_{\geq 0}^{r}$, consider the associated space $q\Op_{\bm{n},\bm{d}}(Y_{\bm{v},\bm{w}},\bm{B})$ of generalized $Z$-twisted reflection-invariant Miura-Pl\"ucker $(GL(r+1),q)$-opers. More precisely, $q\Op_{\bm{n},\bm{d}}(Y_{\bm{v},\bm{w}},\bm{B})$ is defined as follows.

Counting vertices from left to right, let us associate to each vertex $i$ with the label $v_i$ a Laurent polynomial $Q_{i}^{+}(z)$ that is invariant under the involution
\begin{align}
    z\mapsto \frac{q^{r+1-i}}{z}.
\end{align}
In other words, $Q_{i}^{+}(z)\in\mathbb{C}\left[z+\frac{q^{r+1-i}}{z}\right]$. Moreover, we require $\deg(Q_{i}^{+})=v_i$ as a polynomial in $z+\frac{q^{r+1-i}}{z}$. In addition, to each framed vertex $i$ with the label $w_i$, we associate a Laurent polynomial $\Lambda_{i}(z)$ that is invariant under the involution
\begin{align}\label{eq:twistrelations}
    z\mapsto \frac{q^{r-i}}{z}
\end{align}
and $\deg(\Lambda_{i})=w_i$ as a polynomial in $z+\frac{q^{r-i}}{z}$. Moreover, the structure of the twists is determined by $\bm{n}$ and $\bm{d}$, that is,
\begin{align}
    \zeta_{i}(z)=c_{i}\frac{z^{n_{i}}+\alpha_{i,1}z^{n_{i}-1}+\ldots+\alpha_{i,n_i}}{z^{d_{i}}+\beta_{i,1}z^{d_{i}-1}+\ldots+\beta_{i,d_i}},
    \qquad
    c_i\in\mathbb{C}^{\times},\quad 
    1\leq i\leq r.
\end{align}
Then $q\Op_{\bm{n},\bm{d}}(Y_{\bm{v},\bm{w}},\bm{B})$ is by definition the space of nondegenerate generalized $Z$-twisted reflection-invariant Miura-Pl\"ucker $(GL(r+1),q)$-opers with $Q^{+}_{i}(z), \Lambda_{i}(z),\zeta_{i}(z)$ as above and whose twists $\zeta_{i}(z)$ satisfy the gluing conditions:
\begin{align}\label{eq:gluingconditionontwists}
    \zeta_{i+1}(z)-B_{i}(z)\zeta_{i}\left(\frac{1}{q^{i+1-u_i}z}\right)=0,
    \qquad 1\leq i\leq r-1.
\end{align}
Next, in order to define the algebra $\Fun\left(q\Op_{\bm{n},\bm{d}}(Y_{\bm{v},\bm{w}},\bm{B})\right)$ of functions on the space $q\Op_{\bm{n},\bm{d}}(Y_{\bm{v},\bm{w}},\bm{B})$, let us write 
\begin{align}
    Q_{i}^{+}(z)=\prod_{j=1}^{v_i}\left(z-s_{i,j}\right)\left(\frac{q^{r+1-i}}{z}-s_{i,j}\right),
    \qquad
     \Lambda_{i}(z)=\prod_{j=1}^{w_i}\left(z-a_{i,j}\right)\left(\frac{q^{r-i}}{z}-a_{i,j}\right).
\end{align}
To simplify the notation for further discussion, let us rescale Bethe root variables $s_{i,j}$ and $q$-oper singularities $a_{i,k}$ as follows
\begin{align}
    u_{i,j}=q^{-(r+1-i)/2}s_{i,j}, 
    \qquad
    t_{i,k}=q^{-(r-i)/2}a_{i,k}.
\end{align}

As in \cite[Section~3]{Koroteev:2023aa}, we define the space of functions on the space of $q$-opers $\Fun\left(q\Op_{\bm{n},\bm{d}}(Y_{\bm{v},\bm{w}},\bm{B})\right)$ as follows.

\begin{Def}\label{Def:FunqOp}
Let 
\begin{equation}
\Fun\left(q\Op_{\bm{n},\bm{d}}(Y_{\bm{v},\bm{w}},\bm{B})\right):=\frac{\mathbb{C}\left[q^{\pm 1},c_i^{\pm 1},\{\alpha_{i,k}^{\pm 1}\},\{\beta_{i,k}^{\pm 1}\}\right]\otimes_{\mathbb{C}} S\left(\{u_{i,j}\},\{t_{i,k}\}\right)^{\mathbb{Z}/{2\mathbb{Z}}}}{(\text{Gluing }\eqref{eq:gluingconditionontwists},\text{Bethe }\eqref{eq:betheforgeneralizedmiurapluckeropers})},
\end{equation}
where the $\mathbb{Z}/{2\mathbb{Z}}$ acts as $u_{i,j}\mapsto u_{i,j}^{-1}$ and $t_{i,k}\mapsto t_{i,k}^{-1}$ and $S(\{u_{i,j}\},\{t_{i,k}\})$ is the $\mathbb{C}$-algebra of rational functions in the variables $\{u_{i,j}\},\{t_{i,k}\}$ that are symmetric in $\{u_{i,j}\}$ for each $i$ and are also symmetric in $\{t_{i,k}\}$ for each $i$, and the variables $q,c_i,\{\alpha_{i,k}\},\{\beta_{i,k}\}$ satisfy the gluing condition \eqref{eq:gluingconditionontwists} on the twists and $\{u_{i,j}\}$ satisfy the Bethe equations \eqref{eq:betheforgeneralizedmiurapluckeropers}.
\end{Def}

\section{tRS Models for Classical Groups}\label{Sec:tRSClassical}
In this section, we reconstruct the formulation of the trigonometric Ruijsenaars--Schneider models for all classical groups using the notation of \cite{Koroteev:2023aa}. The commuting trigonometric finite-difference Hamiltonians for the classical root systems were first constructed by van Diejen in \cite{vanDiejen1994,vanDiejen1995}. We shall use the Cauchy kernel realization for the direct type-$A$ foldings and the Cherednik-operator Lax realization of Chalykh \cite{Chalykh2019} for types $B$ and $D$.

Since the integrability of tRS models for these classical types has been demonstrated in the papers we have previously cited, we shall refrain from giving it here. The goal of this section is the to set up the framework for the quantum/classical duality which will be explored immediately after in \Secref{sec:QCD}.

\subsection{The Calogero-Moser Space}
Let us recall the definition of the Calogero-Moser space in type $A$.
\begin{Def}[\cite{Oblomkov:aa}]
Let $V$ be an $N$-dimensional vector space over $\mathbb{C}$. Let $\mathcal{M}'$ be the subset of $GL(V)\times GL(V)\times V\times V^\ast$ consisting of elements $(M,T, u,v)$ such that
\begin{equation}\label{eq:FlatConNew}
q M T -  T M = u\otimes v^T\,.
\end{equation}
The group $GL(N;\mathbb{C})=GL(V)$ acts on $\mathcal{M}'$ by conjugation
\begin{equation}\label{eq:SimilarityTRansf}
(M,T, u,v)\mapsto (gMg^{-1},gTg^{-1}, gu,vg^{-1})\,,\qquad g\in GL(V)\,.
\end{equation}
The quotient of $\mathcal{M}'$ by the action of $GL(V)$ is called the Calogero-Moser space $\mathcal{M}$.
\end{Def}

In the basis where $M$ is a diagonal matrix
\[
M=\operatorname{diag}(x_1,\ldots,x_{N})
\]
This leads to the following solution for the matrix $T$
\begin{equation}
    T_{ij}=\frac{u_iv_j}{qx_i-x_j}
\end{equation}
Next, we conjugate $T$ by a diagonal matrix $U=\text{diag}(u_1,\dots, u_N)$ to get (the so-called Cauchy kernel)
\begin{equation}
    T_{ij}\mapsto (U^{-1}TU)_{ij}=\frac{\Gamma_j}{qx_i-x_j}\,,\qquad \Gamma_j = u_jv_j.
\end{equation}
This gives us the description of the model in terms of canonical variables $(\Gamma_i, x_i)$. 

The diagonal part then reads
\[
  T_{ii}=\frac{\Gamma_i}{(q-1)x_i}.
\]
In order to follow the conventions of \cite{Oblomkov:aa} as well as \cite{Koroteev:2023aa} we need to further rescale $\Gamma_j$ by $(q-1)x_j$ to get
\begin{equation}\label{eq:CauchyKernel}
  T_{ij}=\frac{(q-1)x_j}{qx_i-x_j}\,\Gamma_j,
\end{equation}
so that $T_{ii}=\Gamma_i$ and Tr $T=\sum_i \Gamma_i$ is the first tRS Hamiltonian.

Next, we consider a canonical transformation 
\begin{equation}
    (\Gamma_i, x_i)\mapsto (p_i, a_i)\,,\quad \Gamma_i =  V^R_i(a)p_i
\end{equation}
where the interacting potential $V^G_i(x)$ is invariant under the Weyl group $W_G$ reflections.

The coefficients of the characteristic polynomial of the Lax matrix are the tRS Hamiltonians $T_1,\dots, T_N$
\begin{equation}
 \text{det}\left(u\cdot 1 -  T \right) = \sum_{k=0}^N (-1)^kT_k(a_i, p_i,q) u^{N-k}\,,
\label{eq:tRSLaxDecomp1}
\end{equation}

The corresponding level set for the tRS model Hamiltonians can be obtained by equating the above characteristic polynomial to $\prod_i (u-\xi_i)$ whose roots are the $q$-distinct parameters.
\begin{equation}
T_k(a_i) = e_k (\xi_i)\,,\qquad k=1,\dots, N
\label{eq:tRSLevelEquations1}
\end{equation}
where $e_l$ is the $k$-th elementary symmetric function. In particular, the energy level set for the first tRS Hamiltonian 
\begin{equation}
T_1=\text{Tr}\, T = \sum_{i=1}^N\Gamma_i
\end{equation}
then reads
\begin{equation}\label{eq:TraceHam1}
    \sum_{i=1}^N V^R_i(a)p_i=\xi_1+\dots+\xi_N\,.
\end{equation}

\subsubsection{Spectrum of the tRS Lax Matrix}
Once one can find a change of variables from the Darboux coordinates $(p_i,a_i)$ to another set of variables such that $T_1=\text{Tr}\, T=e_1(\xi_1,\dots,\xi_M)$ then we have diagonalized the entire Lax matrix and have found its spectrum. This assumption is based on nondegeneracy of the eigenvalues $\xi_i$.

In the rest of this section, we shall provide such a change of variables for each of the classical Lie algebra types.

\subsubsection{Functions $f_q$ and $g_q$}
The following functions will be convenient for further use
\begin{equation}
      f_q(x)=\frac{x-q}{qx-1},\qquad g_q(x)=\frac{x-q}{x-1}\label{eq:fg}
\end{equation}
These functions satisfy
\[
  f_q(x^{-1})=f_q(x)^{-1},
  \qquad
  f_q(x)=\frac{g_q(x)}{g_q(x^{-1})}.
\]
as well as 
\begin{equation}
  g_q(x^{-1})=\frac{qx-1}{x-1},
  \qquad
  \frac{g_q(x)}{g_q(x^{-1})}
  =\frac{x-q}{qx-1}.
  \label{eq:g-orientations}
\end{equation}
Additionally we have
\[
  g_q(x)=\sqrt{q}\,f_{\sqrt{q}}\!\left(\frac{x}{\sqrt q}\right).
\]

\subsubsection{$A_M$ Model}
For $x=(a_1,\ldots,a_M)$ define
\begin{equation}
  V_i^A(a)=\prod_{\substack{j=1\\j\neq i}}^M
  g_q\!\left(\frac{a_i}{a_j}\right).
  \label{eq:a-potential}
\end{equation}
The Lax matrix and its first Hamiltonian are
\begin{equation}
  T_{ij}^A
  =\frac{(q-1)a_j}{qa_i-a_j}\,V_j^A(a)p_j,
  \qquad
T_1^A=\sum_{i=1}^M V_i^A(a)p_i.
  \label{eq:a-lax-hamiltonian}
\end{equation}

\subsection{$\mathbb{Z}/2\mathbb{Z}$ Folding}
Consider the tuple of the original type $A$ coordinates $(x_1,\dots,x_M)$. Depending on the parity of $M$ we will obtain two different classes of folded root systems.

For the coordinates we have the following folding prescription
\[
  x_i=a_i,
  \qquad
  x_{\bar i}=a_i^{-1},
\]
then we can write
\begin{equation}
  \Gamma_i=V_i^{R,+}(a)p_i,
  \qquad
  \Gamma_{\bar i}=V_i^{R,-}(a)p_i^{-1},
  \label{eq:oriented-columns}
\end{equation}
where
\begin{equation}
  V_i^{R,-}(a)
  =V_i^{R,+}(a_1,\ldots,a_i^{-1},\ldots,a_N).
  \label{eq:negative-orientation}
\end{equation}
The Weyl reflection $a_i\mapsto a_i^{-1}$ exchanges the two summands
in \eqref{eq:oriented-columns}.  Weyl invariance is therefore a
property of the complete Hamiltonian, while each oriented potential is
Weyl covariant.

For all folded systems define the common pair potential which corresponds to type $D$
\begin{align}
  V_i^{D,+}(a)
  &=\prod_{\substack{j=1\\j\neq i}}^n
    g_q\!\left(\frac{a_i}{a_j}\right)g_q(a_ia_j),
  \label{eq:d-plus-potential}\\
  V_i^{D,-}(a)
  &=\prod_{\substack{j=1\\j\neq i}}^n
    g_q\!\left(\frac{a_j}{a_i}\right)
    g_q\!\left(\frac{1}{a_ia_j}\right).
  \label{eq:d-minus-potential}
\end{align}
The same parameter $q$ is used for every root orbit:
\begin{equation}
  \begin{array}{c|c|c}
    R&V_i^{R,+}(a)&V_i^{R,-}(a)\\ \hline
    D_N&V_i^{D,+}&V_i^{D,-}\\[1mm]
    B_N&V_i^{D,+}g_q(a_i)&
        V_i^{D,-}g_q(a_i^{-1})\\[1mm]
    C_N&V_i^{D,+}g_q(a_i^2)&
        V_i^{D,-}g_q(a_i^{-2})\\[1mm]
    BC_N&V_i^{D,+}g_q(a_i)g_q(a_i^2)&
        V_i^{D,-}g_q(a_i^{-1})g_q(a_i^{-2}).
  \end{array}
  \label{eq:folded-potential-table}
\end{equation}
The moving folded coordinates (using the terminology of \cite{Chalykh2019}) are labelled as follows
\begin{equation}
  X_i=a_i,
  \qquad X_{\bar i}=a_i^{-1},\qquad i,{\bar i}=1,\dots, N
  \label{eq:folded-coordinates}
\end{equation}
We can also write
\begin{equation}
  \mathcal I=\{1,\ldots,N,\bar 1,\ldots,\bar N\},
  \qquad \overline{\bar i}=i.
  \label{eq:moving-index-set}
\end{equation}
For the direct $BC_N$ folding we additionally use the fixed label $0$ with
$X_0=1$.  A fixed label for the type-$B$ vector representation will be
introduced separately below; it is not part of the smallest Weyl orbit
$W(B_N)e_1=\{\pm e_i\}$ used in Chalykh's Lax matrix.

For later use set
\begin{equation}
  K_{AB}=\frac{(q-1)X_B}{qX_A-X_B}.
  \label{eq:folded-cauchy-kernel}
\end{equation}
The oriented column weights are
\begin{equation}
  \Gamma_i^R=V_i^{R,+}p_i,
  \qquad\qquad
  \Gamma_{\bar i}^R=V_i^{R,-}p_i^{-1}.
  \label{eq:folded-column-weights}
\end{equation}

For $C_N$ and $BC_N$, which are direct reductions of the type-$A$ Cauchy
matrix, we therefore retain
\begin{equation}
  T_{AB}^R=K_{AB}\Gamma_B^R.
  \label{eq:folded-lax}
\end{equation}
In type $BC_N$ the fixed column has weight
\begin{equation}
  \Gamma_0^{BC}
  =\prod_{j=1}^N g_q(a_j)g_q(a_j^{-1}).
  \label{eq:Gamma0-BC}
\end{equation}

The above are the Lax matrices of types $C_N$ and $BC_N$ from 
\cite{ChenHou2001}. It was shown in \cite{AvanRollet2002} that their characteristic polynomials reproduce the Koornwinder--van Diejen commuting family.  Their first Hamiltonians are as follows
\begin{equation}
  T_1^R=U^R+\sum_{i=1}^N
  \left(V_i^{R,+}p_i+V_i^{R,-}p_i^{-1}\right),
  \qquad
  U^C=0,\quad U^{BC}=\Gamma_0^{BC}.
  \label{eq:folded-hamiltonian-CBC}
\end{equation}

\subsubsection{Chalykh Normalization for Types $D_N$ and $B_N$}
The type-$D_N$ and type-$B_N$ Lax matrices require a modification at the
opposite-weight locations $A\leftrightarrow\bar A$. Indeed, they would
correspond to the non-roots $2e_i$.  Chalykh's construction starts from
the Koornwinder--van Diejen Lax matrix and specializes its boundary
parameters as follows (see Proposition~4.3 and Corollary~4.4 of \cite{Chalykh2019})
\begin{equation}
  (\tau_0,\tau_0^\vee,\tau_N,\tau_N^\vee)
  =
  \begin{cases}
    (1,1,1,1),&R=D_N,\\
    (1,1,\tau,\tau),&R=B_N,
  \end{cases}
  \qquad \tau=q^{-1/2}.
  \label{eq:Chalykh-BD-specialization}
\end{equation}
Here $q$ denotes the interaction parameter used throughout this section;
Chalykh's difference-step parameter has already been specialized to its
classical value.

In order to write the resulting classical matrix in our variables, define
\begin{equation}
  \mathcal V_i^R=V_i^{R,+},\qquad
  \mathcal V_{\bar i}^R=V_i^{R,-},\qquad
  \mathfrak p_i=p_i,\qquad
  \mathfrak p_{\bar i}=p_i^{-1}.
  \label{eq:extended-BD-data}
\end{equation}
The normalized Chalykh Lax matrix is
\begin{equation}
  \widehat T_{AB}^R=
  \begin{cases}
    K_{AB}\mathcal V_B^R\mathfrak p_B,
      &B\neq\bar A,\\[2mm]
    \displaystyle
    \left(1-\sum_{C\in\mathcal I\setminus\{\bar A\}}
      K_{AC}\mathcal V_C^R\right)\mathfrak p_{\bar A},
      &B=\bar A,
  \end{cases}
  \qquad R=B_N,D_N.
  \label{eq:Chalykh-BD-lax}
\end{equation}
Thus the ordinary Cauchy formula remains valid at every diagonal and
root-supported entry; only the opposite-weight entry is replaced by the
row-completion term in the second line of
\eqref{eq:Chalykh-BD-lax}.  For type $D_N$ this is the same exceptional
matrix position that is treated separately in the block Lax matrix of
Chen and Hou \cite{ChenHou2001}.

The diagonal entries are unchanged, and hence
\begin{equation}
  \operatorname{Tr}\widehat T^R
  =\sum_{i=1}^N
  \left(V_i^{R,+}p_i+V_i^{R,-}p_i^{-1}\right),
  \qquad R=B_N,D_N.
  \label{eq:normalized-BD-trace}
\end{equation}
In Chalykh's original normalization one has
\begin{equation}
  L_{\mathrm{Ch}}^R=\kappa_R\widehat T^R,
  \qquad
  \kappa_D=\tau^{2N-2}=q^{1-N},
  \qquad
  \kappa_B=\tau^{2N-1}=q^{\frac12-N},
  \label{eq:Chalykh-BD-normalization}
\end{equation}
and therefore
\begin{align}
  \operatorname{Tr}L_{\mathrm{Ch}}^D
  &=q^{1-N}\sum_{i=1}^N
    \left(V_i^{D,+}p_i+V_i^{D,-}p_i^{-1}\right),
  \label{eq:Chalykh-D-trace}\\
  \operatorname{Tr}L_{\mathrm{Ch}}^B
  &=q^{\frac12-N}\sum_{i=1}^N
    \left(V_i^{B,+}p_i+V_i^{B,-}p_i^{-1}\right).
  \label{eq:Chalykh-B-trace}
\end{align}
These expressions agree, up to the displayed overall normalization and
the standard additive zeroth-order term, with the classical symbols of
van Diejen's first $D_N$ and $B_N$ difference Hamiltonians
\cite{vanDiejen1994,vanDiejen1995}.

Chalykh's type-$B_N$ Lax matrix is indexed by the $2N$ moving weights
$\{\pm e_i\}$.  If the vector representation, and hence a fixed
eigenvalue $1$, is required, we use the decoupled extension
\begin{equation}
  T_{\mathrm{vec}}^B=1\oplus\widehat T^B,
  \qquad
  \det(u\mathbf 1_{2N+1}-T_{\mathrm{vec}}^B)
  =(u-1)\det(u\mathbf 1_{2N}-\widehat T^B).
  \label{eq:B-vector-extension}
\end{equation}
The interacting first Hamiltonian is consequently the reduced trace
\begin{equation}
  T_1^B=\operatorname{Tr}T_{\mathrm{vec}}^B-1
  =\operatorname{Tr}\widehat T^B.
  \label{eq:B-reduced-trace}
\end{equation}
In particular, the fixed-column weight \eqref{eq:Gamma0-BC} belongs to the
$BC_N$ folding and is not inserted into the type-$B_N$ Hamiltonian.

\vspace{.1in}
For each of the classical types we have the following folding constructions.
\subsubsection{$C_N$ Model}
Let $M=2N$, then ${\bf x}=(a_1,\ldots,a_N,a_{1}^{-1},\ldots,a_{N}^{-1})$. The interaction potential reads
\begin{equation}
    V_i^{C,+}(a)=
\prod_{\substack{j=1\\j\neq i}}^N
g_q\!\left(\frac{a_i}{a_j}\right)
g_q(a_ia_j)
\cdot g_q(a_i^2)
\end{equation}
and the universal Cauchy kernel \eqref{eq:CauchyKernel} into a folded Lax matrix whose pole divisors are
\[
a_i/a_j=1,\qquad a_ia_j=1,\qquad a_i^2=1.
\]
These are precisely the reflection hypertori of the $C_N$ root system whose Weyl group is $W_C=S_N\ltimes(\mathbb{Z}/2\mathbb{Z})^N$.

\subsubsection{$D_N$ Model}
The $D_N$ root system does not have a long root and its Weyl group is $W_D=S_N\ltimes(\mathbb{Z}/2\mathbb{Z})^{N-1}$. Therefore the reflection hypertorus $a_i^2=1$ is not present and we get 
\begin{equation}
    V_i^{D,+}(a)=
\prod_{\substack{j=1\\j\neq i}}^N
g_q\!\left(\frac{a_i}{a_j}\right)
g_q(a_ia_j)
\end{equation}
The Lax matrix is \eqref{eq:Chalykh-BD-lax}, rather than the universal
Cauchy matrix at the opposite-weight entries.  Its root-supported entries
also agree with the trigonometric $D_N$ construction of
\cite{ChenHou2001}.

\subsubsection{$BC_N$ Model}
Let $M=2N+1$, then ${\bf x}=(a_1,\ldots,a_N,1,a_{1}^{-1},\ldots,a_{N}^{-1})$. The interaction potential reads
\begin{equation}
    V_i^{BC}(a)=
\prod_{\substack{j=1\\j\neq i}}^N
g_q\!\left(\frac{a_i}{a_j}\right)
g_q(a_ia_j)
\cdot g_q(a_i^2)\cdot g_q(a_i)
\end{equation}
Now the reflection hypertori include the short and the long roots simultaneously.

\subsubsection{$B_N$ Model}
At the level of the interaction potential, the $B_N$ formula is obtained by
eliminating the long-root contribution from the $BC_N$ potential:
\begin{equation}
    V_i^{B}(a)=\frac{V_i^{BC}(a)}{g_q(a_i^2)}=
\prod_{j\neq i}
g_q\!\left(\frac{a_i}{a_j}\right)
g_q(a_ia_j)\,
g_q(a_i)\,.
\end{equation}
At the matrix level this operation is not applied entrywise to the $BC_N$
Cauchy matrix.  We instead use the $2N\times2N$ matrix
\eqref{eq:Chalykh-BD-lax}; the optional vector-weight extension is
\eqref{eq:B-vector-extension}.

\subsection{The tRS Potentials}
Let us summarize our findings so far. The pole divisors of the positive potentials are
\begin{equation}
  \begin{array}{c|l}
    A_N&a_i/a_j=1,\\
    D_N&a_i/a_j=1,\ a_ia_j=1\\
    B_N&a_i/a_j=1,\ a_ia_j=1,\ a_i=1\\
    C_N&a_i/a_j=1,\ a_ia_j=1,\ a_i^2=1\\
    BC_N&a_i/a_j=1,\ a_ia_j=1,\ a_i=1,\ a_i^2=1.
  \end{array}
  \label{eq:root-hypertori}
\end{equation}
They reproduce the reflection hypertori of the corresponding root
systems. The summary table of tRS interaction potentials then looks as follows. One can see the hierarchy in the complexity of the functions based on the combinatorics of the reflection hypertori of the corresponding systems.
\begin{align}
A_{M}:\qquad
x
&=
(a_1,\ldots,a_{M}),
&
V_i^{A}(a)
&=
\prod_{j\neq i}
g_q\!\left(\frac{a_i}{a_j}\right),\notag
\\[1ex]
D_N:\qquad
x
&=
(a_1,\ldots,a_N,a_1^{-1},\ldots,a_N^{-1}),
&
V_i^{D,+}(a)
&=
\prod_{j\neq i}
g_q\!\left(\frac{a_i}{a_j}\right)
g_q(a_ia_j),\notag
\\[1ex]
C_N:\qquad
x
&=
(a_1,\ldots,a_N,a_1^{-1},\ldots,a_N^{-1}),
&
V_i^{C,+}(a)
&=
\prod_{j\neq i}
g_q\!\left(\frac{a_i}{a_j}\right)
g_q(a_ia_j)\,
g_q(a_i^2),\notag
\\[1ex]
B_N:\qquad
x
&=
(a_1,\ldots,a_N,1,a_1^{-1},\ldots,a_N^{-1}),
&
V_i^{B,+}(a)
&=
\prod_{j\neq i}
g_q\!\left(\frac{a_i}{a_j}\right)
g_q(a_ia_j)\,
g_q(a_i),\notag
\\[1ex]
BC_N:\qquad
x
&=
(a_1,\ldots,a_N,1,a_1^{-1},\ldots,a_N^{-1}),
&
V_i^{BC,+}(a)
&=
\prod_{j\neq i}
g_q\!\left(\frac{a_i}{a_j}\right)
g_q(a_ia_j)\,
g_q(a_i)\,
g_q(a_i^2).
\end{align}
For the type-$B_N$ vector representation one adjoins a decoupled fixed
coordinate with eigenvalue $1$, as in \eqref{eq:B-vector-extension}; it
does not contribute to the reduced trace \eqref{eq:B-reduced-trace}.

In the above list, types $C$ and $BC$ are obtained by direct type-$A$
folding.  The type-$B$ and type-$D$ potentials are root-orbit
specializations, while their Lax matrices require the opposite-weight
completion \eqref{eq:Chalykh-BD-lax}.  The Poisson commutativity of the
resulting classical Hamiltonians follows from the Chalykh construction
\cite{Chalykh2019} and, independently, from the classical limit of van
Diejen's commuting difference operators \cite{vanDiejen1994,vanDiejen1995}.

\section{The Quantum/Classical Duality}\label{sec:QCD}
In this section, we shall derive  Bethe ansatz equations for the quantum XXZ spin chains from the tRS energy level set relations \eqref{eq:TraceHam1} that we studied earlier. We find the canonical change of variables from the original Darboux coordinates of the tRS model $\{(a_i,p_i)\}$ to another set of coordinates in which the tRS Hamiltonians reduce merely to elementary symmetric functions. This in turn will result in diagonalization of the tRS Lax matrix. We will show that the XXZ Bethe ansatz equations arise as obstructions in the process of this diagonalization.

We refer the reader to \cite{Koroteev:2023aa} (see also \cite{Gaiotto:2013bwa}, Appendix A) for the details. A complete analysis of the tRS Lax matrix diagonalization in type $A$ was done in Theorem 4.5 (Lemma 4.6) of \cite{Koroteev:2023aa}. A similar procedure can be done for the classical groups as well, however, we refrain from reproducing it here.

\subsection{The Master Residue Identity}
First, we introduce the following tuple of variables which in Bethe ansatz literature is referred to as Bethe roots at nesting level $k$.
\begin{equation}
  \cA_k
  =\{u_{k,1},u_{k,1}^{-1},\ldots,
      u_{k,n_k},u_{k,n_k}^{-1}\}.
  \label{eq:folded-level-set}
\end{equation}
At the physical level we set $n_0=N$ and $u_{0,i}=a_i$.

Introduce the auxiliary set $\cA_{k+1}$ and consider the following functions
\begin{equation}
  \mathcal G_{R,k}(z)
  =\frac{K_R(z)}{z(1-q)}
   \prod_{x\in\cA_k}\frac{z-qx}{z-x}
   \prod_{y\in\cA_{k+1}}\frac{z-q^{-1}y}{z-y}.
  \label{eq:residue-function-g}
\end{equation}
where we introduce the following dressing factors $K_R(z)$ for each type
\begin{equation}
  \begin{array}{c|c}
    R&K_R(z)\\ \hline
    A_M&1\\[1mm]   
    D_N&g_q(z^2)^{-1}\\[1mm]
    B_N&g_q(z)g_q(z^2)^{-1}\\[1mm]
    C_N&1\\[1mm]
    BC_N&g_q(z).
  \end{array}
  \label{eq:classical-corrections-g}
\end{equation}
The purpose of $K_R(z)$ is to cancel the poles and zeros in the potential function which do not exist in the corresponding root system. 

The next step is to apply the Cauchy's theorem to one-form $\mathcal G_{R,k}(z) dz$ on $\mathbb{P}^1$. The coordinate residues are
\begin{equation}
  \text{Res}_{z=x}\mathcal G_{R,k}(z)
  =\Delta_R(x\mid\cA_k)
   \prod_{y\in\cA_{k+1}}
       \frac{x-q^{-1}y}{x-y}.
  \label{eq:coordinate-residue-g}
\end{equation}
Here the universal folded Lax weight is
\begin{equation}
  \Delta_R(x\mid\cA_k)
  =K_R(x)
   \prod_{\substack{x'\in\cA_k\\x'\neq x}}
      g_q\!\left(\frac{x}{x'}\right).
  \label{eq:folded-delta-g}
\end{equation}
The product over $x'$ contains $g_q(x^2)$ from the mirror
$x'=x^{-1}$.  The factor $K_R$ cancels this contact for $D_N,B_N$ and
retains it for $C_N,BC_N$; it also supplies $g_q(x)$ in the types with
a short root.  Therefore
\begin{equation}
  \Delta_R(a_i\mid\cA_0)=V_i^{R,+},
  \qquad
  \Delta_R(a_i^{-1}\mid\cA_0)=V_i^{R,-}.
  \label{eq:delta-equals-potentials}
\end{equation}

The auxiliary residues are
\begin{align}
  \text{Res}_{z=y}\mathcal G_{R,k}(z)
  ={}&-\frac{1}{q}K_R(y)
   \prod_{x\in\cA_k}\frac{y-qx}{y-x}
   \prod_{\substack{y'\in\cA_{k+1}\\y'\neq y}}
       \frac{y-q^{-1}y'}{y-y'}.
  \label{eq:auxiliary-residue-g}
\end{align}

The following lemma is a direct consequence of the residue theorem.

\begin{Lem}
\label{lem:cauchy-identity}
Assume that the elements of $\cA_k$ and $\cA_{k+1}$ are pairwise
distinct and that no denominator in \eqref{eq:residue-function-g}
vanishes.  Let $E_{R,k}$ be minus the sum of the residues of the one-form
$\mathcal G_{R,k}(z)\,dz$ at zero, infinity, and the additional poles
of $K_R$.  Then
\begin{align}
  &\sum_{x\in\cA_k}\Delta_R(x\mid\cA_k)
   \prod_{y\in\cA_{k+1}}
       \frac{x-q^{-1}y}{x-y}
  \notag\\
  &\quad=E_{R,k}
   +\frac1q\sum_{y\in\cA_{k+1}}K_R(y)
   \prod_{x\in\cA_k}\frac{y-qx}{y-x}
   \prod_{\substack{y'\in\cA_{k+1}\\y'\neq y}}
       \frac{y-q^{-1}y'}{y-y'}.
  \label{eq:master-residue-g}
\end{align}
\end{Lem}

\begin{proof}
At a coordinate pole $z=x\in\cA_k$ and $z=y\in\cA_{k+1}$, direct evaluations give
equations \eqref{eq:coordinate-residue-g} and \eqref{eq:auxiliary-residue-g}.
In order to make the remaining residues explicit, set
\begin{equation}
  m_k=|\cA_k|,
  \qquad
  d_k=m_k-m_{k+1},
  \qquad
  \Phi_k(\rho)=
  \prod_{x\in\cA_k}\frac{\rho-qx}{\rho-x}
  \prod_{y\in\cA_{k+1}}\frac{\rho-q^{-1}y}{\rho-y}.
  \label{eq:exceptional-product}
\end{equation}
Since every coordinate belongs to $\mathbb C^\times$, the residue at
the origin reads
\begin{equation}
  \text{Res}_{z=0}\mathcal G_{R,k}(z)\,dz
  =\frac{K_R(0)q^{d_k}}{1-q}.
  \label{eq:residue-zero}
\end{equation}
Indeed, each factor indexed by $\cA_k$ tends to $q$, while each factor
indexed by $\cA_{k+1}$ tends to $q^{-1}$.  Since $K_R(\infty)=1$ for
all five types,
\begin{equation}
  \text{Res}_{z=\infty}\mathcal G_{R,k}(z)\,dz
  =-\frac{1}{1-q}.
  \label{eq:residue-infinity}
\end{equation}
If $\rho$ is an additional simple pole of $K_R$, then
\begin{equation}
  \text{Res}_{z=\rho}\mathcal G_{R,k}(z)\,dz
  =\frac{\text{Res}_{z=\rho}K_R(z)}{\rho(1-q)}\Phi_k(\rho).
  \label{eq:residue-extra-pole}
\end{equation}
\end{proof}

To summarize, the type-dependent data goes as follows
\begin{equation}
\begin{array}{c|c|c|c}
R&K_R(0)&\mathfrak P_R&\text{Res}_{z=\rho}K_R(z)\\ \hline
A_M&1&\varnothing&-\\[1mm]
B_N&1&\{\rho:\rho^2=q\}&(1-q)/2\\[1mm]
C_N&1&\varnothing&-\\[1mm]
D_N&q^{-1}&\{\rho:\rho^2=q\}&(q-1)/(2\rho)\\[1mm]
BC_N&q&\{1\}&1-q.
\end{array}
\label{eq:exceptional-residue-table}
\end{equation}
Here $\mathfrak P_R$ is the set of additional finite poles.  Substituting
\eqref{eq:residue-zero}--\eqref{eq:exceptional-residue-table} into the
definition of $E_{R,k}$ gives
\begin{equation}
\begin{aligned}
 E_{A,k}&=\frac{1-q^{d_k}}{1-q},\\
 E_{B,k}&=\frac{1-q^{d_k}}{1-q}
   -\frac12\sum_{\rho^2=q}\frac{\Phi_k(\rho)}{\rho},\\
 E_{C,k}&=\frac{1-q^{d_k}}{1-q},\\
 E_{D,k}&=\frac{1-q^{d_k-1}}{1-q}
   +\frac{1}{2q}\sum_{\rho^2=q}\Phi_k(\rho),\\
 E_{BC,k}&=\frac{1-q^{d_k+1}}{1-q}-\Phi_k(1).
\end{aligned}
\label{eq:external-residues-by-type}
\end{equation}
The sum of all residues of a meromorphic one-form on $\mathbb P^1$ is
zero.  Moving the auxiliary and exceptional residues to the opposite
side gives \eqref{eq:master-residue-g}.

\subsection{Momentum Parameterization and Induction}
We now use \eqref{eq:master-residue-g} to reduce the trace of the tRS
Lax matrix one nesting level at a time.  Introduce
$\zeta_0,\ldots,\zeta_L\in\mathbb C^\times$ and write
\begin{equation}
  T_{1,k}^R
  =U_{R,k}+\sum_{x\in\cA_k}
    \Delta_R(x\mid\cA_k)p_x^{(k)}.
  \label{eq:level-k-trace}
\end{equation}
Here $U_{R,k}$ denotes the fixed-column contribution.  We use the reduced
trace for type $B_N$, so its optional decoupled zero weight is not included
in $U_{B,k}$.  At the physical level $k=0$, this is the trace written in
the previous section.  Parameterize its momenta by
\begin{equation}
  p_x^{(k)}
  =\zeta_k
   \prod_{y\in\cA_{k+1}}
       \frac{x-q^{-1}y}{x-y},
  \qquad x\in\cA_k.
  \label{eq:right-momentum-g}
\end{equation}
On the folded physical locus,
\begin{equation}
  p_{a_i}^{(0)}=p_i,
  \qquad
  p_{a_i^{-1}}^{(0)}=p_i^{-1}.
  \label{eq:folded-momenta-g}
\end{equation}

\begin{Lem}
\label{lem:inductive-cauchy}
Under the parameterization \eqref{eq:right-momentum-g}, the level-$k$
trace has the following recursive form
\begin{equation}
  T_{1,k}^R
  =U_{R,k}+\zeta_kE_{R,k}
   +\sum_{y\in\cA_{k+1}}
    \Delta_R(y\mid\cA_{k+1})
    p_{y,\mathrm{ind}}^{(k+1)},
  \label{eq:trace-recursion-g}
\end{equation}
where
\begin{equation}
  p_{y,\mathrm{ind}}^{(k+1)}
  =\frac{\zeta_k}{q}
   \prod_{x\in\cA_k}\frac{y-qx}{y-x}
   \prod_{\substack{y'\in\cA_{k+1}\\y'\neq y}}
       \frac{y-q^{-1}y'}{y-qy'}.
  \label{eq:induced-momentum-g}
\end{equation}
Let
\begin{equation}
  p_{y,\mathrm{fwd}}^{(k+1)}
  =\zeta_{k+1}\frac{K_R(y^{-1})}{K_R(y)}
   \prod_{z\in\cA_{k+2}}
       \frac{y-q^{-1}z}{y-z}\,.
  \label{eq:forward-momentum-g}
\end{equation}
Then the reduced trace can be iterated at level $k+1$ if and only if,
for every $y\in\cA_{k+1}$ the following XXZ Bethe ansatz equations are satisfied
\begin{align}
  1={}&\frac{\zeta_k}{q\zeta_{k+1}}\frac{K_R(y)}{K_R(y^{-1})}
   \prod_{x\in\cA_k}\frac{y-qx}{y-x}
   \prod_{\substack{y'\in\cA_{k+1}\\y'\neq y}}
       \frac{y-q^{-1}y'}{y-qy'}
   \prod_{z\in\cA_{k+2}}
       \frac{y-z}{y-q^{-1}z}\,.
  \label{eq:raw-bethe-from-residue}
\end{align}
\end{Lem}

\begin{proof}
Multiplying \eqref{eq:master-residue-g} by $\zeta_k$ gives
\begin{align}
  T_{1,k}^R
  ={}&U_{R,k}+\zeta_k E_{R,k}+\frac{\zeta_k}{q}
   \sum_{y\in\cA_{k+1}}K_R(y)
   \prod_{x\in\cA_k}\frac{y-qx}{y-x}
   \prod_{\substack{y'\in\cA_{k+1}\\y'\neq y}}
       \frac{y-q^{-1}y'}{y-y'}.
  \label{eq:master-residue-g2}
\end{align}
Factor the level-$(k+1)$ Lax weight as
\begin{equation}
  \Delta_R(y\mid\cA_{k+1})
  =K_R(y)
   \prod_{\substack{y'\in\cA_{k+1}\\y'\neq y}}
       \frac{y-qy'}{y-y'}.
  \label{eq:delta-next-level}
\end{equation}
The quotient of the summand in \eqref{eq:master-residue-g2} by
\eqref{eq:delta-next-level} is exactly
\eqref{eq:induced-momentum-g}, proving \eqref{eq:trace-recursion-g}.
The first two terms of the recursion give the separated eigenvalue;
the last term is a tRS trace on $\cA_{k+1}$.

The folded $y$ and $y^{-1}$ columns have opposite orientation, which
fixes the factor $\rho_R(y)=K_R(y^{-1})/K_R(y)$ in the forward
parameterization.  Finally,
$p_{y,\mathrm{ind}}^{(k+1)}=p_{y,\mathrm{fwd}}^{(k+1)}$ is equivalent,
after division by the forward product, to
\eqref{eq:raw-bethe-from-residue}.  This proves both necessity and
sufficiency for the next inductive step.
\end{proof}


For the direct $BC_N$ folding, define the level-$k$ fixed-column weight
by
\begin{equation}
  \Gamma_{0,k}^{BC}
  =\prod_{j=1}^{n_k}g_q(u_{k,j})g_q(u_{k,j}^{-1}).
  \label{eq:level-k-BC-fixed-column}
\end{equation}
Thus $\Gamma_{0,0}^{BC}=\Gamma_0^{BC}$ on the physical locus.  The
additive term in \eqref{eq:level-k-trace} is
\begin{equation}
\begin{array}{c|c}
R&U_{R,k}\\ \hline
A_M&0\\
B_N&0\\
C_N&0\\
D_N&0\\
BC_N&\Gamma_{0,k}^{BC}
\end{array}
\label{eq:fixed-column-table}
\end{equation}
If one works instead with the vector extension
$T_{\mathrm{vec}}^B=1\oplus\widehat T^B$, then
$U_{B,k}^{\mathrm{vec}}=1$ at every level.  These unit contributions
cancel from the recursion below.

When \eqref{eq:raw-bethe-from-residue} holds, the induced and forward
momenta coincide.  Since
\begin{equation}
  T_{1,k+1}^R
  =U_{R,k+1}
   +\sum_{y\in\cA_{k+1}}
     \Delta_R(y\mid\cA_{k+1})p_{y,\mathrm{fwd}}^{(k+1)},
  \label{eq:next-level-trace}
\end{equation}
equation \eqref{eq:trace-recursion-g} becomes
\begin{equation}
  T_{1,k}^R=\tau_k^{(R)}+T_{1,k+1}^R,
  \qquad
  \tau_k^{(R)}
  =U_{R,k}+\zeta_kE_{R,k}-U_{R,k+1}.
  \label{eq:isolated-eigenvalue}
\end{equation}
Thus, explicitly,
\begin{equation}
\begin{aligned}
 \tau_k^{(A)}&=\zeta_kE_{A,k},\\
 \tau_k^{(B)}&=\zeta_kE_{B,k},\\
 \tau_k^{(C)}&=\zeta_kE_{C,k},\\
 \tau_k^{(D)}&=\zeta_kE_{D,k},\\
 \tau_k^{(BC)}&=\Gamma_{0,k}^{BC}+\zeta_kE_{BC,k}
   -\Gamma_{0,k+1}^{BC}.
\end{aligned}
\label{eq:isolated-eigenvalues-by-type}
\end{equation}
For the optional type-$B_N$ vector extension, the two unit terms cancel;
for type $BC_N$, the dynamical fixed-column weights remain part of the
isolated eigenvalue.  Ordering the separated eigenvalues compatibly with the
nesting direction, and allowing a root-independent normalization power
$q^{\nu_k}$, we write the XXZ twist ratio as
\begin{equation}
  \chi_k=q^{\nu_k}
  \frac{\tau_{k-1}^{(R)}}{\tau_k^{(R)}}.
  \label{eq:chi-from-isolated-eigenvalues}
\end{equation}

The additive data $U_{R,k}$ have already entered the above expression
$\chi_k$ through \eqref{eq:isolated-eigenvalue} and
\eqref{eq:chi-from-isolated-eigenvalues}.  The remaining
root-dependent boundary data arise from the ratio of the two Weyl
orientations of $K_R$.  No additional boundary function is needed:
using $f_q(z)=g_q(z)/g_q(z^{-1})$, the required quotients are
\begin{align}
 \frac{K_A(u)}{K_A(u^{-1})}
 &=1,
 \notag\\
 \frac{K_B(u)}{K_B(u^{-1})}
 &=\frac{g_q(u)g_q(u^2)^{-1}}
         {g_q(u^{-1})g_q(u^{-2})^{-1}}
   =\frac{f_q(u)}{f_q(u^2)},
 \notag\\
 \frac{K_C(u)}{K_C(u^{-1})}
 &=1,
 \notag\\
 \frac{K_D(u)}{K_D(u^{-1})}
 &=\frac{g_q(u^2)^{-1}}{g_q(u^{-2})^{-1}}
   =f_q(u^2)^{-1},
 \notag\\
 \frac{K_{BC}(u)}{K_{BC}(u^{-1})}
 &=\frac{g_q(u)}{g_q(u^{-1})}
   =f_q(u).
 \label{eq:k-ratios-explicit}
\end{align}
For the folded cases, introduce the universal compatibility product
\begin{align}
 \mathcal S_k(y)={}&
   \prod_{x\in\cA_k}\frac{y-qx}{y-x}
   \prod_{\substack{y'\in\cA_{k+1}\\y'\neq y}}
       \frac{y-q^{-1}y'}{y-qy'}
   \prod_{z\in\cA_{k+2}}
       \frac{y-z}{y-q^{-1}z}.
 \label{eq:universal-compatibility-product}
\end{align}
Then \eqref{eq:raw-bethe-from-residue} reads explicitly
\begin{equation}
\begin{array}{c|c}
R&\text{compatibility equation at }y\in\cA_{k+1}\\ \hline
B_N&1=\dfrac{\zeta_k}{q\zeta_{k+1}}
      \dfrac{f_q(y)}{f_q(y^2)}\mathcal S_k(y)\\[2mm]
C_N&1=\dfrac{\zeta_k}{q\zeta_{k+1}}\mathcal S_k(y)\\[2mm]
D_N&1=\dfrac{\zeta_k}{q\zeta_{k+1}}
      f_q(y^2)^{-1}\mathcal S_k(y)\\[2mm]
BC_N&1=\dfrac{\zeta_k}{q\zeta_{k+1}}
      f_q(y)\mathcal S_k(y).
\end{array}
\label{eq:raw-compatibility-by-type}
\end{equation}
Thus $U_{R,k}$ and $K_R$ have distinct roles: $U_{R,k}$ contributes
to the constant part of the spectral ratio $\chi_k$, whereas the quotient
$K_R(u)/K_R(u^{-1})$ supplies the root-dependent boundary factor.

Relabel the levels so that the compatibility condition is imposed on
the roots at node $k$.  We use the variables of
\cite{Koroteev:2025tka},
\begin{equation}
  u_{k,j}=q^{(k-1)/2}s_{k,j},
  \qquad
  \cA_k=\{u_{k,j},u_{k,j}^{-1}\}_{j=1}^{p_k},
  \label{eq:centered-roots}
\end{equation}
and define the open-chain polynomials
\begin{equation}
  Q_k(z)
  =\prod_{j=1}^{p_k}
   (z-s_{k,j})
   \left(\frac{q^{1-k}}{z}-s_{k,j}\right).
  \label{eq:open-q-function}
\end{equation}

A direct computation leads to the following.
\begin{Prop}
\label{prop:open-q-form}
The compatibility equations \eqref{eq:raw-bethe-from-residue} are equivalent to
\begin{equation}
  -1=\mathfrak x_k^R(s_{k,i})
  \frac{
   Q_{k-1}(qs_{k,i})
   Q_k(q^{-1}s_{k,i})
   Q_{k+1}(s_{k,i})}
  {
   Q_{k-1}(s_{k,i})
   Q_k(qs_{k,i})
   Q_{k+1}(q^{-1}s_{k,i})},
  \label{eq:open-q-bethe}
\end{equation}
where
\begin{equation}
  \mathfrak x_k^R(s)
  =q^{1-p_k+p_{k-1}}\chi_k
   \frac{K_R\!\left(q^{(k-1)/2}s\right)}
        {K_R\!\left(q^{-(k-1)/2}s^{-1}\right)}.
  \label{eq:full-open-twist}
\end{equation}
Equivalently,
\begin{align}
 -1={}&\mathfrak x_k^R(s_{k,i})q^{p_k-p_{k-1}}
 \prod_{j=1}^{p_{k-1}}
 \frac{qs_{k,i}-s_{k-1,j}}{s_{k,i}-s_{k-1,j}}
 \frac{q^{k-1}s_{k,i}s_{k-1,j}-1}
      {q^{k-2}s_{k,i}s_{k-1,j}-1}
 \notag\\
 &\times\prod_{j=1}^{p_k}
 \frac{s_{k,i}-qs_{k,j}}{qs_{k,i}-s_{k,j}}
 \frac{q^{k-2}s_{k,i}s_{k,j}-1}
      {q^ks_{k,i}s_{k,j}-1}
 \notag\\
 &\times\prod_{j=1}^{p_{k+1}}
 \frac{s_{k,i}-s_{k+1,j}}{s_{k,i}-qs_{k+1,j}}
 \frac{q^ks_{k,i}s_{k+1,j}-1}
      {q^{k-1}s_{k,i}s_{k+1,j}-1}.
 \label{eq:full-open-product}
\end{align}
\end{Prop}

\begin{proof}
The product over the roots at node $k$ in
\eqref{eq:full-open-product} includes $j=i$. Using $u_{k,i}^2=q^{k-1}s_{k,i}^2$,
one finds
\begin{equation}
 \frac{s_{k,i}-qs_{k,i}}{qs_{k,i}-s_{k,i}}
  \frac{q^{k-2}s_{k,i}^2-1}{q^ks_{k,i}^2-1}
=-\frac1q
 \frac{u_{k,i}^2-q}{qu_{k,i}^2-1}
 =-q^{-1}f_q(u_{k,i}^2).
 \label{eq:self-contact}
\end{equation}
\end{proof}

The following Lemma also follows by a direct computation
\begin{Lem}
\label{lem:self-contact}
Define the contact-free product
\begin{align}
 \mathscr P_{k,i}^{\circ}={}&
 \prod_{j=1}^{p_{k-1}}
 \frac{qs_{k,i}-s_{k-1,j}}{s_{k,i}-s_{k-1,j}}
 \frac{q^{k-1}s_{k,i}s_{k-1,j}-1}
      {q^{k-2}s_{k,i}s_{k-1,j}-1}
 \notag\\
 &\times\prod_{\substack{j=1\\j\neq i}}^{p_k}
 \frac{s_{k,i}-qs_{k,j}}{qs_{k,i}-s_{k,j}}
 \frac{q^{k-2}s_{k,i}s_{k,j}-1}
      {q^ks_{k,i}s_{k,j}-1}
 \notag\\
 &\times\prod_{j=1}^{p_{k+1}}
 \frac{s_{k,i}-s_{k+1,j}}{s_{k,i}-qs_{k+1,j}}
 \frac{q^ks_{k,i}s_{k+1,j}-1}
      {q^{k-1}s_{k,i}s_{k+1,j}-1}.
 \label{eq:contact-free-product}
\end{align}
Then Bethe equations \eqref{eq:full-open-product} are equivalent to
\begin{equation}
  1=\chi_k f_q(u_{k,i}^2)
    \frac{K_R(u_{k,i})}{K_R(u_{k,i}^{-1})}
    \mathscr P_{k,i}^{\circ}.
  \label{eq:contact-free-master}
\end{equation}
\end{Lem}

\subsection{Boundary Factors}
The correction factors in \eqref{eq:classical-corrections-g} determine
the full boundary twists. A direct evaluation yields
\begin{equation}
\begin{array}{c|c|c|c}
R&K_R(u)&K_R(u)/K_R(u^{-1})
 &f_q(u^2)K_R(u)/K_R(u^{-1})\\ \hline
D_N&g_q(u^2)^{-1}&f_q(u^2)^{-1}&1\\[1mm]
B_N&g_q(u)g_q(u^2)^{-1}
 &f_q(u)f_q(u^2)^{-1}&f_q(u)\\[1mm]
C_N&1&1&f_q(u^2)\\[1mm]
BC_N&g_q(u)&f_q(u)&f_q(u)f_q(u^2).
\end{array}
\label{eq:boundary-table}
\end{equation}
The contact $j=i$ term contributes the universal factor of $f_q(u^2)$. The
correction term $K_D$ removes it in type $D$ and $K_B$ removes it while
leaving the short-root factor $f_q(u)$. The type $C$ contribution retains the long-root
contact term, while the $BC$-type retains both the short and long root factors.

After the contact term in \eqref{eq:self-contact} is singled out, the dual
boundary functions in the centered $s$-variables read
\begin{equation}
\begin{aligned}
 f_A^{(k)}(s)
 &=1,\\
 f_D^{(k)}(s)
 &=1,\\
 f_B^{(k)}(s)
 &=f_q\!\left(q^{(k-1)/2}s\right)
  =\frac{q^{(k-1)/2}s-q}{q^{(k+1)/2}s-1},\\
 f_C^{(k)}(s)
 &=f_q\!\left(q^{k-1}s^2\right)
  =\frac{q^{k-1}s^2-q}{q^ks^2-1},\\
 f_{BC}^{(k)}(s)
 &=f_B^{(k)}(s)f_C^{(k)}(s)
  =\frac{(q^{(k-1)/2}s-q)(q^{k-1}s^2-q)}
         {(q^{(k+1)/2}s-1)(q^ks^2-1)}.
\end{aligned}
\label{eq:explicit-boundary-functions-s}
\end{equation}
Consequently, the five contact-free XXZ Bethe equations are
\begin{equation}
  \begin{aligned}
  A_M:\quad&1=\chi_k^Aq^{p_{k+1}-p_k}
    \mathscr P_{k,i}^{A,\circ},\\
  D_N:\quad&1=\chi_k\mathscr P_{k,i}^{\circ},\\
  B_N:\quad&1=\chi_k
    \frac{q^{(k-1)/2}s_{k,i}-q}{q^{(k+1)/2}s_{k,i}-1}
    \mathscr P_{k,i}^{\circ},\\
  C_N:\quad&1=\chi_k
    \frac{q^{k-1}s_{k,i}^2-q}{q^ks_{k,i}^2-1}
    \mathscr P_{k,i}^{\circ},\\
  BC_N:\quad&1=\chi_k
    \frac{(q^{(k-1)/2}s_{k,i}-q)(q^{k-1}s_{k,i}^2-q)}
         {(q^{(k+1)/2}s_{k,i}-1)(q^ks_{k,i}^2-1)}
    \mathscr P_{k,i}^{\circ}.
  \end{aligned}
  \label{eq:folded-type-equations}
\end{equation}

\subsection{Summary of the Quantum/Classical Duality}
We would like to summarize our results for each type below.
\subsubsection{Type $A_N$}
In type $A$ there is no Weyl reflection identifying $a_i$ with $a_i^{-1}$ and no mirror Bethe root.  Thus
\begin{equation}
 X_i=a_i,\qquad p_{X_i}=p_i,\qquad
 \cA_k^A=\{s_{k,1},\ldots,s_{k,p_k}\}.
 \label{eq:type-a-folding}
\end{equation}
Put the following for the type $A$ Baxter polynomials
\begin{equation}
  Q_k^A(z)=\prod_{j=1}^{p_k}(z-s_{k,j}).
  \label{eq:a-q-function}
\end{equation}
The same residue induction gives
\begin{equation}
  -1=\chi_k^A
  \frac{Q_{k-1}^A(qs_{k,i})Q_k^A(q^{-1}s_{k,i})
        Q_{k+1}^A(s_{k,i})}
       {Q_{k-1}^A(s_{k,i})Q_k^A(qs_{k,i})
        Q_{k+1}^A(q^{-1}s_{k,i})}.
  \label{eq:type-a-q-equation}
\end{equation}
Equivalently,
\begin{align}
 1={}\chi_k^Aq^{p_{k+1}-p_k}
  \mathscr P_{k,i}^{A,\circ}.
 \label{eq:type-a-contact-free}
\end{align}

\subsubsection{Type \texorpdfstring{$B_N$}{B(N)}}
As in the preceding section, the type-$B_N$ Lax matrix is the
normalized Chalykh matrix indexed by the
$2N$ moving weights.  It is not the $(2N+1)\times(2N+1)$ direct
$BC_N$ Cauchy folding.  With
$\mathcal I=\{1,\ldots,N,\bar1,\ldots,\bar N\}$, we write
\begin{equation}
 \widehat T^B=(\widehat T_{AB}^B)_{A,B\in\mathcal I},
 \qquad
 T_{\mathrm{vec}}^B=1\oplus\widehat T^B,
 \label{eq:type-b-lax-matrix}
\end{equation}
where the entries of $\widehat T^B$ are given by
\eqref{eq:Chalykh-BD-lax}.  The second matrix in
\eqref{eq:type-b-lax-matrix} is only the optional vector-weight
extension of \eqref{eq:B-vector-extension}.

Write the prescribed spectra as
\begin{align}
 \operatorname{Spec}(\widehat T^B)
 &=\left\{
   \xi_1^{(B)},\ldots,\xi_N^{(B)},
   (\xi_N^{(B)})^{-1},\ldots,(\xi_1^{(B)})^{-1}
  \right\},
 \notag\\
 \operatorname{Spec}(T_{\mathrm{vec}}^B)
 &=\left\{
   \xi_1^{(B)},\ldots,\xi_N^{(B)},1,
   (\xi_N^{(B)})^{-1},\ldots,(\xi_1^{(B)})^{-1}
  \right\},\\
 T_1^B=\operatorname{Tr}\widehat T^B
 &=\operatorname{Tr}T_{\mathrm{vec}}^B-1
 =\sum_{\alpha=1}^N
  \left(\xi_\alpha^{(B)}+(\xi_\alpha^{(B)})^{-1}\right).
 \label{eq:type-b-spectrum}
\end{align}

The residue reduction fixes the trace of the block eliminated at level
$k$ to be
\begin{equation}
 \tau_k^{(B)}
 =\zeta_k\left[
   \frac{1-q^{d_k}}{1-q}
   -\frac12\sum_{\rho^2=q}\frac{\Phi_k(\rho)}{\rho}
  \right].
 \label{eq:type-b-block-trace}
\end{equation}
Therefore the contact-free Bethe equations read
\begin{align}
 1={}q^{\nu_k}\frac{\tau_{k-1}^{(B)}}{\tau_k^{(B)}}
 \frac{q^{(k-1)/2}s_{k,i}-q}{q^{(k+1)/2}s_{k,i}-1}
\mathscr P_{k,i}^{\circ}, \quad k=1,\dots, N\,.
\end{align}

For type $B_N$ we use the optional decoupled vector-weight extension,
whereas type $BC_N$ has the genuine fixed point of the direct Cauchy
folding.  In both cases we order the odd spectrum as
\begin{equation}
\begin{aligned}
  \bigl(\zeta_1^{(R)},\ldots,\zeta_{2N+1}^{(R)}\bigr)
  &={}
  \bigl(\xi_1^{(R)},\ldots,\xi_N^{(R)},1,\\
  &\hspace{3em}
         (\xi_N^{(R)})^{-1},\ldots,(\xi_1^{(R)})^{-1}\bigr),
  \qquad R=B,BC.
\end{aligned}
  \label{eq:odd-spectrum}
\end{equation}
hence the Bethe equations read
\begin{align}
  1=
  \frac{\zeta_{2N-k+2}^{(B)}}{\zeta_{2N-k+1}^{(B)}}
  \frac{q^{(k-1)/2}s_{k,i}-q}
       {q^{(k+1)/2}s_{k,i}-1}
  \mathscr P_{k,i}^{\circ}.
  \label{eq:b-bethe}
\end{align}

\subsubsection{Type \texorpdfstring{$BC_N$}{BC(N)}}
The spectral fold has the same fixed point and inverse pairing:
\begin{align}
 \operatorname{Spec}(T^{BC})
 &=\left\{
   \xi_1^{(BC)},\ldots,\xi_N^{(BC)},1,
   (\xi_N^{(BC)})^{-1},\ldots,(\xi_1^{(BC)})^{-1}
  \right\},\\
 \operatorname{Tr}T^{BC}
 &=1+\sum_{\alpha=1}^N
   \left(\xi_\alpha^{(BC)}+(\xi_\alpha^{(BC)})^{-1}\right).
 \label{eq:type-bc-reflected-trace}
\end{align}
The residue at the additional pole $z=1$ gives the eliminated block
trace and the positive-chamber spectral ratio
\begin{align}
 \tau_k^{(BC)}
 &=\Gamma_{0,k}^{BC}+\zeta_k\left[
    \frac{1-q^{d_k+1}}{1-q}-\Phi_k(1)
   \right]-\Gamma_{0,k+1}^{BC},
 \label{eq:type-bc-block-trace}\\
 \chi_k
 &=q^{\nu_k}\frac{\xi_k^{(BC)}}{\xi_{k+1}^{(BC)}}
  =q^{\nu_k}\frac{\tau_{k-1}^{(BC)}}{\tau_k^{(BC)}}.
 \label{eq:type-bc-spectral-ratio}
\end{align}
Therefore the contact-free Bethe equations read
\begin{align}
  1=
  \frac{\zeta_{2N-k+2}^{(BC)}}{\zeta_{2N-k+1}^{(BC)}}
  \frac{
   \bigl(q^{(k-1)/2}s_{k,i}-q\bigr)
   \bigl(q^{k-1}s_{k,i}^{2}-q\bigr)}
  {\bigl(q^{(k+1)/2}s_{k,i}-1\bigr)
   \bigl(q^ks_{k,i}^{2}-1\bigr)}
  \mathscr P_{k,i}^{\circ}.
  \label{eq:bc-bethe}
\end{align}

\subsubsection{Type \texorpdfstring{$C_N$}{C(N)}}
For types $C_N$ and $D_N$, the residual Weyl symmetry fixes the ordering of the diagonal twist in the following way
\begin{equation}
  \bigl(\zeta_1^{(R)},\ldots,\zeta_{2N}^{(R)}\bigr)
  =\bigl(\xi_1^{(R)},\ldots,\xi_N^{(R)},
         (\xi_N^{(R)})^{-1},\ldots,(\xi_1^{(R)})^{-1}\bigr),
  \qquad R=C,D.
  \label{eq:even-spectrum}
\end{equation}

Write
\begin{align}
\operatorname{Tr}T^C
 =\sum_{\alpha=1}^N
  \left(\xi_\alpha^{(C)}+(\xi_\alpha^{(C)})^{-1}\right),
 \qquad \tau_k^{(C)}
 =\zeta_k\frac{1-q^{d_k}}{1-q},
 \label{eq:type-c-block-trace}
\end{align}

The Bethe equations then read
\begin{align}
 1={}&q^{\nu_k}
  \frac{\zeta_{2N-k+1}^{(C)}}{\zeta_{2N-k}^{(C)}}
  \frac{q^{k-1}s_{k,i}^{2}-q}
       {q^ks_{k,i}^{2}-1}
  \mathscr P_{k,i}^{\circ}, \quad k=1,\dots, 2N-1\,.
 \label{eq:type-c-contact-free-bethe}
\end{align}

\subsubsection{Type \texorpdfstring{$D_N$}{D(N)}}
The $K_D$ quotient cancels the reflected self-contact, and therefore
no one-body boundary factor remains:
\begin{equation}
  1=
  \frac{\zeta_{2N-k+1}^{(D)}}{\zeta_{2N-k}^{(D)}}
  \mathscr P_{k,i}^{\circ}.
  \label{eq:d-bethe}
\end{equation}
for $k=1,\ldots,2N-1$.

\subsection{Main Theorem}
Earlier in this section we derived the trace recursion for the tRS Lax
matrices and showed that the compatibility of the induced and forward
momentum parameterizations is equivalent to the XXZ Bethe Ansatz
equations.  It remains to lift this trace recursion to the full
characteristic polynomial.

For the folded types $R=B,C,D,BC$, a generic spectral parameter belongs
to the two-element Weyl orbit
\begin{equation}
  \mathcal O\bigl(\xi_{k+1}^{(R)}\bigr)
  =
  \left\{
    \xi_{k+1}^{(R)},
    \bigl(\xi_{k+1}^{(R)}\bigr)^{-1}
  \right\}.
  \label{eq:folded-spectral-orbit}
\end{equation}
Accordingly, one step of the folded induction lowers the dimension of
the Lax matrix by two. The quantity $\tau_k^{(R)}$ appearing in the
trace recursion is not an individual eigenvalue, rather it is the trace of the orbit removed at that step
\begin{equation}
  \tau_k^{(R)} = \xi_{k+1}^{(R)}+\bigl(\xi_{k+1}^{(R)}\bigr)^{-1}.
  \label{eq:folded-orbit-trace}
\end{equation}
This relation is Weyl invariant and does not distinguish
$\xi_{k+1}^{(R)}$ from its inverse.  A choice of positive Weyl chamber
selects an ordered representative when the individual XXZ twists are
needed.

The following Lemma upgrades the trace recursion to a recursion for the
characteristic polynomial without introducing an intermediate
half-step.

\begin{Lem}
\label{lem:rank-two-characteristic-reduction}
Let $R=B,C,D$, or $BC$, and let
\begin{equation}
  T_k^R\in\operatorname{End}(V_k),
  \qquad
  T_{k+1}^{R,\mathrm{ind}}
  \in\operatorname{End}(V_{k+1}),
  \qquad
  \dim V_{k+1}=\dim V_k-2,
\end{equation}
be two consecutive normalized Lax matrices in the folded induction.
Assume that the rank-two Cauchy reduction gives a full-column-rank map
\begin{equation}
  G_k^R\colon V_{k+1}\longrightarrow V_k
\end{equation}
satisfying
\begin{equation}
  T_k^R G_k^R
  =
  G_k^R T_{k+1}^{R,\mathrm{ind}}.
  \label{eq:rank-two-intertwining}
\end{equation}
Suppose, moreover, that the reciprocal normalization is used, so that
\begin{equation}
  \frac{\det T_k^R}
       {\det T_{k+1}^{R,\mathrm{ind}}}
  =1.
  \label{eq:rank-two-determinant-ratio}
\end{equation}
Then
\begin{align}
  \det\left(u\mathbf 1-T_k^R\right)
  ={}&
  \left(u-\xi_{k+1}^{(R)}\right)
  \left(
    u-\bigl(\xi_{k+1}^{(R)}\bigr)^{-1}
  \right)
  \det\left(
    u\mathbf 1-T_{k+1}^{R,\mathrm{ind}}
  \right),
  \label{eq:rank-two-reciprocal-recurrence}
\end{align}
provided that the trace difference is given by
\begin{equation}
  \operatorname{Tr}T_k^R
  -
  \operatorname{Tr}T_{k+1}^{R,\mathrm{ind}}
  =\tau_k^{(R)}
  =\xi_{k+1}^{(R)}
   +\bigl(\xi_{k+1}^{(R)}\bigr)^{-1}.
  \label{eq:rank-two-trace-difference}
\end{equation}
On the Bethe locus, one may replace
$T_{k+1}^{R,\mathrm{ind}}$ in
\eqref{eq:rank-two-reciprocal-recurrence} by the forward matrix
$T_{k+1}^R$.
\end{Lem}

\begin{proof}
Since $G_k^R$ has full column rank,
\eqref{eq:rank-two-intertwining} identifies $V_{k+1}$ with a
$T_k^R$-invariant subspace of $V_k$.  In a basis adapted to this
subspace,
\begin{equation}
  \left(S_k^R\right)^{-1}T_k^R S_k^R
  =
  \begin{pmatrix}
    T_{k+1}^{R,\mathrm{ind}} & *\\
    0 & Q_k^R
  \end{pmatrix},
  \qquad
  Q_k^R\in\operatorname{Mat}_{2\times2}.
  \label{eq:rank-two-block-form}
\end{equation}
It follows from
\eqref{eq:rank-two-determinant-ratio} and
\eqref{eq:rank-two-trace-difference} that
\begin{equation}
  \det Q_k^R=1,
  \qquad
  \operatorname{Tr}Q_k^R=\tau_k^{(R)}.
\end{equation}
Therefore
\begin{equation}
  \det\left(u\mathbf 1_2-Q_k^R\right)=u^2-\tau_k^{(R)}u+1=\left(u-\xi_{k+1}^{(R)}\right)
    \left(
      u-\bigl(\xi_{k+1}^{(R)}\bigr)^{-1}
    \right).
\end{equation}
Taking the determinant of the block-triangular matrix
\eqref{eq:rank-two-block-form} proves
\eqref{eq:rank-two-reciprocal-recurrence}.

By Lemma~\ref{lem:inductive-cauchy}, the Bethe equations are equivalent
to
\begin{equation}
  p_{y,\mathrm{ind}}^{(k+1)}
  =p_{y,\mathrm{fwd}}^{(k+1)},
  \qquad y\in\mathcal A_{k+1}.
  \label{eq:rank-two-induced-forward}
\end{equation}
Hence $T_{k+1}^{R,\mathrm{ind}}=T_{k+1}^R$ on the Bethe locus, proving
the last assertion.
\end{proof}

We can now formulate the main theorem of the section, which generalizes
Theorem~4.5 of \cite{Koroteev:2023aa}.

\begin{Thm}
\label{thm:betheandtrs}
Let $R$ be one of the types $A$, $B$, $C$, $D$, or $BC$, and let $u$ be
a formal variable.  Then there exists an isomorphism of algebras
\begin{equation}
  \frac{
    S\left(\{s_{i,j}\},\{a_{i,k}\},q\right)
  }{
    \left(\text{Bethe}\,\eqref{eq:folded-type-equations}\right)
  }
  \cong
  \frac{
    \mathbb C[\{\xi_i^{\pm 1}\},\{a_i^{\pm 1}\},\{p_i^{\pm 1}\},q^{\pm 1}]
  }{
    \left(
      \det\left(
        u\mathbf 1-T^R(\{p_i\},\{a_i\},q)
      \right)
      -F^R(u,\{\xi_i\})
    \right)
  }.
  \label{eq:bethe-trs-algebra-isomorphism}
\end{equation}
Here the prescribed characteristic polynomials are
\begin{align}
  F^A(u,\{\xi_i\})
  &=\prod_{i=1}^{M}
    \left(u-\xi_i^{(A)}\right),
  \label{eq:FA}\\
  F^R(u,\{\xi_i\})
  &=\prod_{i=1}^{N}
    \left(u-\xi_i^{(R)}\right)
    \left(u-\bigl(\xi_i^{(R)}\bigr)^{-1}\right),
  \qquad R=B,C,D,
  \label{eq:FBCD}\\
  F^{BC}(u,\{\xi_i\})
  &=(u-1)
    \prod_{i=1}^{N}
    \left(u-\xi_i^{(BC)}\right)
    \left(u-\bigl(\xi_i^{(BC)}\bigr)^{-1}\right).
  \label{eq:FBC}
\end{align}
\end{Thm}

\begin{proof}
For the system $\mathcal A_k$, define
\begin{equation}
  T_k^R=T^R\bigl(\mathcal A_k,p^{(k)}\bigr),
  \qquad
  P_k^R(u)
  =\det\left(u\mathbf 1_{m_k}-T_k^R\right),
  \qquad
  m_k=\dim T_k^R,
  \label{eq:level-k-characteristic-polynomial}
\end{equation}
with $T_0^R=T^R$.

In type $A$, the codimension-one intertwining relation of
Lemma~4.6 of \cite{Koroteev:2023aa}, together with the Bethe
compatibility condition, gives
\begin{equation}
  P_k^A(u)
  =\left(u-\xi_{k+1}^{(A)}\right)P_{k+1}^A(u),
  \qquad k=0,\ldots,M-1.
  \label{eq:type-A-characteristic-recursion}
\end{equation}
Since the terminal block is empty, $P_M^A(u)=1$, and iteration yields
$P_0^A(u)=F^A(u,\{\xi_i\})$.

For $R=B,C,D,BC$, Lemma~\ref{lem:rank-two-characteristic-reduction}
and the Bethe equations give
\begin{equation}
  P_k^R(u)
  =
  \left(u-\xi_{k+1}^{(R)}\right)
  \left(
    u-\bigl(\xi_{k+1}^{(R)}\bigr)^{-1}
  \right)
  P_{k+1}^R(u),
  \qquad k=0,\ldots,N-1.
  \label{eq:folded-characteristic-induction}
\end{equation}
For $R=B,C,D$, the terminal block is empty and hence $P_N^R(u)=1$.
For the genuine odd-dimensional $BC_N$ folding, the terminal block is
the inversion-fixed eigenvalue and
\begin{equation}
  P_N^{BC}(u)=u-1.
\end{equation}
Iterating \eqref{eq:folded-characteristic-induction} therefore gives
\begin{equation}
  P_0^R(u)=F^R(u,\{\xi_i\})
\end{equation}
for every folded type.  Equivalently,
\begin{equation}
  \det\left(u\mathbf 1-T^R\right)
  =F^R(u,\{\xi_i\}).
  \label{eq:full-characteristic-factorization}
\end{equation}
Together with the momentum parameterization and
compatibility with the Bethe Ansatz equations established above, this proves the claimed algebra isomorphism.
\end{proof}

\section{Folding of \texorpdfstring{$GL(2)$}{GL(2)} \texorpdfstring{$QQ$}{QQ}-Systems}\label{Sec:GL2folding}

In this section, we present the folding of the $QQ$-systems associated with $A_1$ quivers in Section \ref{sec:rank1folding} and the reflection-invariant conditions as well as their solutions in Section \ref{Sec:RI} for the higher rank cases. 

We begin with the following definition:
\begin{Def}
For any generic $A_r$-type quiver with its dimension vectors $\mathsf v$ and $\mathsf w$, we call it a \emph{good} quiver when
\begin{align}
    \mathsf v_{i+1}+\mathsf v_{i-1}+\mathsf w_{i}-2\mathsf v_{i}\in\Z_{\geq0},
\end{align}
for all $1\leq i \leq r$. When the bound is saturated for all $i$, the quiver is called \emph{balanced}.
\end{Def}

\subsection{Warm Up: rank-one Quivers}\label{sec:rank1folding}
In this subsection, we demonstrate the folding of $QQ$-systems associated with $A_1$ quivers to obtain the $C_2$ and $BC_2$ tRS models. After these canonical examples, we provide the folding of the $QQ$-system associated with the generic rank-1 quiver $T^*\mathrm{Gr}(k,n)$.  Those with odd and even $n$ respectively generalize tRS models of the $BC$ and $C$ types. The trivial case  $T^*\mathrm{Gr}(1,2)$ is treated separately after the general theory.
\subsubsection{$C_2$ model}
Consider the following balanced $A_1$ quiver for the $C_2$ model. 
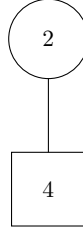
\begin{center}
 \begin{tikzpicture}[scale=.75, every node/.style={scale=0.8}]
        \draw[] (-0.5,0) circle (.7);
        \draw[] (-0.5,-.7) -- (-0.5,-2);
        \draw[] (-1.15,-2) rectangle (0.15,-3.3);
        \node at (-0.5,0) {$2$};
        \node at (-0.5,-2.65) {$4$};
    \end{tikzpicture}    
\end{center}
The $QQ$-system associated with $A_1$ quiver data $\{\xi_{1},\xi_{2},\mathsf v_1=2, \mathsf w_1=4\}$ reads
\begin{equation}
    \xi_2 Q^+(qz)Q^-(z)-\xi_1 Q^+(z)Q^-(qz)=\Lambda(z)\,.
\end{equation}
Before the folding, $Q^+(z)$ and $\Lambda(z)$ are monic polynomials of degree $2$ and $4$ from the dimension vectors of the quiver:
\begin{align}
  Q^+(z)=\a_+ \prod_{k=1}^{2}(z-s_k) \,,&&  \Lambda(z)=\lambda\prod_{k=1}^{4}(z-a_k)\,.    
\end{align}
By the quantum/classical duality in the electric frame \cite{Koroteev:2023aa}, the $A_1$ tRS momenta \eqref{eq:right-momentum-g} are given by
\begin{equation} 
    p_i= \xi_{2}\frac{Q^{+}(qa_i)}{Q^{+}(a_i)}\,,\quad i=1,\dots,4\,.
\end{equation}
By the $C_2$ folding condition, the canonical tRS variables should be folded as
\begin{align}
  p_i p_{5-i}=\xi_{2}^2\frac{Q^{+}(qa_i)}{Q^{+}(a_i)}\frac{Q^{+}(qa_{i}^{-1})}{Q^{+}(a_{i}^{-1})}=1,&&   a_ia_{5-i}=1,&&\quad i=1, 2.
\end{align}

\begin{Def}\label{def:C2folding}
The folded $QQ$-system dual to $C_2$ model is induced from the $QQ$-system associated with the given $A_1$ quiver by
\begin{align}
     \xi_2 \widetilde Q^+(qz) \widetilde Q^-(z)- \xi_1 \widetilde Q^+(z)\widetilde Q^-(qz)=q^{-1}\widetilde \Lambda(z)\,,
\end{align}
which consists of the following Laurent polynomials
\begin{align}
     \widetilde Q^\pm(z):=\left.\frac{Q^\pm(z)}{(-z)}\right|_{\textrm{folded}} =\a_{\pm} (z-s^\pm_1)\left(\frac{q}{z}-s^\pm_1\right)\in \mathbb{C}\left[z+\frac{q}{z}\right]\,,\label{eq:C2Q}
\end{align}
\begin{align}
     \widetilde \Lambda(z):=\left. \frac{\Lambda(z)}{(-z)^2}\right|_{\textrm{folded}}=\l z^{-2}\left(z-a_1\right)\left(z-a_2\right)\left(z-\frac{1}{a_1}\right)\left(z-\frac{1}{a_2}\right)\in\mathbb{C}\left[z+\frac{1}{z}\right]\,,\label{eq:C2lamb}
\end{align}
with the folding conditions: 
\begin{align}
    s^{\pm}_{1}s^{\pm}_{2}=q,&&a_{i}a_{5-i}=1,&&i=1,2.
\end{align}
\end{Def}

\begin{Prop}
There exists a folded $QQ$-system of Definition \ref{def:C2folding} with $q^2\xi^2_2=1$.
\end{Prop}

\begin{proof}
Let us first divide the $QQ$-system by $qz^2$. The right-hand side becomes the Laurent polynomial \eqref{eq:C2lamb} with the folding condition of tRS positions $a_3=1/a_2$ and $a_4=1/a_1$. On the left-hand side, each $z^{-1}$ can be absorbed into $Q^\pm(z)$ to make them Laurent polynomials $\widetilde Q^\pm(z)$. 

In terms of new $ \widetilde Q^+(z)$, the tRS momenta and their folding conditions become
\begin{align}
   p_i= q\xi_{2}\frac{ \widetilde Q^{+}(qa_i)}{ \widetilde Q^{+}(a_i)}\,,&& p_i p_{5-i}=   q^2 \xi_{2}^2\frac{ \widetilde Q^{+}(qa_i)}{ \widetilde Q^{+}(a_i)}\frac{ \widetilde Q^{+}(qa_{i}^{-1})}{ \widetilde Q^{+}(a_{i}^{-1})}=1\,,&& i=1, 2.
\end{align}
By taking $s^\pm_2=q/s^\pm_1$, $\widetilde Q^+(z)$ becomes in the form of \eqref{eq:C2Q}, and $\xi_{2}=\pm q^{-1}$ from the folding of tRS momenta due to
\begin{align}
\frac{\widetilde Q^{+}(qa_i)}{\widetilde Q^{+}(a_i)}\frac{\widetilde Q^{+}(qa_{i}^{-1})}{\widetilde Q^{+}(a_{i}^{-1})}=1.
\end{align}
\end{proof}
Note that the folding procedure includes folding the Bethe roots as well as the roots of $\Lambda$.

\begin{Prop} A reflection-invariant condition on the twists,
\begin{align}
   \xi_2 =- \xi_1\,,\label{eq:C2RIcond}
\end{align}
is induced by the consistency of the folded $QQ$-system, which consists of  $\widetilde Q^\pm(z)\in\C\left[z+\frac{q}{z}\right]$ and $\widetilde\Lambda(z)\in\C\left[z+\frac{1}{z}\right]$.

\end{Prop}
\begin{proof}
By definition, $\widetilde Q^\pm$ has the following property
\begin{align}
    \widetilde Q^\pm(qz)=\widetilde Q^\pm\left(\frac{1}{z}\right)\,.
\end{align}
The $QQ$-system can be written as
\begin{align}
 \xi_2 \widetilde Q^+\left(\frac{1}{z}\right)\widetilde Q^-(z)-\xi_1\widetilde Q^+(z) \widetilde Q^-\left(\frac{1}{z}\right)=q^{-1}\widetilde\Lambda(z)\in\C\left[z+\frac{1}{z}\right]\,,
\end{align}
To be symmetric under $z\mapsto 1/z$, the LHS should be
\begin{align}
-\xi_1\left( \widetilde Q^+\left(\frac{1}{z}\right)\widetilde Q^-(z)+\widetilde Q^+(z) \widetilde Q^-\left(\frac{1}{z}\right)\right)\in\C\left[z+\frac{1}{z}\right]\,.
\end{align}
The reflection-invariant condition is $\xi_2  =- \xi_1$.
\end{proof}

The Bethe equation can be obtained by the standard trick of eliminating $\widetilde Q^-$, we have the general formula with $s=s_1$
\begin{equation}
     -\frac{\xi_2}{\xi_1}\frac{\widetilde Q^+(qs)}{\widetilde Q^+(q^{-1}s)}\frac{\widetilde \Lambda(q^{-1}s)}{\widetilde \Lambda(s)}=1\,,
\end{equation}
or explicitly with $\xi_2  =- \xi_1$
\begin{align}
\begin{split}
 &-q^{-1}  \frac{(s^2-1)}{ (s^2-q^{2})} 
\prod_{k=1}^{2}\frac{(s-qa_{k})(q -s a_{ k})}{(s -a_{ k})(1-s a_{ k})}= 1.  
  \end{split}  
  \label{eq:BE2-4}
\end{align}
This is the XXZ Bethe ansatz equation for open $\mathfrak{sl}_2$ spin chain on 2 sites with one excitation (Bethe root) with trivial boundary conditions. This Bethe ansatz is dual to the $C_2$ tRS model.

The formula shows the $C_2$ boundary contributions as in \eqref{eq:boundary-table}.

\subsubsection{$BC_2$ model} Consider the following balanced $A_1$ quiver.
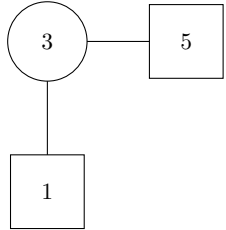
\begin{center}
 \begin{tikzpicture}[scale=.75, every node/.style={scale=0.8}]
        \draw[] (-0.5,0) circle (.7);
        \draw[] (1.3,0.65) rectangle (2.6,-0.65);
        \draw[] (0.2,0) -- (1.3,0);
        \draw[] (-0.5,-.7) -- (-0.5,-2);
        \draw[] (-1.15,-2) rectangle (0.15,-3.3);
        \node at (-0.5,0) {$3$};
        \node at (1.95,0) {$5$};
        \node at (-0.5,-2.65) {$1$};
    \end{tikzpicture}    
\end{center}
The $QQ$-system associated with the quiver data $\{\xi_{1},\xi_{2},\mathsf v_{i=1,2,3}=2i+1\}$ with $\mathsf v_0=\mathsf w'_1$ and $\mathsf v_2=\mathsf w'_2$ reads
\begin{equation}
    \xi_2 Q_1^+(qz)Q_1^-(z)-\xi_1 Q_1^+(z)Q_1^-(qz)=\Lambda(z)Q^+_0(qz)\,.
\end{equation}
where $Q^+_1(z)$, $Q^+_0(z)$ and $\Lambda(z)$ are monic polynomials of degree $3$, $1$ and $5$ from the dimension vectors of the quiver:
\begin{align}
  Q^+_{i=0,1}(z)=\a^{+}_i(z-\gamma^+_i)\prod_{k=1}^{2i}(z-s^+_{i,k}) \,,&&  \Lambda(z)=\l\prod_{k=1}^{5}(z-a_k) \,.    
\end{align}
Note that $Q^-_1(z)$ has degree 3.
By the $BC_2$ folding condition, the canonical tRS variables should be folded as
\begin{align}
  p_i p_{6-i}=\xi_{2}^2\frac{Q^{+}_1(qa_1)}{Q^{+}_1(a_1)}\frac{Q^{+}_1(qa_{1}^{-1})}{Q^{+}_1(a_{1}^{-1})}=1,&&   a_{i}a_{6-i}=1, 
\end{align}
for $i=1,2$ and
\begin{align}
    p_3=\xi_{2}\frac{Q^{+}_1(q)}{Q^{+}_1(1)}=1,&& a_3=1.
\end{align}
\begin{Def}\label{def:BC2folding}
The folded $QQ$-system dual to $BC_2$ model is induced from the $QQ$-system associated with the given $A_1$ quiver by
\begin{align}
 \widetilde\zeta_2(z)  \widetilde Q^+_1(qz)\widetilde Q^-_1(z)-\zeta_1(z)\widetilde Q^+_1(z)\widetilde Q^-_1(qz)=\widetilde\Lambda(z),
\end{align}
which consists of the following Laurent polynomials
\begin{align}
\widetilde Q^\pm_1(z):=\left.\frac{ Q^\pm_1(z)}{(-z)(z-\gamma^\pm_1)}\right|_{\textrm{folded}}=\a^\pm(z-s^\pm_1)\left(\frac{q}{z}-s^\pm_1\right)\in \mathbb{C}\left[z+\frac{q}{z}\right]\label{eq:BC1Q}
\end{align}
 \begin{align}
   &\widetilde   \Lambda(z):=\left.\frac{(z-a_2)Q^+_0(qz)}{(-qz)}\frac{\Lambda(z)}{(z-a_2)(-z)^2}\right|_{\textrm{folded}}\label{eq:BC1lamb}
   \\
   &=\l(z-1)\left(\frac{1}{z}-1\right)\prod_{k=1}^{2}\left(z-a_k\right) \left(\frac{1}{z}-a_k\right) \in\mathbb{C}\left[z+\frac{1}{z}\right].\nonumber
 \end{align}
with the folding conditions:
\begin{align}
    s^-_{1}s^-_{2}=q,&&a_1a_3=1,&&a_2=1,&&a_2\gamma^+_0=q.
\end{align}
The twists are generalized to meromorphic functions on $\P^1$ as 
 \begin{align}
       \widetilde \zeta_2(z)=\xi_2\frac{(qz-\gamma^+_1)(z-\gamma^-_1)}{(-z)},&&
       \zeta_1(z)=\xi_1\frac{(z-\gamma^+_1)(qz-\gamma^-_1)}{(-z)}\,.\label{eq:BC1twists}
 \end{align}
\end{Def}

\begin{Prop} 
There exists a folded $QQ$-system of Definition \ref{def:BC2folding} with $\xi_{2}=q^{-1} \frac{(1-\gamma^+_1)}{(q-\gamma^+_1)}$.
\end{Prop}
\begin{proof}
We divide the $QQ$-system by $z^3$.
\begin{align}
&\xi_{2}\frac{(qz-\gamma^+_1)(z-\gamma^-_1)}{(-z)} \frac{Q^+_{1}(qz)}{(-qz) (qz-\gamma^+_1)}\frac{Q^-_{1}(z)}{(-z)  (z-\gamma^-_1) }
\\&
-\xi_{1}\frac{(z-\gamma^+_1)(qz-\gamma^-_1)}{(-z)} \frac{Q^+_{1}( z)}{(- z) ( z-\gamma^+_1)}\frac{Q^-_{1}(qz)}{(-qz)  (qz-\gamma^-_1) }\nonumber
\\
&=\frac{(z-a_2)(qz-\gamma^+_{0})}{(-qz)} \frac{\Lambda^+_{1}(z)}{(-z)^2 (z-\gamma^+_{1})}  \,.\nonumber
\end{align}
By imposing the folding of the Bethe roots, we obtain the $QQ$-system of Definition \ref{def:BC2folding} with the generalized twists
 \begin{align}
       \widetilde \zeta_2(z):=-\xi_2\frac{(qz-\gamma^+)(z-\gamma^-)}{z},&&
       \zeta_1(z):=-\xi_1\frac{(z-\gamma^+)(qz-\gamma^-)}{z}\,.\label{eq:BC2twists}
 \end{align}
For the folding of tRS momenta, it suffices to solve
\begin{align}
    p_3=\xi_{2}\frac{Q^{+}_1(q)}{Q^{+}_1(1)}=q\xi_{2}\frac{(q-\gamma^+_1)}{(1-\gamma^+_1)}\frac{\widetilde Q^{+}_1(q)}{\widetilde  Q^{+}_1(1)}=1
\end{align}
The momentum folding condition is simply
\begin{align}
q\xi_{2} \frac{(q-\gamma^+_1)}{(1-\gamma^+_1)}=1.\label{eq:BC2foldingmomenta1}
\end{align}
\end{proof}

\begin{Prop}\label{prop:BC1RIcond}
The reflection-invariant condition on the twists,
\begin{align}
   \widetilde\zeta_2(z) =- \zeta_1\left(\frac{1}{z}\right)\,,\label{eq:BC1RIcond}
\end{align}
is induced by the consistency of the folded $QQ$-system, which consists of $\widetilde Q^\pm_{1}(z)\in\C\left[z+\frac{q}{z}\right]$ and $\widetilde\Lambda(z)\in\C\left[z+\frac{1}{z}\right]$.
\end{Prop}
\begin{proof}
By definition, $\widetilde Q^\pm_{1}$ has the following property
\begin{align}
    \widetilde Q^\pm_{1}(qz)=\widetilde Q^\pm_{1}\left(\frac{1}{z}\right)\,.
\end{align}
The $QQ$-system can be written as
\begin{align}
 \widetilde\zeta_2(z) \widetilde Q^+_{1}\left(\frac{1}{z}\right)\widetilde Q^-_{1}(z)-\zeta_1(z)\widetilde Q^+_{1}(z) \widetilde Q^-_{1}\left(\frac{1}{z}\right)=\widetilde\Lambda(z)\in\C\left[z+\frac{1}{z}\right]\,,
\end{align}
To be symmetric under $z\mapsto 1/z$, the LHS should be
\begin{align}
-\left( \zeta_1\left(\frac{1}{z}\right)\widetilde Q^+_{1}\left(\frac{1}{z}\right)\widetilde Q^-_{1}(z)+\zeta_1(z)\widetilde Q^+_{1}(z) \widetilde Q^-_{1}\left(\frac{1}{z}\right)\right)\in\C\left[z+\frac{1}{z}\right]\,.
\end{align}
The reflection-invariant condition is 
\begin{align}
   \widetilde\zeta_2(z) =- \zeta_1\left(\frac{1}{z}\right).
\end{align}
\end{proof}
The above reflection-invariant condition on the generalized twists is solved by
\begin{align}
 (\gamma^\pm_1)^2=q,&&  \frac{\xi_2}{ \xi_1}
=-\frac{\gamma^+_1\gamma^-_1}{q}.
\end{align}

The remaining folding conditions of tRS momenta
\begin{align}
   \xi_{2}^2\frac{Q^{+}_1(qa_i)}{Q^{+}_1(a_i)}\frac{Q^{+}_1(qa_{i}^{-1})}{Q^{+}_1(a_{i}^{-1})}=1 \,,&&i=1,2\,,
\end{align}
are solved by \eqref{eq:BC2foldingmomenta1} with $\gamma^+_1=q^{\frac12}$.

The Bethe equation is obtained by a similar way, but with the Bethe roots of $\widetilde Q^-$.
\begin{align}
 -\frac{ \widetilde\zeta_2(s^+_1)  }{\zeta_1(q^{-1}s^+_1) } \frac{ \widetilde Q^+_{1}(qs) }{\widetilde Q^+_{1}(q^{-1}s^+_1) }\frac{\widetilde\Lambda(q^{-1}s^+_1)  }{\widetilde\Lambda(s^+_1)  }=1.
\end{align}
or explicitly
\begin{align}
\pm q^{-4}\frac{(qs^+_1-\gamma^+)}{(s^+_1-q\gamma^+)}\frac{(1-(s^+_1)^2)}{(q^2-(s^+_1)^2)}\cdot\frac{(s^+_1-q)^2}{(s^+_1-1)^2}\prod_{j=1}^{2}\frac{\left(s^+_1-qa_j\right) \left(q- s^+_1a_j\right)}{\left(s^+_1-a_j\right) \left(1-s^+_1a_j\right)}=1\,.
\end{align}
where the sign comes from $\gamma^\pm=\pm \sqrt{q}$. This is the XXZ Bethe ansatz equation for open $\mathfrak{sl}_2$ spin chain on 3 sites ($1,a_1,a_2$) with one excitation (Bethe root) with trivial boundary conditions. This Bethe ansatz is dual to the $BC_2$ tRS model.
In terms of $f_q$ and $g_q$,
\begin{align}
    \pm q^{-4} f^{-1}_q\left(\frac{s^+_1}{\gamma^+_1}\right)f^{-1}_q\left(q^{-1}(s^+_1)^2\right)  g^{2}_q(s^+_1)\prod_{j=1}^{2}g_q\left(\frac{s^+_1}{a_j}\right)g_q\left(s^+_1 a_j\right)=1\,.
\end{align}
One can notice the $BC_2$ boundary contributions as in \eqref{eq:boundary-table}.

\subsubsection{$T^*\mathrm{Gr}(k,n)$ model}\label{sec:Gr}
Consider the following good $A_1$ quivers.
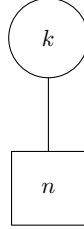
\begin{center}
   \begin{tikzpicture}[scale=.75, every node/.style={scale=0.8}]
        \draw[] (-0.5,0) circle (.7);
        \draw[] (-0.5,-.7) -- (-0.5,-2);
        \draw[] (-1.15,-2) rectangle (0.15,-3.3);
        \node at (-0.5,0) {$k$};
        \node at (-0.5,-2.65) {$n$};
    \end{tikzpicture}
\end{center}
The $QQ$-system associated with the quiver reads
\begin{equation}
    \xi_2 Q^+(qz)Q^-(z)-\xi_1 Q^+(z)Q^-(qz)=\Lambda(z)\,,
\end{equation}
where $Q^+(z)$ and $\Lambda(z)$ are monic polynomials of degree $k$ and $n$
\begin{align}
  Q_+(z)=\prod_{j=1}^{k}(z-s_j) \,,&&  \Lambda(z)=\prod_{j=1}^{n}\lambda(z-a_j) \,.    
\end{align}
Let us define $\e^+, \e^-$ and $\e^\circ$ be the remainder of $\mathsf v^+=k$, $\mathsf v^-=n-k$, and $n$ when divided by $2$. The folding conditions are 
\begin{align}
   p_i p_{n+1-i}=\xi_{2}^2\frac{Q^{+}(qa_i)}{Q^{+}(a_i)}\frac{Q^{+}(qa_{i}^{-1})}{Q^{+}(a_{i}^{-1})}=1,&&   a_ia_{n+1-i}=1,&&\quad i=1,\cdots,\frac{n-\e^\circ}{2},
\end{align}
and if $n$ is odd, that is, $\e^\circ=1$, there are supplementary conditions 
\begin{align}
    p_{\frac{n+1}{2}}=\xi_{2}\frac{Q^{+}(q)}{Q^{+}(1)}=1\,,&& a_\frac{n+1}{2}=1\,.
\end{align}

\begin{Def}\label{Def:Gr(k,n)folding}
The type folded $QQ$-system dual to the $T^*\mathrm{Gr}(k,n)$ model is induced from the $QQ$-system associated with the given $A_1$ quiver by
\begin{align}
 \widetilde\zeta_2(z) \widetilde Q^+(qz)\widetilde Q^-(z)-\zeta_1(z)\widetilde Q^+(z) \widetilde Q^-(qz)=\widetilde\Lambda(z)\,,
\end{align}    
which consists of Laurent polynomials induced from the monic polynomials
\begin{align}
\widetilde Q^\pm(z):=\left.\frac{(-z)^{-\frac{\mathsf v^\pm-\e^\pm}{2}}}{(z-\gamma^\pm)^{\e^\pm}}Q^\pm(z)\right|_{\textrm{folded}}=\a_{\pm}\prod_{k=1}^{\frac{1}{2}(\mathsf v^\pm-\e^\pm)}(z-s^\pm_k)\left(\frac{q}{z}-s^\pm_k\right),\, \label{eq:GrknQpmBCC}    
\end{align}
for the $BC/C$ type folding, and  
\begin{align}
\widetilde Q^+(z):=\left.\frac{(-z)^{-\frac{\mathsf v^+ -\e^+}{2}}Q^+(z)}{(z-\gamma^+)^{\e^+}}\right|_{\textrm{folded}}=\a_{+}\prod_{k=1}^{\frac{1}{2}(\mathsf v^+-\e^+)}(z-s^+_k)\left(\frac{q}{z}-s^+_k\right)\,,  \label{eq:GrknQpBD}
\\
\widetilde Q^-(z):=\left.\frac{(-z)^{-\frac{\mathsf v^- -\e^\pm}{2}}Q^-(z)}{(z-\gamma^-)^{\e^-}(z-\gamma^-_1)(z-\gamma^-_2)}\right|_{\textrm{folded}}=\a_{-}\prod_{k=1}^{\frac{1}{2}(\mathsf v^- -2-\e^-)}(z-s^-_k)\left(\frac{q}{z}-s^-_k\right)\,,   \label{eq:GrknQmBD}
\end{align}
for $B/D$ type folding where $\mathsf v^+=k$ and $\mathsf v^-=n-k$. For the framing, we define
\begin{align}
\widetilde \Lambda(z):=\left.\frac{(-z)^{-\frac{n-\e^\circ}{2}}}{(z-\gamma^+_2)^{\e^\circ}}\Lambda(z)\right|_{\textrm{folded}}=\l\prod_{k=1}^{\frac{1}{2}(n-\e^\circ)}(z-a_k)\left(\frac{1}{z}-a_k\right)\,,&&BC/C
\label{eq:GrknQLBCC}
\end{align}
with the folding conditions for $BC/C$ type: 
\begin{align}
    s^\pm_{i}s^\pm_{\mathsf v^\pm+1-i}=q,\quad i=1,\cdots,\frac{\mathsf v^\pm-\e^\pm}{2},&&a_{j}a_{n+1-j}=1,\quad j=1,\cdots,\frac{n-\e^\circ}{2}, 
\end{align}
and $a_{\frac{n-\e^\circ}{2}}=1$ for odd $n$.
Similarly
\begin{align}
    \widetilde \Lambda(z):=\left.\frac{(-z)^{-\frac{n-\e^\circ}{2}}\Lambda(z)}{(z-\gamma^+_2)^{\e^\circ}(z-\gamma^\circ)(z+\gamma^\circ)}\right|_{\textrm{folded}}=\l\prod_{k=1}^{\frac{1}{2}(n-2-\e^\circ)}(z-a_k)\left(\frac{1}{z}-a_k\right)\,,&&B/D
\label{eq:GrknQLBD}
\end{align}
with the folding conditions for $B/D$ type: $\gamma^\circ_1+\gamma^\circ_2=0$, 
\begin{align}
    s^\pm_{i}s^\pm_{\mathsf v^\pm+1-i}=q,\quad i=1,\cdots,\frac{\mathsf v^\pm-\e^\pm}{2},&&a_{j}a_{n+1-j}=1,\quad j=1,\cdots,\frac{n-2-\e^\circ}{2}, 
\end{align}
and $a_{\frac{n-2-\e^\circ}{2}}=1$ for odd $n$.

The twists for $BC/C$ types are generalized to meromorphic functions on $\P^1$ as 
\begin{align}
    \widetilde\zeta_2(z):=\xi_2 q^{\frac{k-\e^+}{2}}(-z)^{\frac{1}{2}(\e^\circ-\e^+ -\e^-)}\frac{(qz-\gamma^+)^{\e^+}(z-\gamma^-)^{\e^-}}{(z-\gamma^+_2)^{\e^\circ}},
\\
\zeta_1(z):=\xi_1 q^{\frac{n-k-\e^-}{2}}(-z)^{\frac{1}{2}(\e^\circ- \e^+ - \e^-)}\frac{(z-\gamma^+)^{\e^+}(qz-\gamma^-)^{\e^-}}{(z-\gamma^+_2)^{\e^\circ}}.\label{eq:GrkntwistsBCC}
\end{align}
The twists for $B/D$ types are generalized to meromorphic functions on $\P^1$ as 
\begin{align}
    \widetilde\zeta_2(z):=-\xi_2 q^{\frac{k-\e^+}{2}}z^{\frac{1}{2}(\e^\circ+2-\e^+ -\e^-)}\frac{(qz-\gamma^+)^{\e^+}(z-\gamma^-)^{\e^-}}{(z-\gamma^+_2)^{\e^\circ}(z-\gamma^\circ)(z+\gamma^\circ)},
\\
\zeta_1(z):=- \xi_1 q^{\frac{n-k-\e^-}{2}}(-z)^{\frac{1}{2}(\e^\circ+2- \e^+ - \e^-)}\frac{(z-\gamma^+)^{\e^+}(qz-\gamma^-)^{\e^-}}{(z-\gamma^+_2)^{\e^\circ}(z-\gamma^\circ)(z+\gamma^\circ)}.\label{eq:GrkntwistsBD}
\end{align}
\end{Def}
\begin{Lem}\label{lem:generic_rank1}
  There exists a folded $QQ$-system in the sense of Definition \ref{Def:Gr(k,n)folding} for any $k,n$ with $n\geq2k$, $\xi_{2}^2q^{k}=1$ for even $k$ ($C/D$ type) or $\xi_{2}q^{\frac{k-1}{2}}\frac{(q -\gamma) }{(1-\gamma) }=1$ for odd $k$ ($BC/B$ type).
\end{Lem}
\begin{proof}
First, divide the $QQ$-system by $(-z)^{\frac{n-\e^\circ}{2}}(z-\gamma^+_2)^{\e^\circ}$. 
\begin{align}
  & \xi_2  \frac{q^{\frac{1}{2}(k-\e^+)}(qz-\gamma^+)^{\e^+}(z-\gamma^-)^{\e^-}}{(z-\gamma^+_2)^{\e^\circ}(-z)^{\frac{1}{2}(-\e^\circ+\e^++\e^-)}}\frac{Q^+_i(qz)}{(qz-\gamma^+)^{\e^+}(-qz)^{\frac{1}{2}(k-\e^+)}}\frac{Q^-_i(z)}{(z-\gamma^-)^{\e^-}(-z)^{\frac{1}{2}(n-k-\e^-)}}
    \\
  &-\xi_1\frac{q^{\frac{1}{2}(n-k-\e^-)} (z-\gamma^+)^{\e^+}(qz-\gamma^-)^{\e^-}}{(z-\gamma^+_2)^{\e^\circ}(-z)^{\frac{1}{2}(-\e^\circ+\e^++\e^-)}}\frac{Q^+_i(z)}{(z-\gamma^+)^{\e^+}(-z)^{\frac{1}{2}(k-\e^+)}}\frac{Q^-_i(qz)}{(qz-\gamma^-)^{\e^-}(-qz)^{\frac{1}{2}(n-k-\e^-)}}  \nonumber
  \\
  &=\frac{\Lambda(z)}{(z-\gamma^+_2)^{\e^\circ}(-z)^{\frac{1}{2}(n-\e^\circ)}}\nonumber\,.
\end{align}
By imposing the folding conditions for $BC/C$ type on Bethe roots, the $QQ$-systems in Definition \ref{Def:Gr(k,n)folding} are obtained. 
In terms of $\widetilde Q^+$, tRS momenta become
\begin{align}
     p_i =\xi_{2}\frac{Q^{+}(qa_i)}{Q^{+}(a_i)}=\xi_{2}q^{\frac{k-\e^+}{2}}\frac{(qa_i-\gamma)^{\e^+}}{(a_i-\gamma)^{\e^+}}\frac{\widetilde Q^{+}(qa_i)}{\widetilde Q^{+}(a_i)},&&i=1,\cdots,\frac{k-\e^+}{2}.
\end{align}
The momenta folding conditions for the generalized $C$ and $BC$ type are
\begin{itemize}
    \item $\e^+=0$, 
\begin{align}
     p_i p_{n+1-i}=\xi_{2}^2q^{k}\frac{\widetilde Q^{+}(qa_i)}{\widetilde Q^{+}(a_i)}\frac{\widetilde Q^{+}(qa_{i}^{-1})}{\widetilde Q^{+}(a_{i}^{-1})}=\xi_{2}^2q^{k}=1\,.
\end{align}
 \item  $\e^+=1$, 
\begin{align}
  p_{\frac{k+1}{2}}= \xi_{2}q^{\frac{k-1}{2}}\frac{(q -\gamma) }{(1-\gamma) }\frac{\widetilde Q^{+}(q )}{\widetilde Q^{+}(1)}=\xi_{2}q^{\frac{k-1}{2}}\frac{(q -\gamma) }{(1-\gamma) }=1\,.
\end{align}
\end{itemize}
For the $B/D$ types, divide the $QQ$-system by $(-z)^{\frac{n-2-\e^\circ}{2}}(z-\gamma^+_2)^{\e^\circ}(z-\gamma^\circ_1)(z-\gamma^\circ_2)$. The remaining procedures are the same.
\end{proof}
\begin{Lem}\label{lem:rank1RIcond} 
The reflection-invariant condition on the generalized twists 
\begin{align}
   \widetilde\zeta_2(z) =- \zeta_1\left(\frac{1}{z}\right)\,.\label{eq:rank1-RIcond}
\end{align}
is induced by the consistency of the folded $QQ$-system induced from the generic $A_1$ quiver, which consists of  $\widetilde Q^\pm(z)\in\C\left[z+\frac{q}{z}\right]$ and $\widetilde\Lambda(z)\in\C\left[z+\frac{1}{z}\right]$.
\end{Lem}
\begin{proof}
It is the $r=1$ case of Lemma \ref{lem:generalRItwists}.
\end{proof}
In the higher rank cases, we will see that the consistency of the $QQ$-system requires the reflection-invariant conditions in the same way.

\begin{Lem}
The reflection-invariant condition \eqref{eq:rank1-RIcond} on the generalized twists for $BC/C$ type is solved by 
\begin{align}
\frac{\xi_2}{\xi_1} =  - q^{\frac{n-2k-\e^+-\e^-}{2}}\frac{(\gamma^-)^{\e^-}(\gamma^+)^{\e^+}}{(\gamma^+_2)^{\e^\circ}},&&   (\gamma^\pm)^{2}=q\,,
\end{align}
as well as for $B/D$ type,
\begin{align}
\frac{\xi_2}{\xi_1}=   q^{\frac{n-2k-\e^+-\e^-}{2}}\frac{(\gamma^-)^{\e^-}(\gamma^+)^{\e^+}}{(\gamma^+_2)^{\e^\circ}(\gamma^\circ)^2},&&   (\gamma^\pm)^{2}=(\gamma^\pm_2)^{2}=q\,,&&(\gamma^\circ)^2=\pm1\,.
\end{align}    
\end{Lem}
\begin{proof}
They follow from Lemma \ref{Lem:Gen_RI_Sols}.
\end{proof}

The Bethe equations from the folded $QQ$-system are obtained by the same standard trick to eliminate $\widetilde Q^-$ for $k>1$.
\begin{align}
   - \frac{ \widetilde\zeta_2(s^+_i)}{\zeta_1(q^{-1}s^+_i)}\frac{ \widetilde Q^+(qs^+_i)}{\widetilde Q^+(q^{-1}s^+_i) }\frac{\widetilde\Lambda(q^{-1}s^+_i)}{\widetilde\Lambda(s^+_i)}=1\,,
\end{align}
where $i=1,\cdots,\frac{k-\e^+}{2}$. 

The Bethe equations for $T^*\mathrm{Gr}(1,n)$ models for $n>2$ are written in terms of the Bethe roots of $\widetilde Q^-$.
\begin{align}
 -\frac{ \widetilde\zeta_2(q^{-1}s^-_i ) }{\zeta_1(s^-_i ) } \frac{ \widetilde Q^-(q^{-1}s^-_i ) }{\widetilde Q^-(qs^-_i ) }\frac{\widetilde\Lambda(s^-_i )  }{\widetilde\Lambda(q^{-1}s^-_i )  }=1,
\end{align}
where $i=1,\cdots,\frac{n-k-\e^-}{2}$.

There are additional boundary contributions involving $\gamma^+_2$ in addition to those of \eqref{eq:boundary-table}.

\subsubsection{Trivial $T^*\mathrm{Gr}(1,2)$ model} It is the trivial case of folding. Consider the following good $A_1$ quiver. 
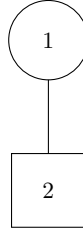
\begin{center}
   \begin{tikzpicture}[scale=.75, every node/.style={scale=0.8}]
        \draw[] (-0.5,0) circle (.7);
        \draw[] (-0.5,-.7) -- (-0.5,-2);
        \draw[] (-1.15,-2) rectangle (0.15,-3.3);
        \node at (-0.5,0) {$1$};
        \node at (-0.5,-2.65) {$2$};
    \end{tikzpicture}
\end{center}
Since $\widetilde Q^\pm(z)=1$ by Lemma \ref{lem:generic_rank1}, the folded $QQ$-system for $\mathrm{Gr}(1,2)$ model is
\begin{align}
 \widetilde\zeta_2(z) -\zeta_1(z) =\widetilde\Lambda(z)\,,
\end{align} 
where
\begin{align}
    \widetilde\zeta_2(z):=\xi_2  z^{-1} (qz-\gamma^+) (z-\gamma^-)  \,,
&&
\zeta_1(z):= \xi_1 z^{-1} (z-\gamma^+) (qz-\gamma^-)\, . 
\end{align}
The reflection-invariant condition is
\begin{align}
  \widetilde\zeta_2(z)   =-\zeta_1\left(\frac{1}{z}\right),
\end{align}
which is solved by $ (\gamma^\pm)^{2}=q$. There is no non-trivial Bethe equation, and it corresponds to the open $\sl_2$ XXZ spin chain on a single site with no excitation.

\subsection{Reflection-Invariance of the Generalized Twists}\label{Sec:RI}

The general procedure of folding the $A$-type $QQ$-system is to make the $Q$-functions be Laurent polynomials with certain symmetries, which induce the reflection-invariant conditions on the twists generalized to the rational functions on $\P^1$ (See Lemma \ref{lem:generalRItwists}). In Lemma \ref{Lem:Gen_RI_Sols}, it is proved that there always exist solutions to the reflection-invariant conditions for the classical groups. 

\begin{Lem}\label{lem:generalRItwists}
For any folded $QQ$-system associated with a good $A_r$ quiver in the form of 
\begin{align}
     \widetilde \zeta_{i+1}(z) \widetilde Q^+_i(qz) \widetilde Q^-_i(z)-  \zeta_i(z) \widetilde Q^+_i(z)\widetilde Q^-_i(qz)=\mathcal N\,\widetilde  Q^+_{i+1}(z) \widetilde Q^+_{i-1}(qz)\,, \label{eq:gen_fold_QQ}
\end{align}
consists of the folded $Q$ functions $\widetilde Q^+_i(z)\in\mathbb{C}\left[z+\dfrac{q^{r+1-i}}{z}\right]$
where $\mathcal N\in \mathbb{C}^\times$, the generalized twists $ \widetilde \zeta_{i+1}(z)$ and $\zeta_i(z)$ must satisfy the following reflection-invariant conditions
\begin{align}
   \widetilde\zeta_{i+1}(z) =- \zeta_i\left(\frac{q^{r-i}}{z}\right)\,,\label{eq:RIcond}
\end{align}
for $i=1,\cdots,r$.
\end{Lem}
\begin{proof}
By definition, $\widetilde Q^\pm_i(z)\in\mathbb{C}\left[z+\frac{q^{r+1-i}}{z}\right]$ has the property that
\begin{align}
    \widetilde Q^\pm_i(qz)=\widetilde Q^\pm_i\left(\frac{q^{r-i}}{z}\right).
\end{align}
Thus,
\begin{align}
     \widetilde \zeta_{i+1}(z) \widetilde Q^+_i\left(\frac{q^{r-i}}{z}\right) \widetilde Q^-_i(z)-  \zeta_i(z) \widetilde Q^+_i(z)\widetilde Q^-_i\left(\frac{q^{r-i}}{z}\right)=q^{-1}\widetilde  Q^+_{i+1}(z) \widetilde Q^+_{i-1}(qz) \,.
\end{align}
Since the RHS of the folded $QQ$-system is in $\mathbb{C}\left[z+\frac{q^{r-i}}{z}\right]$, the LHS should be symmetric under $z\mapsto q^{r-i}/z$ as
\begin{align}
 -\left[\zeta_i\left(\frac{q^{r-i}}{z}\right)  \widetilde Q^+_i\left(\frac{q^{r-i}}{z}\right) \widetilde Q^-_i(z)+  \zeta_i(z) \widetilde Q^+_i(z)\widetilde Q^-_i\left(\frac{q^{r-i}}{z}\right)\right]\,.
\end{align}
Hence, \eqref{eq:RIcond}.
\end{proof}

\begin{Lem}\label{Lem:Gen_RI_Sols}
Let $\widetilde \zeta_{i+1}(z)$ and $\zeta_{i}(z)$ be the rational functions in the form of
\begin{align}
   \widetilde \zeta_{i+1}(z) =\xi_{i+1}\frac{\mathcal{A}^+_i(qz)\mathcal{A}^-_i(z)}{\mathcal{A}^\circ_i(z)\mathcal{A}^+_{i-1}(qz)}\,,&& \zeta_{i}(z)   =\xi_i \frac{\mathcal{A}^+_i(z)\mathcal{A}^-_i(qz)}{\mathcal{A}^\circ_i(z)\mathcal{A}^+_{i-1}(qz)}\,,
\end{align}
where $\mathcal{A}^\bullet_i(z)=z^{\b^\bullet_i}\prod_k^{\a^\bullet_i}(z-\gamma^\bullet_{i,k})$. 
The reflection-invariant conditions
\begin{align}
   \widetilde\zeta_{i+1}(z) =- \zeta_i\left(\frac{q^{r-i}}{z}\right)\,,
\end{align}
have solutions when $-(\a^+_i+\a^-_i-\a^\circ_i-\a^+_{i-1})=2(\b^+_i+\b^-_i-\b^\circ_i-\b^+_{i-1})$ for all $i$. The solutions are
\begin{align}
    \gamma^\circ_{i,k}\gamma^\circ_{i,k'}=q^{r-i}\,,&&\gamma^+_{i-1,k}\gamma^+_{i-1,k'}=q^{r+2-i}\,,
\end{align}
and
\begin{align}
  \gamma^\pm_{i,k}\gamma^\pm_{i,k'}=q^{r+1-i},\qquad\textrm{or}\qquad \gamma^\pm_{i,k}\gamma^\mp_{i,k'}=q^{r+1-i}\, ,  
\end{align}
for $k,k'=1,\cdots,\a^\bullet_i$, and
\begin{align}
      \frac{\xi_{i+1}}{\xi_i}=-q^{\b^-_i-\b^+_i-\a^+_i+\a^+_{i-1}+(r-i)(\b^+_i+\b^-_i-\b^\circ_i-\b^+_{i-1})}\frac{\prod_k^{\a^+_i}\gamma^+_{i,k}\prod_k^{\a^-_i}\gamma^-_{i,k}}{\prod_k^{\a^\circ_i}\gamma^\circ_{i,k}\prod_k^{\a^+_{i-1}}\gamma^+_{i-1,k}}\,,
\end{align}
for $i=1,\cdots,r$.
\end{Lem}
\begin{proof}
The reflection-invariant conditions in terms of $\mathcal{A}^\bullet_i$ are
\begin{align}
 \xi_{i+1}\frac{\mathcal{A}^+_i(qz)\mathcal{A}^-_i(z)}{\mathcal{A}^\circ_i(z)\mathcal{A}^+_{i-1}(qz)}=-\xi_i \frac{\mathcal{A}^+_i\left(\frac{q^{r-i}}{z}\right)\mathcal{A}^-_i\left(\frac{q^{r+1-i}}{z}\right)}{\mathcal{A}^\circ_i\left(\frac{q^{r-i}}{z}\right)\mathcal{A}^+_{i-1}\left(\frac{q^{r+1-i}}{z}\right)}
\end{align}
To solve this, we match the leading coefficient and the singularity structure. The LHS is  
\begin{align}
  \xi_{i+1}q^{\b^+_i-\b^+_{i-1}+\a^+_i-\a^+_{i-1}}z^{\b^+_i+\b^-_i-\b^\circ_i-\b^+_{i-1}}\frac{\prod_k^{\a^+_i}(z-q^{-1}\gamma^+_{i,k})\prod_k^{\a^-_i}(z-\gamma^-_{i,k})}{\prod_k^{\a^\circ_i}(z-\gamma^\circ_{i,k})\prod_k^{\a^+_{i-1}}(z-q^{-1}\gamma^+_{i-1,k})}   \,,
\end{align}
and the RHS is
\begin{align}
-\xi_i  q^{\b^-_i-\b^+_{i-1}} \left(\frac{q^{r-i}}{z}\right)^{\b^+_i+\b^-_i-\b^\circ_i-\b^+_{i-1}}(-z)^{-(\a^+_i+\a^-_i-\a^\circ_i-\a^+_{i-1})}\frac{\prod_k^{\a^+_i}\gamma^+_{i,k}\prod_k^{\a^-_i}\gamma^-_{i,k}}{\prod_k^{\a^\circ_i}\gamma^\circ_{i,k}\prod_k^{\a^+_{i-1}}\gamma^+_{i-1,k}}
\end{align}
\[\times\frac{\prod_k^{\a^+_i}(z-q^{r-i}(\gamma^+_{i,k})^{-1})\prod_k^{\a^-_i}(z-q^{r+1-i}(\gamma^-_{i,k})^{-1})}{\prod_k^{\a^\circ_i}(z-q^{r-i}(\gamma^\circ_{i,k})^{-1})\prod_k^{\a^+_{i-1}}(z-q^{r+1-i}(\gamma^+_{i-1,k})^{-1})}\,. \]
By matching the leading order of the Laurent polynomials,
\begin{align}
  -(\a^+_i+\a^-_i-\a^\circ_i-\a^+_{i-1})=2(\b^+_i+\b^-_i-\b^\circ_i-\b^+_{i-1}) \,. 
\end{align}
For the poles, we have $\gamma^\circ_{i,k}\gamma^\circ_{i,k'}=q^{r-i}$ and $\gamma^+_{i-1,k}\gamma^+_{i-1,k'}=q^{r+2-i}$  for arbitrary pairs of $(k,k')$. For the zeros, any pair among $(\gamma^\pm_{i,k},\gamma^\pm_{i,k'})$ and $(\gamma^\pm_{i,k},\gamma^\mp_{i,k'})$, their product should be $q^{r+1-i}$.
Lastly, the overall coefficient fixes the ratio of the constant twists
\begin{align}
    \frac{\xi_{i+1}}{\xi_i}=-q^{\b^-_i-\b^+_i-\a^+_i+\a^+_{i-1}+(r-i)(\b^+_i+\b^-_i-\b^\circ_i-\b^+_{i-1})}\frac{\prod_k^{\a^+_i}\gamma^+_{i,k}\prod_k^{\a^-_i}\gamma^-_{i,k}}{\prod_k^{\a^\circ_i}\gamma^\circ_{i,k}\prod_k^{\a^+_{i-1}}\gamma^+_{i-1,k}}\,.
\end{align}

\end{proof}

\section{Folding for \texorpdfstring{$C$}{C} and \texorpdfstring{$D$}{D}-Type Systems}\label{sec:CD}

In this section, we present the folding of the $QQ$-systems which lead to tRS systems of $C$ and $D$-type. The $A_{n-1}$ quivers for $C_n$ and $D_n$ are characterized by $\mathsf v_i=2i$ for $1\leq i \leq n$, where $\mathsf v_{n}:=\mathsf w'_{n-1}=2n$ for the convenience of dealing with both types in a unified way.

For $D_n$ model, we introduce extra framings of $\mathsf w_i=2$ to the quiver for $C_n$ model. The corresponding $Q$ functions are designed to cancel the factor $f_q(q^{-(n-i)}(s^+_{i,k})^2)$ in the Bethe equations of the $C_n$ model. This leads to Definition \ref{def:CDfolding}. 

These definitions reproduce the correct Bethe equations of tRS systems for $C$ and $D$ by the standard derivation. Therefore, we formulate Theorem \ref{thm:CD}.

\subsection{$QQ$-Systems Associated with $A_{n-1}$ Quivers}
The $QQ$-systems before folding for $C_n$ and $D_n$ tRS models are associated with the ``even'' type of $A_{n-1}$ quivers. For the $C_n$ model, we consider the following balanced (\textit{i.e.} $2\mathsf v_i= \mathsf v_{i-1}+\mathsf v_{i+1}+\mathsf w_i $) $A_{n-1}$ quiver
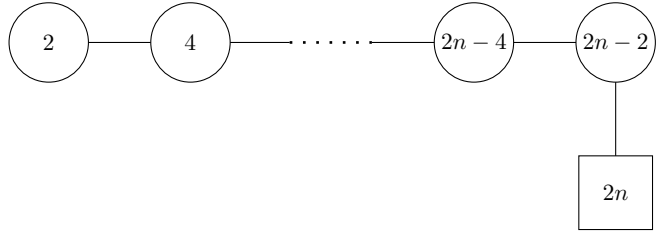
\begin{center}
        \begin{tikzpicture}[scale=.75, every node/.style={scale=0.8}]
        \draw[] (-6,0) circle (.7);
        \draw[] (-5.3,0) -- (-4.2,0);
        \draw[] (-3.5,0) circle (.7);
        \draw[] (-2.8,0) -- (-1.75,0);
        \draw[line width=1pt,loosely dotted] (-1.75,0) -- (-.3,0);
        \draw[] (-.3,0) -- (.8,0);
        \draw[] (1.5,0) circle (.7);   
        \draw[] (2.2,0) -- (3.3,0);
        \draw[] (4,0) circle (.7);      
        \draw[] (4,-.7) -- (4,-2); 
        \draw[] (3.35,-2) rectangle (4.65,-3.3);
  
        \node at (-6,0) {$2$};
        \node at (-3.5,0) {$4$};
        \node at (1.5,0) {$2n-4$};
        \node at (4,0) {$2n-2$};
        \node at (4,-2.65) {$2n$};
    \end{tikzpicture}
    \end{center}
with the quiver data $\{\xi_{i},\mathsf v_{i}=2i\}|_{i=1,\cdots,n}$ with $\mathsf v_n:=\mathsf w'_{n-1}=2n$.

On the other hand, for the $D_n$ model, we add framings to the $A_{n-1}$ quiver dual to the $C_n$ model 
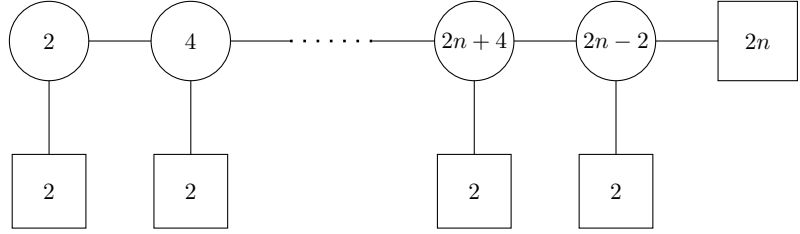
\begin{center}
    \begin{tikzpicture}[scale=.75, every node/.style={scale=0.8}]
        \draw[] (-6,0) circle (.7);
        \draw[] (-5.3,0) -- (-4.2,0);
        \draw[] (-3.5,0) circle (.7);
        \draw[] (-2.8,0) -- (-1.75,0);
        \draw[line width=1pt,loosely dotted] (-1.75,0) -- (-.3,0);
        \draw[] (-.3,0) -- (.8,0);
        \draw[] (1.5,0) circle (.7);   
        \draw[] (2.2,0) -- (3.3,0);
        \draw[] (4,0) circle (.7);      
        \draw[] (4.7,0) -- (5.8,0);
        \draw[] (5.8,-0.7) rectangle (7.2,0.7);

        \draw[] (4,-.7) -- (4,-2);
        \draw[] (1.5,-.7) -- (1.5,-2);
        \draw[] (-3.5,-.7) -- (-3.5,-2);
        \draw[] (-6,-.7) -- (-6,-2);   
        \draw[] (3.35,-2) rectangle (4.65,-3.3);
        \draw[] (0.85,-2) rectangle (2.15,-3.3);
        \draw[] (-4.15,-2) rectangle (-2.85,-3.3);
        \draw[] (-6.65,-2) rectangle (-5.35,-3.3);
        \node at (-6,0) {$2$};
        \node at (-3.5,0) {$4$};
        \node at (1.5,0) {$2n+4$};
        \node at (4,0) {$2n-2 $};
        \node at (6.5,0) {$2n $};
         \node at (-6,-2.65) {$2$};
        \node at (-3.5,-2.65) {$2$};
        \node at (1.5,-2.65) {$2$};
        \node at (4,-2.65) {$2$};
    \end{tikzpicture}
    \end{center}
with the quiver data $\{\xi_{i=1,\cdots,n},\mathsf v_{i=1,\cdots,n}=2i,\mathsf w_{i=1,\cdots,n-1}=2\}$ with $\mathsf v_n:=\mathsf w'_{n-1}=2n$.

The $QQ$-systems associated with the above quivers for $C_n$ and $D_n$ models read
\begin{align}
 \xi_{i+1} Q^+_i(qz)Q^-_i(z)-\xi_i Q^+_i(z)Q^-_i(qz)=Q^\circ_i(z)Q^+_{i+1}(z)Q^+_{i-1}(qz)\,,\label{eq:CDQQb4folding}
\end{align}
for $i=1,\cdots,n-1$ with the boundary conditions $Q^+_0(z):=1$ and $Q^+_n(z):=\Lambda(z)$. 

The $Q^+_i(z)$, $\Lambda(z)$, and $Q^\circ_i(z)$ are monic polynomials of degree $\mathsf v_i=2i$, $\mathsf v_n=2n$, and $\mathsf w_i$ respectively. They are explicitly given by
\begin{align}
    Q^+_{i=1,\cdots n}(z)=\prod_{k=1}^{2i}(z-s^+_{i,k})\,,&& Q^\circ_{i=1,\cdots n-1}(z)=
    \left\{\begin{array}{l}
      1 \\
      (z-s^\circ_{i,1})(z-s^\circ_{i,2})
    \end{array}\right.&&\begin{split}
        (C_n)\\(D_n)
    \end{split}\,.
\end{align}
with $s^+_{n,k}:=a_k$ for $\Lambda(z)$ and each $Q^-_i(z)$ is a monic polynomial of degree $\mathsf v^-_i=\mathsf v_{i-1}+\mathsf v_{i+1}-\mathsf v_i$.
\begin{align}
    Q^-_i(z)=  \mathcal{D}_i(z)\prod_{k=1}^{2i}(z-s^-_{i,k})\,,
&&
    \mathcal{D}_i(z)= \left\{\begin{array}{lr}
      1   &  (C_n)\\
  (z-\gamma^-_{i,1})(z-\gamma^-_{i,2})   & (D_n)
    \end{array}\right.\,.
\end{align}

The tRS momenta \eqref{eq:right-momentum-g} are given by 
\begin{equation} 
    p_i= \xi_{n}\frac{Q^{+}_{n-1}(qa_i)}{Q^{+}_{n-1}(a_i)}\,,\quad i=1,\dots,n\,.
\end{equation}
By the $C_n$ and $D_n$ folding conditions, the canonical tRS variables are folded as
\begin{align}
  p_i p_{2n+1-i}=\xi_{n}^2\frac{Q^{+}_{n-1}(qa_i)}{Q^{+}_{n-1}(a_i)}\frac{Q^{+}_{n-1}(qa_{i}^{-1})}{Q^{+}_{n-1}(a_{i}^{-1})}=1,&&   a_ia_{2n+1-i}=1,&&\quad i=1,\cdots ,n.
\end{align}

\subsection{Folding and Reflection-Invariance}
\begin{Def}\label{def:CDfolding}
The folded $QQ$-systems dual to $C_n$ and $D_n$ are the equivalent $QQ$-systems 
\begin{align}
     \widetilde \zeta_{i+1}(z) \widetilde Q^+_i(qz) \widetilde Q^-_i(z)-  \zeta_i(z) \widetilde Q^+_i(z)\widetilde Q^-_i(qz)=q^{-1}\widetilde  Q^+_{i+1}(z) \widetilde Q^+_{i-1}(qz)\,, \label{eq:C/DQQ}
\end{align}
for $i=1,\cdots,n-1$, which consists of the Laurent polynomials $\widetilde Q^\pm(z)$ induced from the $A_{n-1}$ $QQ$-systems by imposing the folding conditions on the Bethe roots with the boundary condition $s^+_{n,k}:=a_k$:
\begin{align}
 s^\pm_{i,k}s^\pm_{i,2i+1-k}=q^{n-i}, &&k=1,\cdots,i.
\end{align}
 The folded $Q$ functions are the Laurent polynomials in  $\mathbb{C}\left[z+\frac{q^{n-i}}{z}\right]$
\begin{align}
\begin{split}
\widetilde Q^+_{i}(z):=\left.\frac{Q^\pm_{i}(z)}{(-z)^{i}}\right|_{\textrm{folded}} =\a^{+}_i \prod_{k=1}^{i}(z-s^+_{i,k})\left(\frac{q^{n-i}}{z}-s^+_{i,k}\right)\,, 
\\
\widetilde Q^-_{i}(z):=\left.\frac{Q^-_{i}(z)}{(-z)^{i}\mathcal{D}_i (z)}\right|_{\textrm{folded}} =\a^-_i \prod_{k=1}^{i}(z-s^-_{i,k})\left(\frac{q^{n-i}}{z}-s^-_{i,k}\right)\,,      
\end{split}
\label{eq:CDnQ}
\end{align}
with $\widetilde Q^+_{n}(z):=\widetilde\Lambda(z)$ and $\widetilde Q^-_{n}(z):=1$.
The twists are generalized to the rational functions on $\P^1$ as
\begin{align}
  \widetilde \zeta_{i+1}(z)=\xi_{i+1}\frac{ \mathcal{D}_i(z)}{Q^\circ_i(z)}\,,&&  \zeta_{i}(z)=\xi_{i}\frac{ \mathcal{D}_i(qz)}{Q^\circ_i(z)}\,.
\end{align}
They are explicitly given as follows:
\begin{itemize}
    \item For the $C_n$ model, 
\begin{align}\label{eq:Ctwists}
      \widetilde \zeta_{i+1}(z)=\xi_{i+1}\,,&&  \zeta_{i}(z)=\xi_i\,.
\end{align}
    \item For the $D_n$ model, 
\begin{align}\label{eq:Dtwists}
  \widetilde \zeta_{i+1}(z)=\xi_{i+1}\frac{ (z-\gamma^-_{i,1})(z-\gamma^-_{i,2})}{(z-\gamma^\circ_i)(z+\gamma^\circ_i)}\,,&&  \zeta_{i}(z)=\xi_i\frac{(qz-\gamma^-_{i,1})(qz-\gamma^-_{i,2})}{( z-\gamma^\circ_i)(z+\gamma^\circ_i)}\,,
\end{align}
where the folding conditions $s^\circ_{i,2}+s^\circ_{i,1}=0$ on the roots of $Q^\circ(z)$ are imposed. 
\end{itemize}
\end{Def}

\begin{Thm}\label{Lem:CDfolded}
The folded $QQ$-systems of Definition \ref{def:CDfolding} exist for $C_n$ and $D_n$ models. Moreover, the generalized twists must satisfy the following reflection-invariant conditions
\begin{align}
   \widetilde\zeta_{i+1}(z) =- \zeta_i\left(\frac{q^{n-1-i}}{z}\right)\,.\label{eq:CDRIcond}
\end{align}
\end{Thm}
\begin{proof}
Divided by $z^{2i}Q^\circ(z)$, the $QQ$-systems of \eqref{eq:CDQQb4folding} become 
\begin{align}
\xi_{i+1}\frac{ \mathcal{D}_i(z)}{Q^\circ(z)}\frac{Q^+_{i}(qz)}{(-qz)^i}\frac{Q^-_{i}(z)}{(-z)^i \mathcal{D}_i(z)}-\xi_{i} \frac{\mathcal{D}_i(qz)}{Q^\circ(z)}\frac{Q^+_{i}(z)}{(-z)^i}\frac{Q^-_{i}(qz)}{(-qz)^i \mathcal{D}_i(qz)}  =q^{-1} \frac{Q^+_{i+1}(z)}{(-z)^{i+1}}  \frac{Q^+_{i-1}(qz)}{(-qz)^{i-1}}.
\end{align}
By imposing the folding conditions on the Bethe roots $s^{\bullet}$ for $\bullet\in\{+,-,\circ\}$ and $s^\circ_{i,1}=\gamma^\circ_i$, we obtain the folded $QQ$-system of Definition \ref{def:CDfolding} for both $C_n$ and $D_n$ models. By Lemma \ref{lem:generalRItwists} with $r=n-1$, \eqref{eq:CDRIcond} is obtained.
\end{proof}

The folding of the Bethe roots for $\widetilde Q^+_n=\widetilde \Lambda$ corresponds to the folding condition of the tRS positions
\begin{align}
   s^+_{n,k}s^+_{n,2n+1-k}= a_{k}a_{2n+1-k}=1,&&k=1,\cdots,n.
\end{align}
In terms of $\widetilde Q^+_{n-1}$, the tRS momenta are
\begin{equation} 
    p_i= \xi_{n}\frac{Q^{+}_{n-1}(qa_i)}{Q^{+}_{n-1}(a_i)}=q^{n-1}\xi_{n}\frac{\widetilde Q^{+}_{n-1}(qa_i)}{\widetilde Q^{+}_{n-1}(a_i)}\,,\quad i=1,\dots,n\,.
\end{equation}
The folding conditions for the tRS momenta are
\begin{align}
  p_k p_{2n+1-k}=q^{2n-2}\xi_{n}^2\frac{\widetilde Q^{+}_{n-1}(qa_k)}{\widetilde Q^{+}_{n-1}(a_k)}\frac{\widetilde Q^{+}_{n-1}(qa_{i}^{-1})}{\widetilde Q^{+}_{n-1}(a_{k}^{-1})}=q^{2n-2}\xi_{n}^2=1\,,
\end{align}
for $k=1,\cdots ,n$. Thus, $\xi_{n}=\pm q^{1-n}$.

As a corollary to Lemma \ref{Lem:Gen_RI_Sols}, 
\begin{Cor}\label{Cor:C/D}
The reflection-invariant conditions for the generalized twists have the following solutions:
\begin{itemize}
    \item For the $C_n$ model, 
    \begin{align}
        \frac{\xi_{i+1}}{\xi_i}=-1\,. \label{eq:C_nRIsol}
    \end{align}
    \item For the $D_n$ model, 
    \begin{align}
        \frac{\xi_{i+1}}{\xi_i}=\frac{\gamma^-_{i,1}\gamma^-_{i,2}}{(\gamma^\circ_{i})^2}\,,&&(\gamma^\circ_i)^2=\pm q^{n-1-i},&& \gamma^-_{i,k}\gamma^-_{i,k'}=q^{n-i}, &&k,k'=1,2.\label{eq:D_nRIsol}
    \end{align}

\end{itemize}
\end{Cor}

\subsection{Bethe Equations}
The Bethe equations can be obtained by eliminating $\widetilde Q^-_i$ from the $QQ$-system. The general form of the Bethe equations dual to $C_n$ and $D_n$ tRS models is
\begin{align}
    -\frac{ \widetilde \zeta_{i+1}(s^+_{i,k}) }{\zeta_i(q^{-1}s^+_{i,k})} \frac{\widetilde Q^+_i(qs^+_{i,k})}{\widetilde Q^+_i(q^{-1}s^+_{i,k})} \frac{\widetilde  Q^+_{i+1}(q^{-1}s^+_{i,k})}{\widetilde  Q^+_{i+1}(s^+_{i,k})}\frac{\widetilde Q^+_{i-1}(s^+_{i,k})}{\widetilde Q^+_{i-1}(qs^+_{i,k})} =1\,.
\end{align}
\begin{itemize}
    \item The Bethe equations dual to the $C_n$ model with \eqref{eq:C_nRIsol} are
\begin{align}
&-q^{-2i}\frac{\left(q^{n-i}-q(s^+_{i,k})^2\right)}{\left(q^{n-i+1}-(s^+_{i,k})^2\right)} \prod_{\stackrel{j=1}{j\neq k}}^{i} \frac{(qs^+_{i,k}-s^+_{i,j})\left(q^{n-i}-qs^+_{i,j}s^+_{i,k}\right)}{(s^+_{i,k}-qs^+_{i,j})\left(q^{n-i+1}-s^+_{i,k} s^+_{i,j}\right)} 
\\
\times  \prod_{j=1}^{i+1}&\frac{(s^+_{i,k}-qs^+_{i+1,j})\left(q^{n-i}- s^+_{i,k}s^+_{i+1,j}\right)}
{(s^+_{i,k}-s^+_{i+1,j})\left(q^{n-i-1}-s^+_{i,k}s^+_{i+1,j}\right)}\prod_{j=1}^{i-1} \frac{(s^+_{i,k}-s^+_{i-1,j})\left( q^{n-i+1} -s^+_{i,k}s^+_{i-1,j}\right)}{(qs^+_{i,k}-s^+_{i-1,j})\left( q^{n-i} -s^+_{i,k}s^+_{i-1,j}\right)} =1\, ,\nonumber
\end{align}
for $k=1,\cdots,i$ and $i=1,\cdots,n-1$.

\end{itemize}

For the $D_n$ model, let us fix the signs of $\gamma$s by $(\gamma^\circ_i)^2=q^{n-1-i}$ and $\gamma^-_{i,1}\gamma^-_{i,2}=q^{n-i}$. 
The ratio of the generalized twists after the sign fixing is
\begin{align}
\frac{ \widetilde \zeta_{i+1}(s^+_{i,k}) }{\zeta_i(q^{-1}s^+_{i,k})}=q\frac{(q^{-1}s^+_{i,k}-\gamma^\circ_i)(q^{-1}s^+_{i,k}+\gamma^\circ_i)}{(s^+_{i,k}-\gamma^\circ_i)(s^+_{i,k}+\gamma^\circ_i)}=\frac{q^{-1}\left((s^+_{i,k})^2-q^{n+1-i}\right)}{\left((s^+_{i,k})^2-q^{n-1-i}\right)}\,.
\end{align}
This exactly cancels the contact term that appears in the Bethe equations for $C_n$ model.  
\begin{itemize}
    \item The Bethe equations dual to the $D_n$ model with \eqref{eq:D_nRIsol} are
\begin{align}
q^{-2i}\prod_{\stackrel{j=1}{j\neq k}}^{i} \frac{(qs^+_{i,k}-s^+_{i,j})\left(q^{n-i}-qs^+_{i,j}s^+_{i,k}\right)}{(s^+_{i,k}-qs^+_{i,j})\left(q^{n-i+1}-s^+_{i,k} s^+_{i,j}\right)} \prod_{j=1}^{i+1}\frac{(s^+_{i,k}-qs^+_{i+1,j})\left(q^{n-i}- s^+_{i,k}s^+_{i+1,j}\right)}
{(s^+_{i,k}-s^+_{i+1,j})\left(q^{n-i-1}-s^+_{i,k}s^+_{i+1,j}\right)}
\end{align}
\[\times\prod_{j=1}^{i-1} \frac{(s^+_{i,k}-s^+_{i-1,j})\left( q^{n-i+1} -s^+_{i,k}s^+_{i-1,j}\right)}{(qs^+_{i,k}-s^+_{i-1,j})\left( q^{n-i} -s^+_{i,k}s^+_{i-1,j}\right)} =1\, ,\]
for $k=1,\cdots,i$ and $i=1,\cdots,n-1$.
\end{itemize}
\begin{Thm}\label{thm:C/DBAE}
The above Bethe equations are dual to $C_n$ and $D_n$ tRS models.
\end{Thm}
\begin{proof}
To compare with tRS boundary factors, it is convenient to write the Bethe equations in terms of  $f_q(x)=\frac{x-q}{qx-1}$ and $ g_q(x)=\frac{x-q}{x-1}$.

Recall the rescaled Bethe roots
\begin{align}
    u_{i,j}=q^{-(n-i)/2}s^+_{i,j}\,.
\end{align}

\begin{itemize}
    \item For the $C_n$ model, the Bethe equations are 
    \begin{align}\label{eq:BetheCn}
&-q^{-2i}f^{-1}_q(u_{i,k}^{2}) \prod_{\stackrel{j=1}{j\neq k}}^{i} f^{-1}_q\left(\frac{u_{i,k}}{u_{i,j}}\right)f^{-1}_q(u_{i,k} u_{i,j})
\\
\times  \prod_{j=1}^{i+1}& g_q\left(q^{ \frac{1}{2}}\frac{u_{i,k}}{u_{i+1,j}}\right)g_q( q^{ \frac{1}{2}}u_{i,k}u_{i+1,j}) \prod_{j=1}^{i-1}g^{-1}_{q}\left(q^{-\frac{1}{2}}\frac{u_{i-1,j}}{u_{i,k}}\right)g_q(q^{-\frac{1}{2}}u_{i,k}u_{i-1,j}) =1\, .\nonumber
\end{align}

\item For the $D_n$ model, the Bethe equations are
\begin{align}\label{eq:BetheDn}
&q^{-2i}\prod_{\stackrel{j=1}{j\neq k}}^{i} f^{-1}_q\left(\frac{u_{i,k}}{u_{i,j}}\right)f^{-1}_q(u_{i,k} u_{i,j})
\\
\times  \prod_{j=1}^{i+1}& g_q\left(q^{ \frac{1}{2}}\frac{u_{i,k}}{u_{i+1,j}}\right)g_q( q^{ \frac{1}{2}}u_{i,k}u_{i+1,j})  \prod_{j=1}^{i-1}\left(q^{-\frac{1}{2}}\frac{u_{i-1,j}}{u_{i,k}}\right)g_q(q^{-\frac{1}{2}}u_{i,k}u_{i-1,j}) =1\, .\nonumber
\end{align}
\end{itemize}
\end{proof}

\subsection{Gluing Isomorphisms}
The gluing isomorphism \eqref{gluing} in type $C_n$ is given by 
\begin{equation}\label{eq:Cgluing}
B_{i}(z)=-1,\quad u_{i}\in\frac{1}{2}\mathbb{Z}\text{ is arbitrary}.
\end{equation}
The gluing isomorphism \eqref{gluing} in type $D_n$ is given by the following rational function.
\begin{equation}\label{eq:Dgluing}
B_i(z)
=
\sigma_i q\,
\frac{
(qz-\gamma_{i+1,1}^{-})
(qz-\gamma_{i+1,2}^{-})
\left(1-\sigma_i q^{\,n+i+1-2u_i}z^2\right)
}{
\left(z^2-\sigma_{i+1}q^{\,n-i-2}\right)
\left(q-q^{\,i+1-u_i}\gamma_{i,1}^{-}z\right)
\left(q-q^{\,i+1-u_i}\gamma_{i,2}^{-}z\right)
},
\end{equation}
where $\sigma_i\in\{\pm1\}$ and $u_{i}\in\frac{1}{2}\mathbb{Z}$ is arbitrary. The simplest gluing in this case is when $u_{i}=n-\frac{1}{2}$ and $\sigma_i=\sigma_{i+1}$, and we have 
\begin{align}
B_i(z)
=
-q^{\,i-n+3}
\frac{
\left(qz-\gamma_{i+1,1}^{-}\right)
\left(qz-\gamma_{i+1,2}^{-}\right)
}{
\left(q-q^{\,i-n+\frac{3}{2}}\gamma_{i,1}^{-}z\right)
\left(q-q^{\,i-n+\frac{3}{2}}\gamma_{i,2}^{-}z\right)
}.
\end{align}
\begin{Nota}
    Let us denote the singularities $\Lambda_{1}(z),\ldots,\Lambda_{n-1}(z)$ of a $(GL(n),q)$-oper by a $(n-1)$-tuple $\bm{\Lambda}(z)$ of rational functions, that is, 
    \begin{align}
        \bm{\Lambda}(z)=\left(\Lambda_{1}(z),\ldots,\Lambda_{n-1}(z)\right).
    \end{align}
\end{Nota}
The following theorem is an immediate consequence of Theorems \ref{thm:opersandbethe}, \ref{thm:betheandtrs}, \ref{Lem:CDfolded} and \ref{thm:C/DBAE}.
\begin{Thm}\label{thm:CD} The assignment
\begin{equation} 
    p_i= \xi_{n}\frac{Q^{+}_{n-1}(qa_i)}{Q^{+}_{n-1}(a_i)}\,,\quad i=1,\dots,n\,
\end{equation}
induces an isomorphism of algebras:
\begin{align}
    \frac{\mathbb{C}\left[q^{\pm 1},c_i^{\pm 1},\{\alpha_{i,k}^{\pm 1}\},\{\beta_{i,k}^{\pm 1}\}\right]\otimes_{\mathbb{C}} S\left(\{u_{i,j}\},\{t_{i,k}\}\right)^{\mathbb{Z}/{2\mathbb{Z}}}}{(\text{Gluing}^R,\text{Bethe}^R)}\cong
    \frac{
  \mathbb{C}(\{\xi_i\},\{a_i\},q)(\{p_i\})
}{
  \left(
    \det\left(
      u-T^R(\{p_i\},\{a_i\},q)
    \right)
    -F^R(u,\{\xi_i\})
  \right).
}
\end{align}
where $R=C_n, D_n$, $\text{Bethe}^{C_n}$ (resp. $\text{Bethe}^{D_n}$) is given by \eqref{eq:BetheCn} (resp. \eqref{eq:BetheDn}) and $\text{Gluing}^R$ are the gluing equations \eqref{eq:gluingconditionontwists} where the functions $B_i(z)$ are given by \eqref{eq:Cgluing} and \eqref{eq:Dgluing} when $R=C_n$ and $R=D_n$ respectively. Moreover, the above isomorphism is compatible with the following one-to-one correspondence among the four classes of objects:
    \begin{enumerate}[label=\textnormal{(\roman*)}]
        \item the set of isomorphism classes of nondegenerate generalized $Z$-twisted reflection-invariant Miura-Pl\"ucker $(GL(n),q)$-opers with twists given by \eqref{eq:Ctwists} and \eqref{eq:Dtwists} and regular singularities are given by $\bm{\Lambda(z)}=\left(1,\ldots,1,\tilde{\Lambda}(z)\right)$ with gluing data \eqref{eq:Cgluing} and \eqref{eq:Dgluing} respectively,
        \item the set of nondegenerate solutions of the open XXZ Bethe equations \eqref{eq:BetheCn} and \eqref{eq:BetheDn} with boundary contributions in Table \ref{tab:Summary} corresponding to the $C_n$ and $D_n$ rows respectively,
        \item the energy level set of the $tRS$ model with root systems $C_n$ and $D_n$ respectively, and
        \item the set of nondegenerate solutions of folded $QQ$-system with twists given by \eqref{eq:Ctwists} and \eqref{eq:Dtwists}  dual to the $C_n$ and $D_n$ $tRS$ models respectively.
    \end{enumerate}
\end{Thm}

We summarize the results obtained in this section in the following table.

\newpage

\begingroup

\dimen0=\pdfpagewidth
\pdfpagewidth=\pdfpageheight
\pdfpageheight=\dimen0

\setlength{\oddsidemargin}{0.5in}
\setlength{\evensidemargin}{0.5in}
\setlength{\textwidth}{10in}
\setlength{\textheight}{6in}
\setlength{\columnwidth}{\textwidth}
\setlength{\linewidth}{\textwidth}
\setlength{\hsize}{\textwidth}
\setlength{\vsize}{\textheight}

\makeatletter
\setlength{\@colht}{\textheight}
\setlength{\@colroom}{\textheight}
\makeatother

\begin{table}[p]
\thispagestyle{plain}
\centering
\small
\setlength{\tabcolsep}{3pt}
\renewcommand{\arraystretch}{1.35}
\setlength{\abovecaptionskip}{12pt}

\begin{tabular}
{|>{\centering\arraybackslash}p{1.5cm}|
  >{\centering\arraybackslash}p{4.0cm}|
  >{\centering\arraybackslash}p{7.5cm}|
  >{\centering\arraybackslash}p{5.3cm}|
  >{\centering\arraybackslash}p{5.3cm}|}
\hline

\textbf{Root system}
&
\textbf{tRS potential}
&
\textbf{Twists, singularities, and simplest gluing data of $q$-opers}
&
\textbf{Twists of folded $QQ$-systems}
&
\textbf{Boundary contribution in open XXZ chains}
\newline
$\displaystyle
\begin{aligned}\frac{\widetilde{\zeta}_{i+1}(z)}
     {\zeta_i(q^{-1}z)}\\[2mm]\end{aligned}$
\\
\hline

$C_n$
&
$\displaystyle
\begin{aligned}
&\prod_{j\ne i}
 g_q\!\left(\frac{a_i}{a_j}\right)
 g_q(a_i a_j)
\\[-1mm]
&\qquad{}\cdot g_q(a_i^2)
\end{aligned}$
&
$\displaystyle
\begin{aligned}\\
\zeta_i(z)&=(-1)^i q^{\,1-n},
\\
\bm{\Lambda}(z)&={}\left(1,\ldots,1,\widetilde{\Lambda}(z)\right),
\\
u_i&\in\frac12\mathbb Z
\quad\text{arbitrary},
\\
B_i(z)&=-1
\end{aligned}$
&
$\displaystyle
\zeta_i(z)=(-1)^i q^{\,1-n}$
&
$-1$
\\
\hline

$D_n$
&
$\displaystyle
\prod_{j\ne i}
g_q\!\left(\frac{a_i}{a_j}\right)
g_q(a_i a_j)$
&
$\displaystyle
\begin{aligned}\\
\zeta_i(z)
={}&\xi_i
\frac{qz-\gamma_{i,1}^-}
     {z-\gamma_i^\circ}\cdot
\frac{qz-\gamma_{i,2}^-}
     {z+\gamma_i^\circ}
\\[2mm]
\bm{\Lambda}(z)={}&\left(1,\ldots,1,\widetilde{\Lambda}(z)\right),
\\[2mm]
u_i={}&n-\frac12,
\\[2mm]
B_i(z)
={}&-q^{\,i-n+3}
\frac{qz-\gamma_{i+1,1}^-}
     {q-q^{\,i-n+\frac32}\gamma_{i,1}^-z}
\\[-1mm]
&\quad{}\cdot
\frac{qz-\gamma_{i+1,2}^-}
     {q-q^{\,i-n+\frac32}\gamma_{i,2}^-z}
     \\[2mm]
\end{aligned}$
&
$\displaystyle
\begin{aligned}
\zeta_i(z)
={}&\xi_i
\frac{qz-\gamma_{i,1}^-}
     {z-\gamma_i^\circ}\cdot
\frac{qz-\gamma_{i,2}^-}
     {z+\gamma_i^\circ}
\end{aligned}$
&
$\displaystyle
\begin{aligned}
-\frac{\gamma_{i,1}^-}
        {q\gamma_i^\circ}
 \frac{z-\gamma_{i,2}^-}
      {z-\sigma_i\gamma_i^\circ}
\cdot
 \frac{z-q\gamma_i^\circ}
      {z-\gamma_{i,1}^-}
\end{aligned}$
\\
\hline

\end{tabular}

\par\vspace{8pt}

\caption{Correspondence between tRS models, $q$-opers,
$QQ$-systems, and open XXZ spin chains.}
\label{tab:Summary}
\end{table}

\clearpage
\endgroup
\setlength{\textheight}{8.5in}
\setlength{\vsize}{\textheight}
\makeatletter
\setlength{\@colht}{\textheight}
\setlength{\@colroom}{\textheight}
\makeatother
\section{Folding for \texorpdfstring{$B$}{B} and \texorpdfstring{$BC$}{BC}-Type Systems}\label{sec:BCB}
In this section, we present the folding of the $QQ$-systems which lead to tRS systems of $BC$ and $B$-type. The $A_{n-1}$ quivers for $BC_n$ and $B_n$ are characterized by $\mathsf v_i=2i+1$ for $1\leq i \leq n$, where  $\mathsf v_{0}:=\mathsf w'_{1}=1$ and $\mathsf v_{n}:=\mathsf w'_{n-1}=2n+1$ for the convenience of dealing with both types in a unified way.

For $B_n$ model, we introduce extra framings of $\mathsf w_i=2$ to the quiver for $BC_n$ model. The corresponding $Q$ functions are designed to cancel the factor $f_q(q^{-(n-i)}(s^+_{i,k})^2)$ in the Bethe equations of the $BC_n$ model. This leads to Definition \ref{def:BCBfolding}. 

Eventually, the definitions reproduce the correct Bethe equations of tRS systems for $B$ and $BC$ by the standard derivation. Therefore, we formulate Theorem \ref{thm:BCB}.

\subsection{$QQ$-Systems Associated with $A_{n-1}$ Quivers}
The $QQ$-systems before folding for $BC_n$ and $B_n$ tRS models are associated with the ``odd'' type $A$ quivers. For the $BC_n$ model, we consider the following balanced $A_{n-1}$ quiver

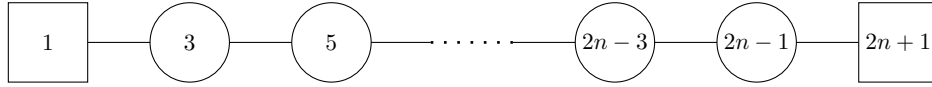
\begin{center}
        \begin{tikzpicture}[scale=.75, every node/.style={scale=0.8}]
        \draw[] (-9.2,-0.7) rectangle (-7.8,0.7);
        \draw[] (-7.8,0) -- (-6.7,0);
        \draw[] (-6,0) circle (.7);
        \draw[] (-5.3,0) -- (-4.2,0);
        \draw[] (-3.5,0) circle (.7);
        \draw[] (-2.8,0) -- (-1.75,0);
        \draw[line width=1pt,loosely dotted] (-1.75,0) -- (-.3,0);
        \draw[] (-.3,0) -- (.8,0);
        \draw[] (1.5,0) circle (.7);   
        \draw[] (2.2,0) -- (3.3,0);
        \draw[] (4,0) circle (.7);      
        \draw[] (4.7,0) -- (5.8,0);
        \draw[] (5.8,-0.7) rectangle (7.2,0.7);

        \node at (-8.5,0) {$1$};
        \node at (-6,0) {$3$};
        \node at (-3.5,0) {$5$};
        \node at (1.5,0) {$2n-3$};
        \node at (4,0) {$2n-1$};
        \node at (6.5,0) {$2n +1$};

    \end{tikzpicture}
    \end{center}

with the quiver data $\{\xi_{i=1,\cdots,n},\mathsf v_{i=0,\cdots,n}=2i+1\}$ with $\mathsf v_{0}:=\mathsf w'_{1}=1$ and  $\mathsf v_{n}:=\mathsf w'_{n-1}=2n+1$.

\vskip.1in
On the other hand, for the $B_n$ model, we add framings to the $A_{n}$ quiver dual to the $BC_n$ model as for the $D_n$ model.

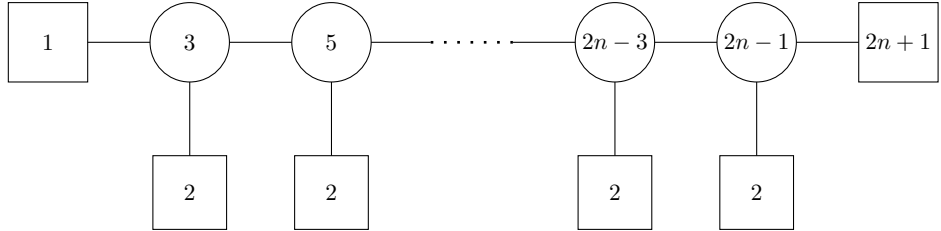
\begin{center}
\begin{tikzpicture}[scale=.75, every node/.style={scale=0.8}]
        \draw[] (-9.2,-0.7) rectangle (-7.8,0.7);
        \draw[] (-7.8,0) -- (-6.7,0);
        \draw[] (-6,0) circle (.7);
        \draw[] (-5.3,0) -- (-4.2,0);
        \draw[] (-3.5,0) circle (.7);
        \draw[] (-2.8,0) -- (-1.75,0);
        \draw[line width=1pt,loosely dotted] (-1.75,0) -- (-.3,0);
        \draw[] (-.3,0) -- (.8,0);
        \draw[] (1.5,0) circle (.7);   
        \draw[] (2.2,0) -- (3.3,0);
        \draw[] (4,0) circle (.7);      
        \draw[] (4.7,0) -- (5.8,0);
        \draw[] (5.8,-0.7) rectangle (7.2,0.7);
        \draw[] (4,-.7) -- (4,-2);
        \draw[] (1.5,-.7) -- (1.5,-2);
        \draw[] (-3.5,-.7) -- (-3.5,-2);
        \draw[] (-6,-.7) -- (-6,-2);   
        \draw[] (3.35,-2) rectangle (4.65,-3.3);
        \draw[] (0.85,-2) rectangle (2.15,-3.3);
        \draw[] (-4.15,-2) rectangle (-2.85,-3.3);
        \draw[] (-6.65,-2) rectangle (-5.35,-3.3);
  
        \node at (-8.5,0) {$1$};
        \node at (-6,0) {$3$};
        \node at (-3.5,0) {$5$};
        \node at (1.5,0) {$2n-3$};
        \node at (4,0) {$2n-1$};
        \node at (6.5,0) {$2n +1$};

        \node at (-6,-2.65) {$2$};
        \node at (-3.5,-2.65) {$2$};
        \node at (1.5,-2.65) {$2$};
        \node at (4,-2.65) {$2$};
    \end{tikzpicture}
    \end{center}

with the quiver data $\{\xi_{i=1,\cdots,n}, \mathsf v_{i=0,\cdots,n}=2i+1,\mathsf w_{i=1,\cdots,n-1}=2\}$ with $\mathsf v_{0}:=\mathsf w'_{1}=1$ and $\mathsf v_{n}:=\mathsf w'_{n-1}=2n+1$.

The $QQ$-systems associated with the above quivers for the $BC_n$ and $B_n$ models read
\begin{align}\label{eq:BCBQQb4folding}
 \xi_{i+1} Q^+_i(qz)Q^-_i(z)-\xi_i Q^+_i(z)Q^-_i(qz)=Q^\circ_i(z)Q^+_{i+1}(z)Q^+_{i-1}(qz)\,,
\end{align}
for $i=1,\cdots,n-1$ with the boundary conditions $Q^+_0(z):=(z-\gamma^+_0)$ and $Q^+_{n}(z):=\Lambda(z)$. 

The $Q^+_i(z)$, $\Lambda(z)$, and $Q^\circ_i(z)$ are the monic polynomials of degrees $\mathsf v_i=2i+1$, $\mathsf v_{n}=2n+1$, and $\mathsf w_i=2$ respectively. They are explicitly given by
\begin{align}
    Q^+_{i=1,\cdots n}(z)=(z-\gamma^+_i)\prod_{k=1}^{2i}(z-s^+_{i,k})\,,
    \end{align}
    \begin{align}Q^\circ_{i=1,\cdots n-1}(z)=
    \left\{\begin{array}{l}
      1 \\
      (z-s^\circ_{i,1})(z-s^\circ_{i,2})
    \end{array}\right.&&\begin{split}
        (BC_n)\\(B_n)
    \end{split}\,.
\end{align}
with $s^+_{n+1,k}:=a_k$ for $\Lambda(z)$. $Q^-_i(z)$ are the monic polynomial of degree $\mathsf v^-_i=\mathsf v_{i-1}+\mathsf v_{i+1}+\mathsf w_i-\mathsf v_i$. $\mathsf v^-_i=2i+1$ for $BC_n$ and $\mathsf v^-_i=2i+3$ for $B_n$, respectively.
They are explicitly given by
\begin{align}
    Q^-_{i=1,\cdots,n-1}(z)=  (z-\gamma^-_i)\mathcal{B}_i(z)\prod_{k=1}^{2i}(z-s^-_{i,k})\,,
\end{align}
\begin{align}
    \mathcal{B}_i(z)= \left\{\begin{array}{lr}
      1   &  (BC_n)\\
  (z-\gamma^-_{i,1})(z-\gamma^-_{i,2})   & (B_n)
    \end{array}\right.\,.
\end{align}
Note that $Q^+_{i+1}(z)Q^+_{i-1}(qz)$ has degree $2(2i+1)$.

The folding conditions of the $BC$/$B$-type tRS models are
\begin{align}
    a_{i}a_{2n+1-i}=1,&&a_{n+1}:=\gamma^+_{n}=1\,.
\end{align}
for the position variables and 
\begin{align}
    p_ip_{2n+1-i}=1,&&p_{n+1}=1\,,
\end{align}
for the momentum variables where
\begin{align}
    p_i=\xi_{n}\frac{Q^+_{n-1}(qa_i)}{Q^+_{n-1}(a_i)}\,.
\end{align}

\subsection{Folding and Reflection-Invariance}
\begin{Def}\label{def:BCBfolding}
The folded $QQ$-systems dual to $BC_n$ and $B_n$ are the equivalent $QQ$-systems 
\begin{align}
     \widetilde \zeta_{i+1}(z) \widetilde Q^+_i(qz) \widetilde Q^-_i(z)-  \zeta_i(z) \widetilde Q^+_i(z)\widetilde Q^-_i(qz)=q^{-1}    \widetilde \L_{i+1}(z)\widetilde Q^+_{i-1}(qz)\,, \label{eq:BC/BQQ}
\end{align}
for $i=1,\cdots,n-1$, which consists of the Laurent polynomials $\widetilde Q^\pm(z)$ induced from the $A_{n-1}$ $QQ$-systems by imposing the folding conditions on the Bethe roots with the boundary condition $s^+_{n,k}:=a_k$ for $k=1,\cdots,n$:
\begin{align}
 s^\pm_{i,k}s^\pm_{i,2i+1-k}=q^{n-i}, &&k=1,\cdots,i.
\end{align}
 The folded functions $Q_i$ are the Laurent polynomials in  $\mathbb{C}\left[z+\frac{q^{n-i}}{z}\right]$
\begin{align}
\widetilde Q^+_{i=0,\cdots,n}(z):=\left.\frac{Q^+_{i}(z)}{(z-\gamma^+_i)(-z)^{i}}\right|_{\textrm{folded}} =\a^+_i \prod_{k=1}^{i}(z-s^+_{i,k})\left(\frac{q^{n-i}}{z}-s^+_{i,k}\right)\,, \label{eq:BCBnQ}
\\
\widetilde Q^-_{i=1,\cdots,n-1}(z):=\left.\frac{Q^-_{i}(z)}{(z-\gamma^-_i)(-z)^{i}\mathcal{B}_i (z)}\right|_{\textrm{folded}} =\a^{-}_i \prod_{k=1}^{i}(z-s^-_{i,k})\left(\frac{q^{n-i}}{z}-s^-_{i,k}\right)\,.\nonumber
\end{align}

For the RHS, we define
\begin{align}
    \widetilde \L_{i+1}(z):=\left.\l_i\left(\frac{(z-\gamma^+_{i+1})(qz-\gamma^+_{i-1})}{(-qz)}\right) \right|_{\textrm{folded}}\widetilde Q^+_{i+1}(z)\,, 
\end{align}
where the folding condition reads $\gamma^+_{i+1}\gamma^+_{i-1}=q^{(n-i)}$ from
\begin{align}\label{foldingconditiononsingularieties}
(z-\gamma^+_{i+1})\left(\frac{q^{n-1-i}}{z}-\frac{q^{(n-i)}}{\gamma^+_{i-1}}\right)=\l_i (z-\gamma^+_{i+1}) \left(\frac{q^{n-1-i}}{z}-\gamma^+_{i+1}\right)    \,.
\end{align}
Note that $\{\lambda_{i+1,k}\}=\{\gamma^{+}_{i\pm1},s^+_{i+1,k=1,\cdots,i},s^+_{i-1,k=1,\cdots,i-2}/q\}$. 
$\widetilde \L_n(z)$ provides the folding of tRS positions as
\begin{align}
\widetilde \L_n(z)=(z-1)\left(\frac{1}{z}-1\right)\prod_{k=1}^{n}(z-a_k)\left(\frac{1}{z}-a_k\right)\,,
\\
\widetilde \L_{i=2,\cdots,n-1}(z)=\prod_{k=1}^{i+1}(z-\l_{i,k})\left(\frac{q^{n-i}}{z}-\l_{i,k}\right)\,.
\end{align}

The twists are generalized to the rational functions on $\P^1$ as
\begin{align}
  \widetilde \zeta_{i+1}(z)=-\xi_{i+1}\frac{(qz-\gamma^+_i)(z-\gamma^-_i)}{z}\frac{ \mathcal{B}_i(z)}{ Q^\circ_i(z) }\,,&&  \zeta_{i}(z)=-\xi_{i}\frac{(z-\gamma^+_i)(qz-\gamma^-_i)}{z}\frac{\mathcal{B}_i(qz)}{Q^\circ_i(z) }\,,
\end{align}
for $i=1,\cdots,n-1$.
They are explicitly given as follows:
\begin{itemize}
    \item For the $BC_n$ model, 
\begin{align}
  \widetilde \zeta_{i+1}(z)=-\xi_{i+1}\frac{(qz-\gamma^+_i)(z-\gamma^-_i)}{z}\,,&&  \zeta_{i}(z)=-\xi_{i}\frac{(z-\gamma^+_i)(qz-\gamma^-_i)}{z}\,.\label{eq:BCtwists}
\end{align}
    \item For the $B_n$ model, 
\begin{align}\label{eq:Btwists}
\begin{split}
  \widetilde \zeta_{i+1}(z)=-\xi_{i+1}\frac{(qz-\gamma^+_i)(z-\gamma^-_i)}{z}\frac{ (z-\gamma^-_{i,1})(z-\gamma^-_{i,2})}{(z-\gamma^\circ_i)(z+\gamma^\circ_i)}\,,
  \\
  \zeta_{i}(z)=-\xi_{i}\frac{(z-\gamma^+_i)(qz-\gamma^-_i)}{z}\frac{(qz-\gamma^-_{i,1})(qz-\gamma^-_{i,2})}{(z-\gamma^\circ_i)(z+\gamma^\circ_i)}\,,
\end{split}
\end{align}
where the folding conditions $s^\circ_{i,2}+s^\circ_{i,1}=0$ on the roots of $Q^\circ(z)$ are imposed. 
\end{itemize}
\end{Def}

\begin{Thm}\label{Lem:BCBfolded}
The folded $QQ$-systems of Definition \ref{def:BCBfolding} exist for $BC_n$ and $B_n$ models. Moreover, the generalized twists must satisfy the following reflection-invariant conditions
\begin{align}
   \widetilde\zeta_{i+1}(z) =- \zeta_i\left(\frac{q^{n-1-i}}{z}\right)\,.\label{eq:BCBRIcond}
\end{align}
\end{Thm}
\begin{proof}
Divided by $(-z)^{i}Q^\circ(z)$, the $QQ$-systems of \eqref{eq:CDQQb4folding} become 
\begin{align}
&\xi_{i+1}\frac{(qz-\gamma^+_i)(z-\gamma^-_i)}{(-z)}\frac{\mathcal{B}_i(z)}{Q^\circ(z)}\frac{Q^+_{i}(qz)}{(-qz)^{i}(qz-\gamma^+_i)}\frac{Q^-_{i}(z)}{(-z)^{i} (z-\gamma^-_i)\mathcal{B}_i(z)}
\\&
-\xi_{i}\frac{(z-\gamma^+_i)(qz-\gamma^-_i)}{(-z)} \frac{\mathcal{B}_i(qz)}{Q^\circ(z)}\frac{Q^+_{i}(z)}{(-z)^{i}(z-\gamma^+_i) }\frac{Q^-_{i}(qz)}{(-qz)^{i}(qz-\gamma^-_i) \mathcal{B}_i(qz)}\nonumber
\\
&=\frac{(z-\gamma^+_{i+1})(qz-\gamma^+_{i-1})}{(-qz)} \frac{Q^+_{i+1}(z)}{(-z)^{i}(z-\gamma^+_{i+1})} \frac{Q^+_{i-1}(qz)}{(-qz)^{i-1}(qz-\gamma^+_{i-1})}\,.\nonumber
\end{align}
By imposing the folding conditions on the Bethe roots $s^{\bullet}$ for $\bullet\in\{+,-,\circ\}$ and $s^\circ_{i,1}=\gamma^\circ_i$, we obtain the folded $QQ$-system of Definition \ref{def:BCBfolding} for both $BC_n$ and $B_n$ models. By Lemma \ref{lem:generalRItwists} with $r=n-1$, \eqref{eq:BCBRIcond} is obtained.
\end{proof}

For the folding of $BC$/$B$ tRS momenta, it suffices to solve the following condition
\begin{align}
    p_{n+1}=\xi_{n}\frac{Q^+_{n-1}(q)}{Q^+_{n-1}(1)}=1\,.
\end{align}
By Definition \ref{def:BCBfolding}, $\xi_{n}$ is determined by $\xi_{n}=q^{-(n-1)}\frac{(1-\gamma^+_{n-1})}{(q-\gamma^+_{n-1})}$ by
\begin{align}
p_{n+1}=q^{n-1}\xi_{n}\frac{(q-\gamma^+_{n-1})}{(1-\gamma^+_{n-1})}\frac{\widetilde Q^+_{n-1}(q )}{\widetilde Q^+_{n-1}(1)}=q^{n-1}\xi_{n}\frac{(q-\gamma^+_{n-1})}{(1-\gamma^+_{n-1})}=1\,.    
\end{align}

As a corollary to Lemma \ref{Lem:Gen_RI_Sols}, 
\begin{Cor}\label{Cor:BC/B}
The reflection-invariant generalized twists exist for both $BC_n$ and $B_n$ models.
 \begin{itemize}
    \item For the $BC_n$ model, $\b^\circ_i=a^\pm_i=1$,
    \begin{align}
    \frac{\xi_{i+1}}{\xi_i}=-q^{-(n-i)}(\gamma^+_{i}\gamma^-_{i})=\pm1\,,&& (\gamma^\pm_{i})^2=q^{n-i},\qquad\textrm{or}\qquad \gamma^+_{i}\gamma^-_{i}=q^{n-i}\, ,  \label{eq:BC_nRIsol}
    \end{align}
    
    \item For the $B_n$ model, $\b^\circ_i=\a^+_i=1,\a^\circ_i=2,\a^-_i=3$,
    \begin{align}
         \frac{\xi_{i+1}}{\xi_i}=q^{-(n-i)}(\gamma^+_{i}\gamma^-_{i})\frac{\gamma^-_{i,1}\gamma^-_{i,2}}{(\gamma^\circ_{i})^2}=\pm q\,,&& \gamma^-_{i,k}\gamma^-_{i,k'}=q^{n-i}, &&k,k'=1,2.\label{eq:B_nRIsol}
    \end{align}
    with
    \begin{align}
        (\gamma^\circ_i)^2=\pm q^{n-1-i}\,,&&(\gamma^\pm_i)^2=q^{n-i}\,,\qquad\textrm{or}\qquad \gamma^+_{i}\gamma^-_{i}=q^{n-i}\,,&&i=1,\cdots,n-1\,,
    \end{align}
\end{itemize}
\end{Cor}

\subsection{Bethe Equations}
The Bethe equations can be obtained by eliminating $\widetilde Q^-_i$ from the $QQ$-system. The general form of the Bethe equations dual to $BC_n$ and $B_n$ tRS models is
\begin{align}
   - \frac{ \widetilde \zeta_{i+1}(s^+_{i,k})}{\zeta_i(q^{-1}s^+_{i,k})}\frac{ \widetilde Q^+_i(qs^+_{i,k})}{\widetilde Q^+_i(q^{-1}s^+_{i,k})}\frac{\widetilde \L_{i+1}(q^{-1}s^+_{i,k})} {\widetilde \L_{i+1}(s^+_{i,k})}\frac{\widetilde Q^+_{i-1}(s^+_{i,k})}{\widetilde Q^+_{i-1}(qs^+_{i,k})}=1\,. 
\end{align}

\begin{itemize}
     \item For $BC_n$,
  \begin{align}
       \pm  q^{-2}\frac{(qs^+_{i,k}-\gamma^+_i)}{(s^+_{i,k}-q\gamma^+_i)} \frac{ \left(1-q^{-(n-i)+1}(s^+_{i,k} )^2\right)}{ \left(q-q^{-(n-i)}(s^+_{i,k})^2\right)} \prod_{\stackrel{j=1}{j\neq k}}^{i}\frac{(qs^+_{i,k}-s^+_{i,j})\left(q^{n-1-i}-s^+_{i,k} s^+_{i,j}\right)}{(s^+_{i,k}-qs^+_{i,j})\left(q^{n+1-i}-s^+_{i,k}s^+_{i,j}\right)} 
    \end{align}
    \[\times \prod_{j=1}^{i+1}\frac{(s^+_{i,k}-q\lambda_{i+1,j})\left(q^{n-i}- s^+_{i,k}\lambda_{i+1,j}\right)   }{(s^+_{i,k}-\lambda_{i+1,j})\left(q^{n-1-i}-s^+_{i,k}\lambda_{i+1,j}\right)}\prod_{j=1}^{i-1}\frac{(s^+_{i,k}-s^+_{i-1,j})\left(q^{n-i} -s^+_{i,k}s^+_{i-1,j}\right)}{(qs^+_{i,k}-s^+_{i-1,j})\left(q^{n+1-i}-s^+_{i,k}s^+_{i-1,j}\right)}=1\,,\]
 for $i=1,\cdots,n-1$ with $\gamma^+_{i}\gamma^-_{i}=\pm q^{n-i}$.

    \item For $B_n$, let us fix the signs of $\gamma$ by $(\gamma^\circ_i)^2=q^{n-1-i}$ and $\gamma^+_{i}\gamma^-_{i}=\gamma^-_{i,1}\gamma^-_{i,2}=q^{n-i}$. Then, the ratio of the generalized twists is  
\begin{align}
\frac{ \widetilde \zeta_{i+1}(s^+_{i,k})}{\zeta_i(q^{-1}s^+_{i,k})}=q^{-(n+2-i)}(\gamma^+_{i}\gamma^-_{i})\frac{\gamma^-_{i,1}\gamma^-_{i,2}}{(\gamma^\circ_{i})^2}\frac{(qs^+_{i,k}-\gamma^+_i)}{( s^+_{i,k}-q\gamma^+_i)}  \frac{((s^+_{i,k})^2-q^{n+1-i})}{((s^+_{i,k})^2-q^{n-1-i}) }
\end{align}
The additional contribution exactly cancels the contact term. The Bethe equations dual to the $B_n$ tRS model with $(\gamma^\circ_i)^2=q^{n-1-i}$ and $\gamma^+_{i}\gamma^-_{i}=\gamma^-_{i,1}\gamma^-_{i,2}=q^{n-i}$ are
 \begin{align}
 \frac{(qs^+_{i,k}-\gamma^+_i)}{( s^+_{i,k}-q\gamma^+_i)}   \prod_{\stackrel{j=1}{j\neq k}}^{i}\frac{(qs^+_{i,k}-s^+_{i,j})\left(q^{n-1-i}-s^+_{i,k} s^+_{i,j}\right)}{(s^+_{i,k}-qs^+_{i,j})\left(q^{n+1-i}-s^+_{i,k}s^+_{i,j}\right)} 
    \end{align}
    \[\times\prod_{j=1}^{i+1}\frac{(s^+_{i,k}-q\lambda_{i+1,j})\left(q^{n-i}- s^+_{i,k}\lambda_{i+1,j}\right)   }{(s^+_{i,k}-\lambda_{i+1,j})\left(q^{n-1-i}-s^+_{i,k}\lambda_{i+1,j}\right)}\prod_{j=1}^{i-1}\frac{(s^+_{i,k}-s^+_{i-1,j})\left(q^{n-i} -s^+_{i,k}s^+_{i-1,j}\right)}{(qs^+_{i,k}-s^+_{i-1,j})\left(q^{n+1-i}-s^+_{i,k}s^+_{i-1,j}\right)}=1\,,\]
for $i=1,\cdots,n-1$.
\end{itemize}

\begin{Thm}\label{thm:BC/BBAE}
The above Bethe equations are dual to $BC_n$ and $B_n$ tRS models.
\end{Thm}
\begin{proof}
To compare with tRS boundary factors, it is convenient to write the Bethe equations in terms of $f_q(x)=\frac{x-q}{qx-1}$ and $ g_q(x)=\frac{x-q}{x-1}$.    

Recall the rescaled Bethe roots 
\begin{align}
    u_{i,j}=q^{-(n-i)/2}s^+_{i,j}\,,   &&
    t_{i,k}=q^{-(n-1-i)/2}\l_{i+1,k}\,,
\end{align}
while setting $\gamma^+_{i}=q^{-(n-i)/2}$.
\begin{itemize}
    \item For $BC_n$,
    \begin{align}\label{eq:BetheBC}
       -q^{-i-1}f^{-1}_q\left( u_{i,k}\right)f^{-1}_{q}(u_{i,k}^{2}) \prod_{\stackrel{j=1}{j\neq k}}^{i} f^{-1}_q\left(\frac{u_{i,k}}{u_{i,j}}\right)f^{-1}_q(u_{i,k} u_{i,j})
       \end{align}
    \[\times\prod_{j=1}^{i+1}g_q\left(q^{\frac{1}{2}}\frac{u_{i,k}}{ t_{i,j}}\right)g_q(q^{\frac{1}{2}}u_{i,k}t_{i,j})\prod_{j=1}^{i-1}g^{-1}_{q}\left(q^{-\frac{1}{2}}\frac{u_{i-1,j}}{u_{i,k}}\right)g_q(q^{-\frac{1}{2}}u_{i,k}u_{i-1,j})=1\,,\]
    for $i=1,\cdots,n-1$.
    
    \item For $B_n$, 
\begin{align}\label{eq:BetheB}
      q^{-(i-1)}f^{-1}_q\left(u_{i,k}\right)  \prod_{\stackrel{j=1}{j\neq k}}^{i}f^{-1}_q\left(\frac{u_{i,k}}{u_{i,j}}\right)f^{-1}_q(u_{i,k} u_{i,j})
\end{align}
\[\times\prod_{j=1}^{i+1}g_q\left(q^{\frac{1}{2}}\frac{u_{i,k}}{ t_{i,j}}\right)g_q(q^{\frac{1}{2}}u_{i,k}t_{i,j})\prod_{j=1}^{i-1}g^{-1}_{q}\left(q^{-\frac{1}{2}}\frac{u_{i-1,j}}{u_{i,k}}\right)g_q(q^{-\frac{1}{2}}u_{i,k}u_{i-1,j})=1\,,\]
    for $i=1,\cdots,n-1$.
\end{itemize}
\end{proof}

\subsection{Gluing Isomorphisms}
The gluing isomorphism in \ref{gluing} for type $BC_n$ is given by the following rational function.
\begin{align}\label{eq:BCgluing}
B_i(z)
=
\frac{\xi_{i+1}}{\xi_i}\,
q^{\,i+1-u_i}
\frac{
\left(z-\gamma_{i+1}^{+}\right)
\left(qz-\gamma_{i+1}^{-}\right)
}{
\left(1-q^{\,i+1-u_i}\gamma_i^{+}z\right)
\left(q-q^{\,i+1-u_i}\gamma_i^{-}z\right)
}
,
\end{align}
where $u_{i}\in\frac{1}{2}\mathbb{Z}$ is arbitrary. The simplest gluing in this case is when $u_{i}=n-\frac{1}{2}$ and $\gamma_{i}^{\pm}=q^{\frac{n-i}{2}}$, we have 
\begin{align}
B_i(z)
=
-q^{-1/2}
.
\end{align}

The gluing isomorphism \ref{gluing} in type $B_n$ is given by the following rational function.
\begin{align}\label{eq:Bgluing}
B_i(z)
=
\frac{\xi_{i+1}}{\xi_i}\,
q^{\,u_i-i-1}
\frac{
(z-\gamma_{i+1}^{+})
(qz-\gamma_{i+1}^{-})
(qz-\gamma_{i+1,1}^{-})
(qz-\gamma_{i+1,2}^{-})
}{
(z-\gamma_{i+1}^{\circ})
(z+\gamma_{i+1}^{\circ})
}
\\
\cdot
\frac{
(q^{\,u_i-i-1}-\gamma_i^{\circ}z)
(q^{\,u_i-i-1}+\gamma_i^{\circ}z)
}{
(q^{\,u_i-i-1}-\gamma_i^{+}z)
(q^{\,u_i-i}-\gamma_i^{-}z)
(q^{\,u_i-i}-\gamma_{i,1}^{-}z)
(q^{\,u_i-i}-\gamma_{i,2}^{-}z)
}
,
\end{align}
where $u_{i}\in\frac{1}{2}\mathbb{Z}$ is arbitrary. The simplest gluing in this case is when $u_{i}=n-\frac{1}{2}$, $\frac{\xi_{i+1}}{\xi_i}=q$, $(\gamma^\circ_i)^2=q^{n-1-i}$, $(\gamma^\pm_i)^2=q^{n-i}$ , and we have 
\begin{align}
B_i(z)
=
-q^{\,2(n-i)-\frac{3}{2}}
\frac{
(z-\gamma_{i+1}^{+})
(qz-\gamma_{i+1}^{-})
(qz-\gamma_{i+1,1}^{-})
(qz-\gamma_{i+1,2}^{-})
}{
\left(q^{\,n-i-\frac{3}{2}}-\gamma_i^{+}z\right)
\left(q^{\,n-i-\frac{1}{2}}-\gamma_i^{-}z\right)
\left(q^{\,n-i-\frac{1}{2}}-\gamma_{i,1}^{-}z\right)
\left(q^{\,n-i-\frac{1}{2}}-\gamma_{i,2}^{-}z\right)
}
.
\end{align}
The following theorem follows immediately from Theorems \ref{thm:opersandbethe}, \ref{thm:betheandtrs}, \ref{Lem:BCBfolded} and \ref{thm:BC/BBAE}.
\begin{Thm}\label{thm:BCB}
    The assignment
\begin{equation} 
    p_i= \xi_{n}\frac{Q^{+}_{n-1}(qa_i)}{Q^{+}_{n-1}(a_i)}\,,\quad i=1,\dots,n\,
\end{equation}
induces an isomorphism of algebras:
\begin{align}
    \frac{\mathbb{C}\left[q^{\pm 1},c_i^{\pm 1},\{\alpha_{i,k}^{\pm 1}\},\{\beta_{i,k}^{\pm 1}\}\right]\otimes_{\mathbb{C}} S\left(\{u_{i,j}\},\{t_{i,k}\}\right)^{\mathbb{Z}/{2\mathbb{Z}}}}{(\text{Gluing}^R,\text{Bethe}^R)}\cong
    \frac{
  \mathbb{C}(\{\xi_i\},\{a_i\},q)(\{p_i\})
}{
  \left(
    \det\left(
      u-T^R(\{p_i\},\{a_i\},q)
    \right)
    -F^R(u,\{\xi_i\})
  \right).
}
\end{align}
where $R=BC_n, B_n$, $\text{Bethe}^{BC_n}$ (resp. $\text{Bethe}^{B_n}$) is given by \eqref{eq:BetheBC} (resp. \eqref{eq:BetheB}) and $\text{Gluing}^R$ are the gluing equations \eqref{eq:gluingconditionontwists} where the functions $B_i(z)$ are given by \eqref{eq:BCgluing} and \eqref{eq:Bgluing} when $R=BC_n$ and $R=B_n$ respectively. Moreover, the above isomorphism is compatible with the following one-to-one correspondence among the following four classes of objects:
    \begin{enumerate}[label=\textnormal{(\roman*)}]
        \item the set of isomorphism classes of nondegenerate generalized $Z$-twisted reflection-invariant Miura-Pl\"ucker $(GL(n),q)$-opers with twists given by \eqref{eq:BCtwists} and \eqref{eq:Btwists} and regular singularities are given by $\bm{\Lambda}(z)=\left(\left(\frac{(z-\gamma^+_{i+1})(qz-\gamma^+_{i-1})}{(-qz)}\right)_{i=1}^{n-2},\,\widetilde \L_n\right)$ where the constants $\gamma_i$ satisfy folded conditions \eqref{foldingconditiononsingularieties} with gluing data \eqref{eq:BCgluing} and \eqref{eq:Bgluing} respectively,
        \item the set of nondegenerate solutions of the open XXZ Bethe equations \eqref{eq:BetheBC} and \eqref{eq:BetheB} with boundary contributions in Table \ref{tab:Summary} corresponding to the $BC_n$ and $B_n$ rows respectively,
        \item the energy level set of the $tRS$ model with root systems $BC_n$ and $B_n$ respectively, and
        \item the set of nondegenerate solutions of folded $QQ$-system with twists given by \eqref{eq:BCtwists} and \eqref{eq:Btwists}  dual to the $BC_n$ and $B_n$ $tRS$ models respectively.
    \end{enumerate}
\end{Thm}

We summarize the results obtained in this section in the following table.

\clearpage
\begingroup

\dimen0=\pdfpagewidth
\pdfpagewidth=\pdfpageheight
\pdfpageheight=\dimen0

\setlength{\oddsidemargin}{0.5in}
\setlength{\evensidemargin}{0.5in}
\setlength{\textwidth}{10in}
\setlength{\textheight}{6in}
\setlength{\columnwidth}{\textwidth}
\setlength{\linewidth}{\textwidth}
\setlength{\hsize}{\textwidth}
\setlength{\vsize}{\textheight}

\makeatletter
\setlength{\@colht}{\textheight}
\setlength{\@colroom}{\textheight}
\makeatother

\begin{table}[p]
\thispagestyle{plain}
\centering
\small
\setlength{\tabcolsep}{3pt}
\renewcommand{\arraystretch}{1.35}
\setlength{\abovecaptionskip}{12pt}

\begin{tabular}
{|>{\centering\arraybackslash}p{1.5cm}|
  >{\centering\arraybackslash}p{4.0cm}|
  >{\centering\arraybackslash}p{7.5cm}|
  >{\centering\arraybackslash}p{5.3cm}|
  >{\centering\arraybackslash}p{5.3cm}|}
\hline

\textbf{Root system}
&
\textbf{tRS potential}
&
\textbf{Twists, singularities, and simplest gluing data of $q$-opers}
&
\textbf{Twists of folded $QQ$-systems}
&
\textbf{Boundary contribution in open XXZ chains}
\newline
$\displaystyle\begin{aligned}
\frac{\widetilde{\zeta}_{i+1}(z)}
     {\zeta_i(q^{-1}z)}\\[2mm]    
\end{aligned}
$
\\
\hline

$BC_n$
&
$\displaystyle
\begin{aligned}
&\prod_{j\ne i}
 g_q\!\left(\frac{a_i}{a_j}\right)
 g_q(a_i a_j)
\\[-1mm]
&\qquad{}\cdot g_q(a_i)g_q(a_i^2)
\end{aligned}$
&
$\displaystyle
\begin{aligned}\\[-2mm]
\zeta_i(z)
={}&-\xi_i
\frac{(z-\gamma_i^+)(qz-\gamma_i^-)}
     {z},
\\[2mm]
\bm{\Lambda}(z)={}&\left(\left(\frac{(z-\gamma^+_{i+1})(qz-\gamma^+_{i-1})}{(-qz)}\right)_{i=1}^{n-2},\,\widetilde \L_n\right),
\\[2mm]
B_i(z)={}&-q^{-1/2},\qquad   u_i={}n-\frac12
\\[2mm]
\end{aligned}$
&
$\displaystyle
\zeta_i(z)
=-\xi_i
\frac{(z-\gamma_i^+)(qz-\gamma_i^-)}
     {z}$
&
$\displaystyle
\frac{\xi_{i+1}}{\xi_i}
\frac{qz-\gamma_i^+}
     {z-q\gamma_i^+}$
\\
\hline

$B_n$
&
$\displaystyle
\begin{aligned}
&\prod_{j\ne i}
 g_q\!\left(\frac{a_i}{a_j}\right)
 g_q(a_i a_j)
\\[-1mm]
&\qquad{}\cdot g_q(a_i)
\end{aligned}$
&
$\displaystyle
\begin{aligned}\\[-2mm]
\zeta_i(z)
={}&-\xi_i
\frac{(z-\gamma_i^+)(qz-\gamma_i^-)}{z}
\\[-1mm]
&\quad{}\cdot
\frac{(qz-\gamma_{i,1}^-)(qz-\gamma_{i,2}^-)}
     {(z-\gamma_i^\circ)(z+\gamma_i^\circ)},
\\[2mm]
\bm{\Lambda}(z)={}&\left(\left(\frac{(z-\gamma^+_{i+1})(qz-\gamma^+_{i-1})}{(-qz)}\right)_{i=1}^{n-2},\,\widetilde \L_n\right),
\\[2mm]
B_i(z)
={}&-q^{\,2(n-i)-\frac32}
\frac{z-\gamma_{i+1}^+}
     {q^{\,n-i-\frac32}-\gamma_i^+z}
\\[-1mm]
&\quad{}\cdot
\frac{qz-\gamma_{i+1}^-}
     {q^{\,n-i-\frac12}-\gamma_i^-z}
\\[-1mm]
&\quad{}\cdot
\frac{qz-\gamma_{i+1,1}^-}
     {q^{\,n-i-\frac12}-\gamma_{i,1}^-z}
\\[-1mm]
&\quad{}\cdot
\frac{qz-\gamma_{i+1,2}^-}
     {q^{\,n-i-\frac12}-\gamma_{i,2}^-z},\qquad u_i={}n-\frac12\\[1mm]
\end{aligned}$
&
$\displaystyle
\begin{aligned}
\zeta_i(z)
={}&-\xi_i
\frac{(z-\gamma_i^+)(qz-\gamma_i^-)}{z}
\\[-1mm]
&\quad{}\cdot
\frac{(qz-\gamma_{i,1}^-)(qz-\gamma_{i,2}^-)}
     {(z-\gamma_i^\circ)(z+\gamma_i^\circ)}
\end{aligned}$
&
$\displaystyle
\begin{aligned}
\frac{\xi_{i+1}}{q^2\xi_i}
 \frac{qz-\gamma_i^+}
      {z-q\gamma_i^+}
\cdot
 \frac{z^2-q^{\,n+1-i}}
      {z^2-q^{\,n-1-i}}
\end{aligned}$
\\
\hline

\end{tabular}

\par\vspace{8pt}

\caption{Correspondence between tRS models, $q$-opers,
$QQ$-systems, and open XXZ spin chains.}
\label{tab:folded-potential}

\end{table}

\clearpage
\endgroup
\setlength{\textheight}{8.5in}
\setlength{\vsize}{\textheight}
\makeatletter
\setlength{\@colht}{\textheight}
\setlength{\@colroom}{\textheight}
\makeatother

\section{Folding of \texorpdfstring{$GL(N)$}{GL(N)} \texorpdfstring{$QQ$}{QQ}-systems}\label{sec:SummaryConj}
In this section, we present the folding of the $QQ$-systems associated with generic good $A_n$ quivers as well as interesting observations regarding the generalized twists. 

There are various ways to fold $QQ$-systems associated with generic framed quivers beyond the electric frame of tRS systems for the classical groups. Different ways of folding provide different classes of deformations of tRS systems. 

Here, we only provide examples of the straightforward generalization of folding used in the previous sections.

\subsection{$QQ$-Systems Associated with generic good $A_n$ quivers}\label{sec:GenAn}
Consider the following good $A_n$ quiver
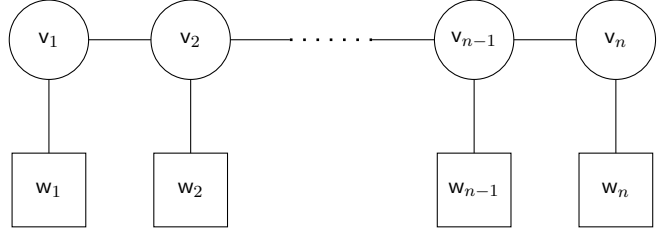
\begin{center}
    \begin{tikzpicture}[scale=.75, every node/.style={scale=0.8}]
        \draw[] (-6,0) circle (.7);
        \draw[] (-5.3,0) -- (-4.2,0);
        \draw[] (-3.5,0) circle (.7);
        \draw[] (-2.8,0) -- (-1.75,0);
        \draw[line width=1pt,loosely dotted] (-1.75,0) -- (-.3,0);
        \draw[] (-.3,0) -- (.8,0);
        \draw[] (1.5,0) circle (.7);   
        \draw[] (2.2,0) -- (3.3,0);
        \draw[] (4,0) circle (.7);      
        \draw[] (4,-.7) -- (4,-2);
        \draw[] (1.5,-.7) -- (1.5,-2);
        \draw[] (-3.5,-.7) -- (-3.5,-2);
        \draw[] (-6,-.7) -- (-6,-2);   
        \draw[] (3.35,-2) rectangle (4.65,-3.3);
        \draw[] (0.85,-2) rectangle (2.15,-3.3);
        \draw[] (-4.15,-2) rectangle (-2.85,-3.3);
        \draw[] (-6.65,-2) rectangle (-5.35,-3.3);
        \node at (-6,0) {$\mathsf v_{1}$};
        \node at (-3.5,0) {$\mathsf v_{2}$};
        \node at (1.5,0) {$\mathsf v_{n-1}$};
        \node at (4,0) {$\mathsf v_{n}$};

         \node at (-6,-2.65) {$\mathsf w_{1}$};
        \node at (-3.5,-2.65) {$\mathsf w_{2}$};
        \node at (1.5,-2.65) {$\mathsf w_{n-1}$};
        \node at (4,-2.65) {$\mathsf w_{n}$};
    \end{tikzpicture}
    \end{center}
with the quiver data $\{\xi_{i=1,\cdots,n+1},\mathsf v_{i=1,\cdots,n},\mathsf w_{i=1,\cdots,n}\}$.
The $QQ$-system reads
\begin{align}
    \xi_{i+1}Q^+_i(qz)Q^-_i(z)- \xi_{i}Q^+_i(z)Q^-_i(qz)=Q^{\circ}_i(z) Q^+_{i+1}(z)Q^+_{i-1}(qz), &&i=1,\cdots, n\,,
\end{align}
with boundary conditions $Q^\circ_n(z)=\Lambda(z)$, 
$Q^+_{0}(z)=Q^+_{n+1}(z)=1$.
$Q^+_i(z)$ and $Q^\circ_i$ are monic polynomials of degree $\mathsf v^+_i:=\mathsf v_i$ and $\mathsf v^\circ_i:=\mathsf w_i$ given by the dimension vectors of the quiver. The degree of $Q^-_i$ is $\mathsf v^-_i:=\mathsf v_{i+1}+\mathsf v_{i-1}+\mathsf w_{i}-\mathsf v_i$.
Let us define
\begin{align}
\Delta^\bullet_i:=\left\lfloor\frac{\mathsf v^\bullet_i}{2}\right\rfloor=\frac{\mathsf v^\bullet_i-\e^\bullet_i}{2}\,,&&\bullet\in\{+,-,\circ\}\,,
\end{align}
where $\e^\bullet_i$ generalizes $\e^\pm$ used for $T^*\mathrm{Gr}(k,n)$. 

The tRS momenta are given by
\begin{align}
    p_i=\xi_{n+1}\frac{Q_{n}^{+}\left(qa_{i}\right)}{Q_{n}^{+}\left(a_{i}\right)},&&i=1,\cdots,\mathsf w_n\,,
\end{align}
where $a_i$ is the $i$-th root of $\Lambda(z)$.

The folding conditions for the canonical variables of deformed tRS models are
\begin{align}
  \xi_{n+1}^{2}\frac{ Q_{n}^{+}(qa_i)}{ Q_{n}^{+}(a_i)}\frac{ Q_{n}^{+}(qa_i^{-1})}{ Q_{n}^{+}(a_i^{-1})}=1\,,
\end{align} 
with $a_ia_{\mathsf w_n+1-i}=1$ for $C/BC$-types and $a_ia_{\mathsf w_n-1-i}=1$ for $B/D$-types.
For $BC/B$-types, $\e^+_{n-1}=\e^\circ_n=1$, it is supplemented by 
\begin{align}
p_{N_{R}}=\xi_{n+1}\frac{Q_{n}^{+}\left(qa_{N_R}\right)}{Q_{n}^{+}\left(a_{N_R}\right)}=1,&&a_{N_R}=1\,,
\end{align}
where $N_{BC}:=\D^\circ_{n}+1$ and $N_{B}:=\D^\circ_{n}$.
\subsection{Folding and Reflection-Invariance}
Let us define $BC/C$ type folding of the generic good $A_n$ quiver as follows.

\begin{Def}\label{def:genericA_nfoldingBCC}
The $BC/C$ type folded $QQ$-system associated with the generic good quiver with the generalized twists  
\begin{align}
     \widetilde  \zeta_{i+1}(z) \widetilde Q^+_i(qz) \widetilde Q^-_i(z)- \zeta_i(z)  \widetilde Q^+_i(z)\widetilde Q^-_i(qz)=\widetilde Q^\circ_i(z) \widetilde \L_{i+1}(z) \widetilde Q^+_{i-1}(qz)\,, &&i=1,\cdots,n\label{eq:genericA_nfolding}
\end{align}
consists of the Laurent polynomials $\widetilde Q^\bullet(z)$ for $\bullet\in\{+,-,\circ\}$ induced from the $QQ$-system associated with a quiver by
\begin{align}
     &\widetilde Q^\bullet_{i}(z):=\left.\frac{Q^\bullet_{i}(z)}{(z-\gamma^{\bullet}_i)^{\e^\bullet_i}(-z)^{\Delta^\bullet_i}}\right|_{\textrm{folded}}=\a^{\bullet}_i \prod_{k=1}^{\Delta^\bullet_i}(z-s^\bullet_{i,k})\left(\frac{q^{\pi^\bullet_i}}{z}-s^\bullet_{i,k}\right)\in \mathbb{C}\left[z+\frac{q^{\pi^\bullet_i}}{z}\right]\,,
\end{align}
where $\pi^\pm_i=n+1-i$ and $\pi^\circ_i=n-i$ for $i=1,\cdots, n$.

For the RHS, we define
\begin{align}
    \widetilde \L_{i+1}(z):=\left.\l_{i+1}\left(\frac{(z-\gamma^+_{i+1})(qz-\gamma^+_{i-1})}{(-qz)}\right)^{\e^+_{i\pm1}} \right|_{\textrm{folded}}\widetilde Q^+_{i+1}(z)\,, && \e^+_{i+1}=\e^+_{i-1}\,.
\end{align}
When $\e^+_{i+1}\neq \e^+_{i-1}$, $\widetilde \L_{i+1}(z)=\widetilde Q^+_{i+1}(z)$ for $i=1,\cdots,n-1$ with the folding conditions for $s^\bullet_{i,k}$
\begin{align}
    s^\bullet_{i,k}s^\bullet_{i,\mathsf v^\bullet_i+1-k}=q^{\pi^\bullet_i}, &&k=1,\cdots,\Delta^\bullet_i\,,\label{eq:foldingAnBetheroots}
\end{align}
The folding of tRS positions is encoded in
\begin{align}
\widetilde \L_{n+1}(z):=\widetilde Q^\circ_{n}(z)     =\prod_{k=1}^{\D^\circ_{n}+1}(z-a_k)\left(\frac{1}{z}-a_k\right),&&\e^\circ_{n}=\e^+_{n-1}=1,
\end{align}
where $a_{\D^\circ_n+1}=1$. This characterizes the generalized $BC$-type folding.

On the other hand, the generalized $C$-type folding is characterized by
\begin{align}
    \widetilde \L_{n+1}(z):=\widetilde Q^\circ_{n}(z) =\prod_{k=1}^{\D^\circ_{n}}(z-a_k)\left(\frac{1}{z}-a_k\right)\,.
\end{align}

The twists are generalized to rational functions on $\P^1$. When $\e^+_{i+1}\neq\e^+_{i-1}$,
\begin{align}
\begin{split}
     \widetilde  \zeta_{i+1}(z):=\xi_{i+1}q^{\Delta^+_i-\Delta^+_{i-1}}(-z)^{-\widetilde  \Gamma_i}\frac{ (qz-\gamma^+_i)^{\e^+_i}(z-\gamma^-_i)^{\e^-_i}}{   (z-\gamma^\circ_i)^{\e^\circ_i} (z-\gamma^+_{i+1})^{\e^+_{i+1}} (qz-\gamma^+_{i-1})^{\e^+_{i-1}}} \,,
     \\
     \zeta_{i}(z):=  \xi_{i}q^{\Delta^-_i-\Delta^+_{i-1}}(-z)^{-\widetilde  \Gamma_i }\frac{(z-\gamma^+_i)^{\e^+_i}(qz-\gamma^-_i)^{\e^-}}{ (z-\gamma^\circ_i)^{\e^\circ_i} (z-\gamma^+_{i+1})^{\e^+_{i+1}} (qz-\gamma^+_{i-1})^{\e^+_{i-1}}}\,,
\end{split}\label{eq:gentwistsgenAn1BCC}
\end{align}
where 
\begin{align}
 \widetilde  \Gamma_i:=\frac{1}{2}\left(\e^+_i+\e^-_{i}-\e^+_{i+1}-\e^\circ_{i}-\e^+_{i-1}\right)=-1,0,1\,,
\end{align} 
\end{Def}
When $\e^+_{i+1}=\e^+_{i-1}$,
\begin{align}
\begin{split}
     \widetilde  \zeta_{i+1}(z):=\xi_{i+1}q^{\Delta^+_i-\Delta^+_{i-1}-\e^+_{i\pm1}}(-z)^{-\widetilde  \Gamma_i-\e^+_{i\pm1}}\frac{ (qz-\gamma^+_i)^{\e^+_i}(z-\gamma^-_i)^{\e^-_i}}{   (z-\gamma^\circ_i)^{\e^\circ_i}} \,,
     \\
     \zeta_{i}(z):=  \xi_{i}q^{\Delta^-_i-\Delta^+_{i-1}-\e^+_{i\pm1}}(-z)^{-\widetilde  \Gamma_i-\e^+_{i\pm1} }\frac{(z-\gamma^+_i)^{\e^+_i}(qz-\gamma^-_i)^{\e^-}}{ (z-\gamma^\circ_i)^{\e^\circ_i}}\,.
\end{split}\label{eq:gentwistsgenAn2BCC}
\end{align}

\begin{Rem}
When $\e^+_{i+1}\neq \e^+_{i-1}$, one can choose another definition for $\widetilde\Lambda_{i+1}(z)$, for example, by pairing the linear factors from $Q^\circ_i(z)$ and $Q^+_{i\pm1}(q^{\frac{1\mp1}{2}}z)$. 
\end{Rem}

\begin{Thm}\label{thm:genericAfoldedBCC}
The folded $QQ$-systems of Definition \ref{def:genericA_nfoldingBCC} exist for generalized $BC$ and $C$ type models. Moreover, the generalized twists must satisfy the following reflection-invariant conditions
\begin{align}
   \widetilde\zeta_{i+1}(z) =- \zeta_i\left(\frac{q^{n-1-i}}{z}\right)\,.\label{eq:genRI}
\end{align}
\end{Thm}
\begin{proof}
When $\e^+_{i+1}=\e^+_{i-1}$, we divided the $QQ$-system by $(z-\gamma^\circ_i)^{\e^\circ_i}(-z)^{\Delta^\circ_i+\Delta^+_{i+1}+\Delta^+_{i-1}+\e^+_{i\pm1}}$. 
\begin{align}
\xi_{i+1}\frac{(qz-\gamma^+_i)^{\e^+_i}(z-\gamma^-_i)^{\e^-_i} q^{\Delta^+_i-\Delta^+_{i-1}-\e^+_{i\pm1}}}{(z-\gamma^\circ_i)^{\e^\circ_i}(-z)^{\Delta^\circ_i+\Delta^+_{i+1}+\Delta^+_{i-1}+\e^+_{i\pm1}-\Delta^+_{i}-\Delta^-_{i}}}\frac{Q^+_i(qz)}{(qz-\gamma^+_i)^{\e^+_i}(-qz)^{\Delta^+_i}}\frac{Q^-_i(z)}{(z-\gamma^-_i)^{\e^-_i}(-z)^{\Delta^-_i}}
\\
-\xi_{i}\frac{(z-\gamma^+_i)^{\e^+_i}(qz-\gamma^-_i)^{\e^-_i}q^{\Delta^-_i-\Delta^+_{i-1}-\e^+_{i\pm1}}}{(z-\gamma^\circ_i)^{\e^\circ_i}(-z)^{\Delta^\circ_i+\Delta^+_{i+1}+\Delta^+_{i-1}+\e^+_{i\pm1}-\Delta^+_{i}-\Delta^-_{i}}}\frac{Q^+_i(z)}{(z-\gamma^+_i)^{\e^+_i}(-z)^{\Delta^+_i}}\frac{Q^-_i(qz)}{(qz-\gamma^-_i)^{\e^-_i}(-qz)^{\Delta^-_i}}\nonumber
\\
= \frac{Q^\circ_i(z)}{(z-\gamma^{\circ}_i)^{\e^\circ_i}(-z)^{\Delta^\circ_n}}\left(\frac{(z-\gamma^+_{i+1}) (qz-\gamma^+_{i-1}) }{(-qz)}\right)^{\e^+_{i\pm1}}\frac{Q^+_{i+1}(z)}{(z-\gamma^+_{i+1})^{^{\e^+_{i\pm1}}}(-z)^{\Delta^+_{i+1}}}\frac{Q^+_{i-1}(qz)}{(qz-\gamma^+_{i-1})^{^{\e^+_{i\pm1}}}(-qz)^{\Delta^+_{i-1}}}\,.\nonumber
\end{align}
By imposing the folding conditions on the Bethe roots, we obtain the folded $QQ$-system.
One can show that $\Delta^+_i+\Delta^-_i-\Delta^{\circ}_{i}-\Delta^+_{i+1}-\Delta^+_{i-1}= -\widetilde  \Gamma_i$.

When $\e^+_{i+1}\neq\e^+_{i-1}$, we divided the $QQ$-system by $(z-\gamma^\circ_i)^{\e^\circ_i}(z-\gamma^+_{i+1})^{\e^+_{i+1}}(z-\gamma^+_{i-1})^{\e^+_{i-1}}(-z)^{\Delta^\circ_i+\Delta^+_{i+1}+\Delta^+_{i-1}}$. The rest of the argument is same.

The reflection-invariant condition is the $r=n$ case of Lemma \ref{lem:generalRItwists}.
\end{proof}

For $B/D$ type, we define the folding of the generic $A_n$ quiver as follows:
\begin{Def}\label{def:genericA_nfoldingBD}
For $B/D$ type folding, we use the same definition for $\widetilde Q^+_i(z)$ and $\widetilde \Lambda_i(z)$. We redefine $ \widetilde Q^-_{i}(z)$ and $  \widetilde Q^\circ_{i}(z)$ such that
\begin{align}
      \widetilde Q^-_{i}(z):=\left.\frac{(z-\gamma^{-}_i)^{-\e^\bullet_i}(-z)^{-\Delta^-_i+1}Q^-_{i}(z)}{(z-\gamma^-_{i,1})(z-\gamma^-_{i,2})}\right|_{\textrm{folded}}=\a^{-}_i \prod_{k=1}^{\Delta^-_i-1}(z-s^-_{i,k})\left(\frac{q^{n+1-i}}{z}-s^-_{i,k}\right)\in \mathbb{C}\left[z+\frac{q^{n+1-i}}{z}\right]\,,\nonumber
      \\
     \widetilde Q^\circ_{i}(z):=\left.\frac{(z-\gamma^{\circ}_i)^{-\e^\circ_i}(-z)^{-\Delta^\circ_i+1}Q^\circ_{i}(z)}{(z-\gamma^\circ_{i,1})(z-\gamma^\circ_{i,2})}\right|_{\textrm{folded}}=\a^{\circ}_i \prod_{k=1}^{\Delta^\circ_i-1}(z-s^\circ_{i,k})\left(\frac{q^{n-i}}{z}-s^\circ_{i,k}\right)\in \mathbb{C}\left[z+\frac{q^{n-i}}{z}\right]\,,\nonumber
\end{align}
where $\pi^\pm_i=n+1-i$ and $\pi^\circ_i=n-i$ for $i=1,\cdots, n$, and $\gamma^\circ_{i,1}+\gamma^\circ_{i,2}=0$. As before, $\widetilde \Lambda_{n+1}(z):=\widetilde Q^\circ_{n}(z)$.

The generalized twists are given by
\begin{align}
 \widetilde\zeta_{i+1}^{B/D}(z)=\frac{(z-\gamma^-_{i,1})(z-\gamma^-_{i,2})}{(z-\gamma^\circ_{i,1})(z+\gamma^\circ_{i,1})}\widetilde\zeta_{i+1}^{BC/C}(z)\,,&&\zeta_{i}^{B/D}(z)=q^{-1}\frac{(qz-\gamma^-_{i,1})(qz-\gamma^-_{i,2})}{(z-\gamma^\circ_{i,1})(z+\gamma^\circ_{i,1})} \widetilde\zeta_{i}^{BC/C}(z)\,.
\end{align}
\end{Def}

\begin{Rem}
The generalized $B/D$-type folding of the generic good  $A_n$ quiver is possible when $\mathsf w_i\geq2$ for all $i$. 
\end{Rem}

\begin{Thm}\label{thm:genericAfoldedBD}
The folded $QQ$-systems of Definition \ref{def:genericA_nfoldingBD} exist for generalized $B$ and $D$ type models. Moreover, the generalized twists must satisfy the same reflection-invariant conditions \eqref{eq:genRI}.
\end{Thm}
\begin{proof}
When  $\e^+_{i+1}=\e^+_{i-1}$, divide the $QQ$-system by 
\begin{align}
(z-\gamma^\circ_{i,1})(z+\gamma^\circ_{i,1}) (z-\gamma^\circ_i)^{\e^\circ_i}(-z)^{\Delta^\circ_i-1+\Delta^+_{i+1}+\Delta^+_{i-1}+\e^+_{i\pm1} }\,.
\end{align}
When $\e^+_{i+1}\neq\e^+_{i-1}$, divide the $QQ$-system by
\begin{align}
  (z-\gamma^\circ_{i,1})(z+\gamma^\circ_{i,1})(z-\gamma^\circ_i)^{\e^\circ_i}(z-\gamma^+_{i+1})^{\e^+_{i+1}}(z-\gamma^+_{i-1})^{\e^+_{i-1}}(-z)^{\Delta^\circ_i-1+\Delta^+_{i+1}+\Delta^+_{i-1}}\,.
\end{align}
By imposing the folding conditions on the Bethe roots, one obtains the folded $QQ$-systems.
\end{proof}

As a corollary to Lemma \ref{Lem:Gen_RI_Sols}, 
\begin{Cor}\label{lem:genericRIsol}
The reflection-invariant conditions on the generalized twists of $BC/C$ and $B/D$ type folding for the good generic $A_n$ quiver are solved for $i=1,\cdots,n$ by     
\begin{itemize}
    \item For $BC/C$ type:
    \begin{align}
\frac{\xi_{i+1}}{\xi_{i}}= - q^{-(\Delta^+_i-\Delta^-_i+\e^+_i-\e^+_{i-1} +\widetilde  \Gamma_i (n-i) )}\frac{(\gamma^+_i)^{\e^+_i}(\gamma^-_i)^{\e^-_i}}{(\gamma^\circ_i)^{\e^\circ_i} (\gamma^+_{i+1})^{\e^+_{i+1}} (\gamma^+_{i-1})^{\e^+_{i-1}}}\, ,&&\e^+_{i+1}\neq \e^+_{i-1}\,,
\\
\frac{\xi_{i+1}}{\xi_{i}}= - q^{-(\Delta^+_i-\Delta^-_i+\e^+_i-\e^+_{i-1} +(\widetilde  \Gamma_i+\e^+_{i\pm1} )(n-i) )}\frac{(\gamma^+_i)^{\e^+_i}(\gamma^-_i)^{\e^-_i}}{(\gamma^\circ_i)^{\e^\circ_i}}\, ,&&\e^+_{i+1}=\e^+_{i-1}\,,
\end{align}
and
\begin{align}
(\gamma^\pm_i)^{2} =q^{n+1-i},&& (\gamma^\circ_i)^{2}=q^{n-i}\,.
\end{align}

    \item For $B/D$ type:
    \begin{align}
\frac{\xi_{i+1}}{\xi_{i}}=  q^{-(\Delta^+_i-\Delta^-_i+1+\e^+_i-\e^+_{i-1} +\widetilde  \Gamma_i (n-i) )}\frac{(\gamma^+_i)^{\e^+_i}(\gamma^-_i)^{\e^-_i}}{(\gamma^\circ_i)^{\e^\circ_i} (\gamma^+_{i+1})^{\e^+_{i+1}} (\gamma^+_{i-1})^{\e^+_{i-1}}}\frac{\gamma^-_{i,1}\gamma^-_{i,2}}{(\gamma^\circ_{i,1})^2}\, ,&&\e^+_{i+1}\neq \e^+_{i-1}\,,
\\
\frac{\xi_{i+1}}{\xi_{i}}= - q^{-(\Delta^+_i-\Delta^-_i+1+\e^+_i-\e^+_{i-1} +(\widetilde  \Gamma_i+\e^+_{i\pm1}  )(n-i) )}\frac{(\gamma^+_i)^{\e^+_i}(\gamma^-_i)^{\e^-_i}}{(\gamma^\circ_i)^{\e^\circ_i}}\frac{\gamma^-_{i,1}\gamma^-_{i,2}}{(\gamma^\circ_{i,1})^2}\, ,&&\e^+_{i+1}=\e^+_{i-1}\,,
\end{align}
and
\begin{align}
(\gamma^\circ_i)^{2}=q^{n-i}\,,&& (\gamma^\circ_{i,1})^{2}=\pm q^{n-i}\,.
\end{align}
and products of suitable pairs among $\gamma^\pm_i,\gamma^-_{i,1},\gamma^-_{i,2}$ are $q^{n+1-i}$. For example,
\begin{align}
    (\gamma^\pm_i)^{2} =(\gamma^-_{i,1})^{2} =(\gamma^-_{i,2})^{2} =q^{n+1-i}\,,&&\textrm{or}&&&&\gamma^+_i \gamma^-_i=q^{n+1-i}\,,&&\textrm{or}&&\gamma^-_{i,1}\gamma^-_{i,2}=q^{n+1-i}\,,
\end{align}
for some $i$.
\end{itemize}
    
\end{Cor}

 The folding conditions of deformed tRS momenta can be solved as follows:
\begin{itemize}
    \item For the generalized $C/D$ type, 
\begin{align}
  \xi_{n+1}^{2}q^{2\Delta^+_n}\left(\frac{(qa_i-\gamma^+_n)(qa_i^{-1}-\gamma^+_n)}{(a_i-\gamma^+_n)(a_i^{-1}-\gamma^+_n)}\right)^{\e^+_n}=1\,.
\end{align} 
    \item For the generalized $BC/B$ type, 
\begin{align}
\xi_{n+1}q^{\Delta^+_n}\left(\frac{q-\gamma^+_n}{q-\gamma^+_n}\right)^{\e^+_n}=1\,,
\end{align}
with $\gamma^+_n=q^{\frac{1}{2}}$ by Corollary \ref{lem:genericRIsol}.

\end{itemize}

\subsubsection{Bethe Equations}\label{sec:GenAnBAE}

The Bethe equations for the generic $A_n$ quiver are obtained in the canonical way.
\begin{align}
    - \frac{\widetilde  \zeta_{i+1}(s^+_{i,k})}{\zeta_i(q^{-1}s^+_{i,k})}\frac{ \widetilde Q^+_i(qs^+_{i,k})}{  \widetilde Q^+_i(q^{-1}s^+_{i,k})}\frac{\widetilde Q^\circ_i(q^{-1}s^+_{i,k})}{\widetilde Q^\circ_i(s^+_{i,k})}\frac{\widetilde  \Lambda^+_{i+1}(q^{-1}s^+_{i,k})}{\widetilde  \Lambda^+_{i+1}(s^+_{i,k})}\frac{\widetilde Q^+_{i-1}(s^+_{i,k})}{\widetilde Q^+_{i-1}(qs^+_{i,k})}=1  \,,
\end{align}
for $i=1,\cdots,n$. Since the generalized twists contain contributions of $\gamma^+_{i\pm 1}$ and $\gamma^\circ_i$, the resulting Bethe equations are naturally dual to deformed tRS models of type $B$, $BC$, $C$, and $D$. The notable factors from the twists for $\mathsf v_i>1$ are
\begin{itemize}
    \item $\e^+_{i+1}=\e^+_{i-1}$,
    \begin{align}
       \left(\frac{(q^{-1}s^+_{i,k}-\gamma^\circ_i)}{(s^+_{i,k}-\gamma^\circ_i)}\right)^{\e^\circ_i}\,.
    \end{align}

    \item $\e^+_{i+1}\neq\e^+_{i-1}$,
    \begin{align}
          \left(\frac{(q^{-1}s^+_{i,k}-\gamma^\circ_i)}{(s^+_{i,k}-\gamma^\circ_i)}\right)^{\e^\circ_i}\left(\frac{q^{-1}s^+_{i,k}-\gamma^+_{i+1}}{s^+_{i,k}-\gamma^+_{i+1}}\right)^{\e^+_{i+1}}\left(\frac{s^+_{i,k}-\gamma^+_{i-1}}{qs^+_{i,k}-\gamma^+_{i-1}}\right)^{\e^+_{i-1}}\,.
    \end{align}   
\end{itemize}

In the case of $\mathsf v_i=1$, there are different factors:
\begin{itemize}
    \item $\e^+_{i+1}=\e^+_{i-1}$, 
    \begin{align}
         \left(\frac{(s^-_{i,k}-\gamma^\circ_i)}{(q^{-1}s^-_{i,k}-\gamma^\circ_i)}\right)^{\e^\circ_i}\prod_{j=1}^{2}\left(\frac{q^{-1}s^-_{i,k}-\gamma^-_{i,j}}{qs^-_{i,k}-\gamma^-_{i,j}}\right)\,.
    \end{align}

    \item $\e^+_{i+1}\neq\e^+_{i-1}$,
    \begin{align}
          \left(\frac{(s^-_{i,k}-\gamma^\circ_i)}{(q^{-1}s^-_{i,k}-\gamma^\circ_i)}\right)^{\e^\circ_i}\left(\frac{ s^-_{i,k}-\gamma^+_{i+1}}{q^{-1}s^-_{i,k}-\gamma^+_{i+1}}\right)^{\e^+_{i+1}}\left(\frac{qs^-_{i,k}-\gamma^+_{i-1}}{s^-_{i,k}-\gamma^+_{i-1}}\right)^{\e^+_{i-1}}
      \prod_{j=1}^{2}\left(\frac{q^{-1}s^-_{i,k}-\gamma^-_{i,j}}{qs^-_{i,k}-\gamma^-_{i,j}}\right)\,.
    \end{align}   
\end{itemize}

Thus, we have the following theorem
\begin{Thm}
 The Bethe equations of the folded $QQ$-systems of the definitions \ref{def:genericA_nfoldingBCC} and \ref{def:genericA_nfoldingBD} are dual to deformations of the tRS models for the classical groups: $B$, $BC$, $C$, and $D$.
\end{Thm}

\subsection{Gluing Isomorphisms}\label{sec:gen_gluing_iso}
Define 
\begin{align}
B_{i-1}(z):=\frac{ \zeta_{i}(q^{u_i}z)}{\widetilde \zeta_{i}(z)}\,,\quad 2\leq i \leq n\,,\label{eq:newgluing}
\end{align}
with undetermined $u_i$. Note that this differs from \eqref{eq:gluing_general}, but the two formulas are related. The modified gluing isomorphism $B_{i-1}$ from $(i)$ to $(i-1)$ is
\begin{itemize}
    \item For $BC/C$ type with $\e^+_{i+1}\neq \e^+_{i-1}$ and $\e^+_i\neq\e^+_{i-2}$
    \begin{align}
        \frac{\zeta_{i}(q^{u_i}z)}{\widetilde\zeta_i(z)}=q^{\Delta^-_i-2\Delta^+_{i-1}+\Delta^+_{i-2}-u_i\widetilde\Gamma_i}(-z)^{-(\widetilde\Gamma_i-\widetilde\Gamma_{i-1})}
    \end{align}
\[\times\frac{(z-\gamma^\circ_{i-1})^{\e^\circ_{i-1}}}{(q^{u_i}z-\gamma^\circ_i)^{\e^\circ_i}}\frac{(z-\gamma^+_{i})^{\e^+_{i}}}{(q^{u_i}z-\gamma^+_{i+1})^{\e^+_{i+1}}}\frac{(qz-\gamma^+_{i-2})^{\e^+_{i-2}}}{(q^{u_i+1}z-\gamma^+_{i-1})^{\e^+_{i-1}}}\frac{(q^{u_i}z-\gamma^+_{i})^{\e^+_{i}}}{(qz-\gamma^+_{i-1})^{\e^+_{i-1}}}\frac{(q^{u_i+1}z-\gamma^-_{i})^{\e^-_{i}}}{(qz-\gamma^-_{i-1})^{\e^-_{i-1}}}\,.\]
    \item For $BC/C$ type with $\e^+_{i+1}=\e^+_{i+1}$ and $\e^+_i\neq\e^+_{i-2}$
\begin{align}
\frac{\zeta_{i}(q^{u_i}z)}{\widetilde\zeta_i(z)}=q^{\Delta^-_i-2\Delta^+_{i-1}+\Delta^+_{i-2}-\e^+_{i\pm1}-u_i(\widetilde\Gamma_i+\e^+_{i\pm1})}z^{-(\widetilde\Gamma_i-\widetilde\Gamma_{i-1})}
\end{align}
\[\times\frac{(z-\gamma^\circ_{i-1})^{\e^\circ_{i-1}}}{(q^{u_i}z-\gamma^\circ_i)^{\e^\circ_i}}\frac{(q^{u_i}z-\gamma^+_{i})^{\e^+_{i}}}{(qz-\gamma^+_{i-1})^{\e^+_{i-1}}}\frac{(q^{u_i+1}z-\gamma^-_{i})^{\e^-_{i}}}{(qz-\gamma^-_{i-1})^{\e^-_{i-1}}}\frac{(z-\gamma^+_{i})^{\e^+_{i}}(qz-\gamma^+_{i-2})^{\e^+_{i-2}}}{z^{\e^+_{i\pm1}}}\,.\]

    \item For $BC/C$ type with $\e^+_{i+1}\neq\e^+_{i+1}$ and $\e^+_i =\e^+_{i-2}$
\begin{align}
\frac{\zeta_{i}(q^{u_i}z)}{\widetilde\zeta_i(z)}=q^{\Delta^-_i-2\Delta^+_{i-1}+\Delta^+_{i-2}+\e^+_{i-1\pm1}-u_i(\widetilde\Gamma_i)}z^{-(\widetilde\Gamma_i-\widetilde\Gamma_{i-1})}
\end{align}
\[\times\frac{(z-\gamma^\circ_{i-1})^{\e^\circ_{i-1}}}{(q^{u_i}z-\gamma^\circ_i)^{\e^\circ_i}}\frac{(q^{u_i}z-\gamma^+_{i})^{\e^+_{i}}}{(qz-\gamma^+_{i-1})^{\e^+_{i-1}}}\frac{(q^{u_i+1}z-\gamma^-_{i})^{\e^-_{i}}}{(qz-\gamma^-_{i-1})^{\e^-_{i-1}}}\frac{z^{\e^+_{i-1\pm1}}}{(z-\gamma^+_{i+1})^{\e^+_{i+1}}(qz-\gamma^+_{i-1})^{\e^+_{i-1}}}\,.\]

    \item For $BC/C$ type with $\e^+_{i+1}=\e^+_{i+1}$ and $\e^+_i =\e^+_{i-2}$
\begin{align}
\frac{\zeta_{i}(q^{u_i}z)}{\widetilde\zeta_i(z)}=q^{\Delta^-_i-2\Delta^+_{i-1}+\Delta^+_{i-2}-\e^+_{i\pm1}+\e^+_{i-1\pm1}-u_i(\widetilde\Gamma_i)}z^{-(\widetilde\Gamma_i+\e^+_{i\pm1}-\widetilde\Gamma_{i-1}-\e^+_{i-1\pm1})}
\end{align}
\[\times\frac{(z-\gamma^\circ_{i-1})^{\e^\circ_{i-1}}}{(q^{u_i}z-\gamma^\circ_i)^{\e^\circ_i}}\frac{(q^{u_i}z-\gamma^+_{i})^{\e^+_{i}}}{(qz-\gamma^+_{i-1})^{\e^+_{i-1}}}\frac{(q^{u_i+1}z-\gamma^-_{i})^{\e^-_{i}}}{(qz-\gamma^-_{i-1})^{\e^-_{i-1}}}\,.\]
\end{itemize}
One can obtain those for $B/D$ types in the same way.

\subsection{Reflection-Invariance of Gluing Morphisms}\label{sec:RIgluing}

Consider
\begin{align}
\widetilde B_{i-1}(z)=c_{i-1}\widetilde  B_{i-1}\left(\frac{q^{b_{i-1}}}{z}\right)\,,
\end{align}
where $\widetilde B_{i-1}(z)=f_{i-1}(z)B_{i-1}(z)$ with some rational function $f_{i-1}(z)$, and  $c_{i-1}$ is an undetermined constant. 

The $BC/C$ type generalized twists with $\e^+_{i+1}\neq \e^+_{i-1}$ and $\e^+_i\neq\e^+_{i-2}$ induce the following reflection-invariant property on the gluing isomorphism.
\begin{equation}
 B_{i-1}(z)= q^{\Delta^-_i-2\Delta^+_{i-1}+\Delta^+_{i-2}+\e^-_i-2\e^+_{i-1}+\e^+_{i-2}}z^{-\widetilde  \Gamma_i+\widetilde  \Gamma_{i-1}}   
\end{equation}
 \[\times \frac{   (z-\gamma^\circ_{i-1})^{\e^\circ_{i-1}} (z-\gamma^+_{i})^{\e^+_{i}} (z-q^{-1}\gamma^+_{i-2})^{\e^+_{i-2}}}{ (z-q^{-u_i}\gamma^\circ_i)^{\e^\circ_i} (z-q^{-u_i}\gamma^+_{i+1})^{\e^+_{i+1}} ( z-q^{-u_i-1}\gamma^+_{i-1})^{\e^+_{i-1}}}  \frac{(z-q^{-u_i}\gamma^+_i)^{\e^+_i}(z-q^{-u_i-1}\gamma^-_i)^{\e^-_i}}{ (z-q^{-1}\gamma^+_{i-1})^{\e^+_{i-1}}(z-\gamma^-_{i-1})^{\e^-_{i-1}}}\,,
\]
and
\begin{align}
 B_{i-1}\left(\frac{q^{b_{i-1}}}{z}\right)&=  q^{\Delta^-_i-2\Delta^+_{i-1}+\Delta^+_{i-2}-u_i\widetilde  \Gamma_i -b_{i-1}(\widetilde  \Gamma_i-\widetilde  \Gamma_{i-1}) -\e^+_i+\e^+_{i-1}}\frac{(\gamma^\circ_{i-1})^{\e^\circ_{i-1}}(\gamma^+_{i-2})^{\e^+_{i-2}} (\gamma^-_i)^{\e^-_i}  }{(\gamma^\circ_i)^{\e^\circ_i}(\gamma^+_{i+1})^{\e^+_{i+1}}(\gamma^-_{i-1})^{\e^-_{i-1}}} z^{-\widetilde  \Gamma_i +\widetilde  \Gamma_{i-1}}\nonumber
\\
&\times \frac{   (z-q^{b_{i-1}}(\gamma^\circ_{i-1})^{-1})^{\e^\circ_{i-1}} (z-q^{b_{i-1}}(\gamma^+_{i})^{-1} )^{\e^+_{i}} (z-q^{b_{i-1}+1}(\gamma^+_{i-2})^{-1})^{\e^+_{i-2}}}{ (z-q^{u_i+b_{i-1}}(\gamma^\circ_i)^{-1} )^{\e^\circ_i} ( z-q^{u_i+b_{i-1}}(\gamma^+_{i+1})^{-1})^{\e^+_{i+1}} (z -q^{u_i+1+b_{i-1}}(\gamma^+_{i-1})^{-1} )^{\e^+_{i-1}}}  
\\
&\times\frac{(z  - q^{u_i+b_{i-1}}(\gamma^+_{i})^{-1})^{\e^+_i}(z - q^{u_i+1+b_{i-1}}(\gamma^-_i)^{-1} )^{\e^-_i}}{ (z - q^{b_{i-1}+1}(\gamma^+_{i-1})^{-1})^{\e^+_{i-1}}(z- q^{b_{i-1}}(\gamma^-_{i-1})^{-1})^{\e^-_{i-1}}}\,.    \nonumber
\end{align}

By comparing the singularity structure of $B_{i-1}$, we can find the generic structure of $\widetilde B_{i-1}$.
\begin{align}
  \frac{  (z-\gamma^+_{i})^{\e^+_{i}}(z-q^{-u_i}\gamma^+_i)^{\e^+_i}(z-q^{-u_i-1}\gamma^-_i)^{\e^-_i}}{ ( z-q^{-u_i-1}\gamma^+_{i-1})^{\e^+_{i-1}}(z-q^{-1}\gamma^+_{i-1})^{\e^+_{i-1}}(z-\gamma^-_{i-1})^{\e^-_{i-1}}}  \frac{  (z-\gamma^\circ_{i-1})^{\e^\circ_{i-1}} (z-q^{-1}\gamma^+_{i-2})^{\e^+_{i-2}}}{ (z-q^{-u_i}\gamma^\circ_i)^{\e^\circ_i} (z-q^{-u_i}\gamma^+_{i+1})^{\e^+_{i+1}} }\nonumber
  \\
= \frac{  (z-q^{b_{i-1}}(\gamma^+_{i})^{-1} )^{\e^+_{i}}(z  - q^{u_i+b_{i-1}}(\gamma^+_{i})^{-1})^{\e^+_i}(z - q^{u_i+1+b_{i-1}}(\gamma^-_i)^{-1} )^{\e^-_i} }{ (z -q^{u_i+1+b_{i-1}}(\gamma^+_{i-1})^{-1} )^{\e^+_{i-1}}(z - q^{b_{i-1}+1}(\gamma^+_{i-1})^{-1})^{\e^+_{i-1}}(z- q^{b_{i-1}}(\gamma^-_{i-1})^{-1})^{\e^-_{i-1}}}  
\\
\times\frac{(z-q^{b_{i-1}}(\gamma^\circ_{i-1})^{-1})^{\e^\circ_{i-1}}  (z-q^{b_{i-1}+1}(\gamma^+_{i-2})^{-1})^{\e^+_{i-2}} }{   (z-q^{u_i+b_{i-1}}(\gamma^\circ_i)^{-1} )^{\e^\circ_i} ( z-q^{u_i+b_{i-1}}(\gamma^+_{i+1})^{-1})^{\e^+_{i+1}}}\,.    \nonumber
\end{align}
There are relations compatible with the reflection-invariant conditions on the generalized twists:
\begin{align}
    (\gamma^+_i)^2=q^{b_{i-1}}=q^{n+1-i},&&\gamma^+_i\gamma^-_i=q^{b_{i-1}+2u_{i}+1}=q^{n+1-i},
    \\
    \gamma^+_{i-1}\gamma^-_{i-1}=q^{b_{i-1}+1}=q^{n+2-i},&& (\gamma^+_{i-1})^2=q^{b_{i-1}+2u_{i}+2}=q^{n+2-i},
     \\
      (\gamma^+_{i-1})^2=q^{b_{i-1}+2}=q^{n+3-i},&& (\gamma^\circ_{i-1})^2=q^{b_{i-1}}=q^{n+1-i},
      \\
     (\gamma^\circ_{i})^2=q^{b_{i-1}+2u_i}=q^{n-i},&& (\gamma^+_{i+1})^2=q^{b_{i-1}+2u_i}=q^{n-i}.
\end{align}
Hence, when $\e^+_i=\e^-_i$ and $ \e^+_{i-1}=\e^-_{i-1}$,
\begin{align}
    b_{i-1}=n+1-i,&&u_i=-\frac{1}{2}\,,
\end{align}
as well as
\begin{align}
    c_{i-1}=q^{\e^+_i+\e^-_i -3\e^+_{i-1}+\e^+_{i-2} +u_i\widetilde  \Gamma_i +b_{i-1}(\widetilde  \Gamma_i-\widetilde  \Gamma_{i-1}) }\frac{(\gamma^\circ_i)^{\e^\circ_i}(\gamma^+_{i+1})^{\e^+_{i+1}}(\gamma^-_{i-1})^{\e^-_{i-1}}}{(\gamma^\circ_{i-1})^{\e^\circ_{i-1}}(\gamma^+_{i-2})^{\e^+_{i-2}} (\gamma^-_i)^{\e^-_i}  }\,.
\end{align}
In most cases, $c_{i-1}$ is $\pm1$ depending on the signs of $\gamma^\bullet_i$.  
The gluing morphisms themselves are not reflection-invariant for some $(i)$ and $(i-1)$ nodes when 
\begin{align}
    \e^+_i \neq \e^-_i ,&&  \textrm{or}  && \e^+_{i-1} \neq  \e^-_{i-1}. 
\end{align}
In this case, we can complete $B_{i-1}$ to $\widetilde B_{i-1}$ by multiplying simple factors $f_{i-1}(z)$. 

The canonical examples of the reflection-invariant $B_{i-1}(z)$
are tRS models for $B_n$, $BC_n$, $C_n$, and $D_n$. For example, $C_n$ model is trivial example of $B_{i-1}=-1$ since $\e^+_i=\e^-_i=0$ for all $i$. On the other hand, the $BC_n$ model is a non-trivial example of the case, $\e^+_i=\e^-_i=1$ for all $i$, with
\[
 B_{i-1}(z)
 = \frac{1}{q^{u_i}}\frac{(q^{u_i}z-\gamma^+_i)(q^{u_i+1}z-\gamma^-_i)}{(qz-\gamma^+_{i-1})(z-\gamma^-_{i-1})}\,.\]
\begin{Thm}\label{thm:RIgluingBBCCD}
The gluing isomorphisms defined by \eqref{eq:newgluing} for generalized twists for $B_n$, $BC_n$, $C_n$, and $D_n$ tRS models are reflection-invariant
\begin{align}
\widetilde B_{i-1}(z)=c_{i-1}\widetilde  B_{i-1}\left(\frac{q^{b_{i-1}}}{z}\right)\,,
\end{align}
with 
\begin{align}
  f_{i-1}(z)=1 ,&& b_{i-1}=n-i,&&u_i=-\frac{1}{2}\,,
\end{align}
for all $i$. For each model, $c_{i-1}$ is uniquely computed with the above data.
\end{Thm}
\begin{proof}
We have already proved $C_n$ and $BC_n$. For $D_n$ and $B_n$ models, it suffices to show
\begin{align}
    q^{}\frac{(z-\gamma^\circ_{i-1})}{(z-q^{\frac{1}{2}}\gamma^\circ_i)}\frac{(z+\gamma^\circ_{i-1})}{(z+q^{\frac{1}{2}}\gamma^\circ_i)}= \frac{(\gamma^\circ_{i-1})^2}{(\gamma^\circ_{i})^2}\frac{(z-q^{n-i}(\gamma^\circ_{i-1})^{-1})}{(z-q^{-\frac{1}{2}+n-i}(\gamma^\circ_i)^{-1})}\frac{(z+q^{n-i}(\gamma^\circ_{i-1})^{-1})}{(z+q^{-\frac{1}{2}+n-i}(\gamma^\circ_i)^{-1})}\,.
\end{align}
It is solved by \eqref{eq:D_nRIsol} and \eqref{eq:B_nRIsol}
\begin{align}
 (\gamma^\circ_i)^{2}=\pm q^{n-1-i}\,,&&   (\gamma^\circ_{i-1})^{2}=\pm q^{n-i}\,.
\end{align}
\end{proof}

\begin{Con}\label{conj:Reflection-invariant B}
For a good, generic $A_n$ quiver, let $B_{i-1}(z)$ be a rational function $\zeta_i(q^{u_i}z)/\widetilde \zeta_i(z)$ of the generalized twists of the folded $QQ$-system of Definition \ref{def:genericA_nfoldingBCC} and \ref{def:genericA_nfoldingBD}. There exists a rational function $f_{i-1}(z)$ including $1$ such that $\widetilde B_{i-1}(z)=f_{i-1}(z)B_{i-1}(z)$ satisfies
\begin{align}
\widetilde B_{i-1}(z)=c_{i-1}\widetilde B_{i-1}\left(\frac{q^{b_{i-1}}}{z}\right)\, \quad c_{i-1}\in\mathbb{C},
\end{align}
and the triple $(c_{i-1},u_i,b_{i-1})$ is uniquely determined for any possible    $(\e^+_i,\e^-_{i}|\e^\circ_i,\e^+_{i+1},\e^+_{i-1})\to     (\e^+_{i-1},\e^-_{i-1}|\e^\circ_{i-1},\e^+_{i},\e^+_{i-2})$ for $i=2,\cdots,n$.
\end{Con}

It is already proved in the case when $ \e^+_i =\e^-_i$ and $\e^+_{i-1} = \e^-_{i-1}$.
For cases when $ \e^+_i \neq \e^-_i$ or $\e^+_{i-1} \neq  \e^-_{i-1}$, one needs to check the following $12$ cases:
\begin{align}
    (\e^+_i, \e^-_i,\e^+_{i-1}, \e^-_{i-1})=&(0,0,0,1), (0,0,1,0),(1,1,0,1),(1,1,1,0) ,
    \\&(0,1,0,0) ,(0,1,0,1),(0,1,1,0),(0,1,1,1),\nonumber
    \\&(1,0,0,0),(1,0,0,1),(1,0,1,0),(1,0,1,1).\nonumber    
\end{align}

\bibliography{cpn1}

\end{document}